\documentclass[11pt]{amsart}
\usepackage{amsmath}
\usepackage{amscd}
\usepackage{amssymb}
\usepackage{amsfonts}
\usepackage{amsthm}
\usepackage{bbm}
\usepackage{cancel}
\usepackage{color}
\usepackage{eucal}
\usepackage{enumerate,yfonts}
\usepackage{enumitem}
\usepackage{pdfsync}
\usepackage[all,cmtip]{xy}
\usepackage{graphicx}
\usepackage{graphics}
\usepackage{hyperref}
\usepackage{latexsym}
\usepackage{mathrsfs}
\usepackage{placeins}
\usepackage{pstricks}
\usepackage{bookmark}
\usepackage{mathtools}
\usepackage{stmaryrd}
\usepackage{url}

\newtheorem{thm}{Theorem}[section]

\newtheorem{corollary}[thm]{Corollary}
\newtheorem{lemma}[thm]{Lemma}

\newtheorem{prop}[thm]{Proposition}

\newtheorem{thm-dfn}[thm]{Theorem-Definition}
\newtheorem{cor}[thm]{Corollary}

\newtheorem{remark}[thm]{Remark}
\newtheorem{definition}[thm]{Definition}

\numberwithin{equation}{section}

\newcommand\myatop[2]{\genfrac{}{}{0pt}{}{#1}{#2}}

\newcommand{\nc}{\newcommand}

\newcommand{\cD}{{\mathcal D}}
\newcommand{\cO}{{\mathcal O}}

\newcommand{\cN}{{\mathcal N}}
\newcommand{\cH}{{\mathcal H}}

\newcommand{\cP}{{\mathcal P}}
\newcommand{\cL}{{\mathcal L}}
\newcommand{\cE}{{\mathcal E}}

\newcommand{\cZ}{{\mathcal Z}}
\newcommand{\cS}{{\mathcal S}}
\newcommand{\cM}{{\mathcal M}}
\newcommand{\cG}{{\mathcal G}}
\newcommand{\tcZ}{{\widetilde \cZ}}
\newcommand{\tF}{{\widetilde F}}
\newcommand{\tM}{{\widetilde M}}

\newcommand{\bC}{{\mathbb C}}
\newcommand{\bZ}{{\mathbb Z}}

\newcommand{\bN}{{\mathbb N}}
\newcommand{\bR}{{\mathbb R}}
\newcommand{\bH}{{\mathbb H}}
\newcommand{\bG}{{\mathbb G}}

\newcommand{\rc}{\mathrm{c}}
\newcommand{\rr}{\mathrm{r}}
\newcommand{\ri}{\mathrm{i}}

\newcommand{\rB}{\mathrm{B}}
\newcommand{\rD}{\mathrm{D}}

\newcommand{\La}{{\mathfrak{a}}}

\newcommand{\Lc}{{\mathfrak{c}}}
\newcommand{\Lt}{{\mathfrak{t}}}
\newcommand{\Lg}{{\mathfrak{g}}}
\newcommand{\Lo}{{\mathfrak{o}}}

\newcommand{\Ll}{{\mathfrak{l}}}
\newcommand{\Lp}{{\mathfrak{p}}}
\newcommand{\Lm}{{\mathfrak{m}}}
\newcommand{\Lk}{{\mathfrak{K}}}

\newcommand{\Ls}{{\mathfrak{s}}}
\newcommand{\Lu}{{\mathfrak{u}}}

\newcommand{\fp}{{\mathfrak{p}}}

\newcommand{\fF}{{\mathfrak{F}}}

\newcommand{\fa}{{\mathfrak{a}}}

\newcommand{\p}{\perp}

\newcommand{\la}{\langle}
\newcommand{\ra}{\rangle}

\nc{\ot}{\otimes}
\nc{\on}{\operatorname}

\nc{\oh}{{\operatorname{H}}}
\nc{\gr}{{\operatorname{gr}}}
\nc{\rk}{{\operatorname{rank}}}
\nc{\codim}{{\operatorname{codim}}}
\nc{\img}{{\operatorname{Im}}}
\nc{\IC}{{\operatorname{IC}}}

\nc{\bI}{{\mathbf 1}}

\nc{\lp}{{\left(}}
\nc{\rp}{{\right)}}

\newcommand{\beqn}{\begin{equation*}}
\newcommand{\eeqn}{\end{equation*}}

\newcommand{\beq}{\begin{equation}}
\newcommand{\eeq}{\end{equation}}

\newcommand{\bern}{\begin{eqnarray*}}
\newcommand{\eern}{\end{eqnarray*}}

\newcommand{\inv}{{\mathbin{/\mkern-4mu/}}}

\newcommand{\dimLp}{\dim \Lp}
\newcommand{\baseRing}{\Bbbk}
\newcommand{\BruhatLength}{L_S}
\newcommand{\LiftToBraids}{\beta}
\newcommand{\Pone}{P}

\newcommand{\basis}{v}
\newcommand{\braid}{b}
\newcommand{\nS}{r}
\newcommand{\nSchi}{n}
\newcommand{\basepta}{a_0}
\newcommand{\barbasepta}{\bar{a}_0}
\newcommand{\hatbasepta}{\hat{a}_0}
\newcommand{\tildebasepta}{\tilde{a}_0}
\newcommand{\baseptas}{a_{0, s}}
\newcommand{\barbaseptas}{\bar{a}_{0, s}}

\newcommand{\genpta}{a}
\newcommand{\bargenpta}{\bar{a}}
\newcommand{\hatgenpta}{\hat{a}}

\newcommand{\groupL}{L}
\newcommand{\groupM}{M}

\begin{document}

\title[Nearby Cycle Sheaves for Symmetric Pairs]{Nearby Cycle Sheaves for Symmetric Pairs}

\author{Mikhail Grinberg}
\email{misha@mishagrinberg.net}

\author{Kari Vilonen}\address{School of Mathematics and Statistics, University of Melbourne, VIC 3010, Australia, also Department of Mathematics and Statistics, University of Helsinki, Helsinki, Finland}
\email{kari.vilonen@unimelb.edu.au, kari.vilonen@helsinki.fi}
\thanks{The second author was supported in part by the ARC grants DP150103525 and DP180101445  and the Academy of Finland}

\author{Ting Xue}
\address{School of Mathematics and Statistics, University of Melbourne, VIC 3010, Australia, also Department of Mathematics and Statistics, University of Helsinki, Helsinki, Finland} 
\email{ting.xue@unimelb.edu.au}
\thanks{The third author was supported in part by the ARC grants DP150103525 and DE160100975.}


\begin{abstract}
We present a nearby cycle sheaf construction in the context of symmetric spaces. This construction can be regarded as a replacement for the Grothendieck-Springer resolution in classical Springer theory.
\end{abstract}

\date{November 6, 2022}

\maketitle

\tableofcontents

\section{Introduction}

In this paper, we present a nearby cycle sheaf construction which is a crucial ingredient  in the development of the theory of character sheaves in the context of symmetric spaces. We generalize the results of the first named author in \cite{G2} (see also \cite{G4}) to the case where we allow equivariant local systems as coefficients. It is shown in~\cite{VX} that this construction, and its variant, produce all character sheaves up to parabolic induction in the setting of classical symmetric spaces. The special case of the split symmetric pair $(SL(n), SO(n))$ was treated in the paper~\cite{CVX}, which also relies on the construction presented here. 

We consider a connected complex reductive group $G$ and an involution $\theta: G \to G$, giving rise to a symmetric pair $(G, K)$, with $K = G^\theta$. We have a decomposition $\Lg = \Lk \oplus \Lp$ into $+1$ and $-1$ eigenspaces of $\theta$. Let us write $\cN$ for the nilpotent cone in $\Lg$ and let $\cN_\Lp = \cN \cap \Lp$. Let $\La \subset \Lp$ be a Cartan subspace, i.e., a maximal abelian subspace of $\Lp$ consisting of semisimple elements, and let $W_\La = N_K (\La) / Z_K (\La)$ be the little Weyl group. We write $\Lp^{rs}$ for the set of regular semisimple elements of $\Lp$ and let $\La^{rs} = \La \cap \Lp^{rs}$. We refer the reader to Section~\ref{subsec-sp} for a definition of a regular element of $\Lp$, and remark that a regular semisimple element of $\Lp$ is not necessarily regular in $\Lg$. We have the adjoint quotient map:
\beqn
f: \Lp \to \Lp \inv K \cong \La / W_\La \, .
\eeqn
Note that $f^{-1} (0) = \cN_\Lp$. The fiber of this map at a point $\barbasepta \in \La^{rs} / W_\fa$ is a regular semisimple $K$-orbit $X_{\barbasepta}$, whose equivariant fundamental group is given by:
\beqn
\pi_1^K (X_{\barbasepta}) = I \coloneqq Z_K (\La) / Z_K (\La)^0.
\eeqn
For each character $\chi \in \hat I$, we construct a nearby cycle sheaf $P_\chi \in \on{Perv}_K (\cN_\Lp)$, a $K$-equivariant perverse sheaf on the nilpotent cone $\cN_\fp$. We study the topological Fourier transform of $P_\chi$, and show that this Fourier transform is an IC-extension of a $K$-equivariant local system on $\Lp^{rs}$ (under a suitable identification of $\Lp$ and $\Lp^*$). The local systems on $\Lp^{rs}$ arising in this way are described explicitly in our main Theorem~\ref{thm-main}, where Hecke algebras with parameters $\pm 1$, attached to certain Coxeter subgroups of $W_\La$, enter the description.  The construction of these Hecke algebras can be viewed as an instance of endoscopy, as explained in \cite[Section 3.3]{VX}. This nearby cycle sheaf construction can be regarded as a replacement for the Grothendieck-Springer resolution in classical Springer theory.

The paper is organized as follows. In Section~\ref{sec-preliminaries}, we fix some notation and discuss preliminaries on the little Weyl group $W_\La$ and the equivariant fundamental group $\pi_1^K (\Lp^{rs})$. In particular, we introduce the notion of a regular splitting homomorphism, which enables us to view the fundamental group $\pi_1^K (\Lp^{rs})$ as a semidirect product of the group $I$ and the braid group $B_{W_\La}$ associated to $W_\La$. This notion is needed to state our main result. We present two constructions of regular splitting homomorphisms. The first construction is geometric, arising from a Kostant-Rallis slice $\Ls \subset \Lp$. The second construction is more algebraic, and it leads to a full classification of such homomorphisms. In Section~\ref{sec-statement}, we define the perverse sheaves $P_\chi$, $\chi \in \hat I$, and state our main Theorem~\ref{thm-main}. We also state Theorem~\ref{thm-chi-eq-one}, which describes the Fourier transform of the sheaf $P_1$, corresponding to the trivial character $\chi = 1$, as well as the monodromy action of the braid group $B_{W_\La}$ on $P_1$. In Section~\ref{sec-example}, we describe in detail the example of the symmetric pair $(SL(2), SO(2))$. This example is later used in the proof of our main result. Proofs of several key facts from Sections \ref{sec-preliminaries} and \ref{sec-statement} are deferred to Section~\ref{sec-proofs-prelim}. In Section~\ref{sec-proof-of-thm-chi-eq-one}, we prove Theorem~\ref{thm-chi-eq-one}. Although this result essentially already appears in \cite{G2, G4}, this section sets the stage for the proof of our main theorem in the next section. Finally, in Section~\ref{sec-proof}, we prove our main theorem.

{\bf Acknowledgements.} We thank the Research Institute for Mathematical Sciences, Kyoto University, Japan, for hospitality, support, and a nice research environment. We also thank the anonymous referees for their valuable comments.

\section{Preliminaries}
\label{sec-preliminaries}

We work throughout over the complex numbers in the classical topology.  However, we expect that our statements can be extended to other settings, such as the \'etale one. The classical topology allows us to work with perverse sheaves with coefficients in a general commutative ring $\baseRing$, which we assume to be an integral domain and of finite global dimension. For perverse sheaves we use the conventions of~\cite{BBD}. The reader can just imagine our ring of coefficients to be $\bZ$, as only integers arise in our constructions. In particular, Propositions~\ref{prop-fourier} and \ref{prop-fourier-chi} imply that it suffices to prove our main theorems in the case $\baseRing = \bZ$; see Remarks \ref{rmk-Z-to-k}, \ref{rmk-Z-to-k-chi}.

\subsection{Symmetric pairs}\label{subsec-sp}
Let $G$ be a connected reductive algebraic group over the complex numbers and let $\theta: G \to G$ be an involution of $G$. We write: 
\beqn
G^\theta = K.
\eeqn
Let $\Lg = \on{Lie} G$. Then we have: 
\beqn
\Lg = \Lk \oplus \Lp,
\eeqn
where $\Lk$ (resp. $\Lp$) is the $+1$-eigenspace (resp. $-1$-eigenspace) of $\theta$. We fix a Cartan subspace $\La$ of $\Lp$; i.e., $\La \subset \Lp$ is a maximal abelian subspace consisting of semisimple elements. We write $\Lp^{rs}$ for the set of regular semisimple elements of $\Lp$, and we let: 
\beqn
\La^{rs} = \La \cap \Lp^{rs}.
\eeqn
Here, an element $x \in \Lp$ is called regular if $\dim Z_K (x) \leq \dim Z_K (x')$ for all $x' \in \Lp$. We note that an element of $\Lp^{rs}$ is not necessarily regular semisimple as an element of $\Lg$. We write $A$ for the maximal $\theta$-split torus of $G$ corresponding to $\La$, where $\theta$-split means that $\theta(t) = t^{-1}$ for all $t \in A$.

Let $\groupL = Z_G (\La)$. Note that $\groupL$ is $\theta$-stable and $\groupL^\theta = Z_K (\La)$. We fix a $\theta$-stable maximal torus $T \subset \groupL$. Then we have $A \subset T$ (as $A$ is contained in the center of $\groupL$), and $T$ is also a maximal torus of $G$. Let us write: 
\beq\label{definition-of-C}
C = T^\theta.
\eeq
On the level of Lie algebras, we have:
\beq\label{lie-alg-direct-sum}
\on{Lie}T = \Lt = \La \oplus \Lc,
\eeq
where $\Lc = \on{Lie} (C) \subset \Lk$.   

We write $X^* (T) = \on{Hom} (T, \bG_m)$, $X_* (T) =$ $ \on{Hom} (\bG_m, T)$, for the characters and the cocharacters of $T$, respectively, and $\langle \; , \, \rangle$ for the pairing $X^* (T) \times X_* (T) \to \bZ$. We let $\Phi \subset X^* (T)$ denote the set of roots of $\Lg$ with respect to the maximal torus $T \subset G$, and we write $\check\Phi \subset X_* (T)$ for the corresponding set of coroots. For each $\alpha \in \Phi$, we write $\check\alpha \in \check\Phi$ for the corresponding coroot. For $\alpha \in \Phi$, we write $\Lg_\alpha \subset \Lg$ for the associated root space. We then have:
\beqn
\Lg \ = \ \Lt \oplus  \bigoplus_{\alpha \in \Phi} \Lg_\alpha.
\eeqn
Similarly, we write $X^* (A) = \on{Hom} (A, \bG_m)$, $X_* (A) = \on{Hom} (\bG_m,A)$, and we continue to write $\langle \; , \, \rangle$ for the pairing $X^* (A) \times X_* (A) \to \bZ$. For a free abelian group $H$ of finite rank, we write $H_\bR = H \otimes_\bZ \bR$ and $H_\bC = H \otimes_\bZ \bC$. We have natural identifications $X_* (T)_\bC \cong \Lt$, $X_* (A)_\bC \cong \La$, $X^* (T)_\bC \cong \Lt^*$, and $X^* (A)_\bC \cong \La^*$. We will use these identifications to regard $X_* (T)$ as a subset of $\Lt$, etc.

\subsection{The little Weyl group \texorpdfstring{$W_\La$}{Lg}}
\label{subsec-structure-of-Wa}
Let us write $W_\La = W(K,\La)$ for the ``little" Weyl group:
\beqn
W_\La = W(K,\La) \coloneqq N_K(\La) / Z_K(\La).
\eeqn

Given an $\alpha \in \Phi$ we have $\theta \alpha = \beta \in \Phi$ and $\theta \Lg_\alpha = \Lg_\beta$, since $T$ is $\theta$-stable. We recall the following notions:
\begin{eqnarray*}
\alpha \in \Phi \ \ \text{is real} && \text{if } \alpha |_{\Lc} = 0 \iff \theta \alpha = -\alpha \\
\alpha \in \Phi \ \ \text{is imaginary} && \text{if } \alpha |_{\La} = 0 \iff \theta \alpha = \alpha \\
\alpha \in \Phi \ \ \text{is complex} && \text{otherwise}.
\end{eqnarray*}
We write:
\beqn
\Phi_\rr = \{ \alpha \in \Phi \; | \; \alpha \text{ real} \}, \ \
\Phi_\ri = \{ \alpha \in \Phi \; | \; \alpha \text{ imaginary} \}.
\eeqn
Both $\Phi_\rr$ and $\Phi_\ri$ are sub-root systems of $\Phi$.

We write $\Sigma \subset X^* (A)$ for the restricted root system, i.e., for the root system formed by the non-zero restrictions of the roots in $\Phi$ to $\La$. Thus, $\Sigma$ consists of the restrictions of the real and complex roots to $\La$. We have:
\beq\label{eqn-restricted-root-decomposition}
\Lg = \ Z_\Lg (\La) \oplus \bigoplus_{\bar\alpha \in \Sigma} \Lg_{\bar\alpha} \, ,
\eeq
where we denote the restricted roots by $\bar\alpha$ and the corresponding root spaces by $\Lg_{\bar\alpha}$. The root system $\Sigma$ is not necessarily reduced, and the root spaces $\Lg_{\bar\alpha}$ are not necessarily one dimensional. We note that for $\alpha \in \Phi \backslash \Phi_\ri$, the roots $\alpha$ and $-\theta\alpha$ give rise to the same restricted root in $\Sigma$, since $(\alpha + \theta\alpha) |_{\La} = 0$. We denote the resulting restricted root by $\bar\alpha$. If we interpret $\La^*$ as a subspace of $\Lt^*$, using the direct sum decomposition \eqref{lie-alg-direct-sum}, then we can write:
\beq\label{restricted-roots}
\bar\alpha \ = \ \frac{1}{2} \, (\alpha - \theta\alpha).
\eeq
In particular, for a complex root $\alpha \in \Phi$, the (restricted) root space attached to $\bar\alpha$ is at least two dimensional.

It is well-known that the little Weyl group $W_\La$ is the Coxeter group associated to the restricted root system $\Sigma$ (see, for example, \cite{K,R}), thus generated by the reflections $s_{\bar\alpha}$ at the restricted roots $\bar\alpha \in \Sigma$.

In what follows we describe the structure of $W_\fa$ in more detail, following \cite{V}. Let $W = W(G, T) = N_G(T) / Z_G(T)$ denote the ``big" Weyl group. Then $W$ is the Coxeter group associated to the root system $\Phi \subset X^* (T)_\bR \subset X^* (T)_\bC \cong \Lt^*$.

Recall that the pair $(G, \theta)$ uniquely determines a $\theta$-invariant real form $G_\bR$ of $G$, such that $K_\bR = K \cap G_\bR$ is a maximal compact subgroup of $K$. We let $\Lg_\bR = \on{Lie} (G_\bR)$, $\Lk_\bR = \on{Lie} (K_\bR) = \Lk \cap \Lg_\bR$, and $\Lp_\bR = \Lp \cap \Lg_\bR$. Then $\Lg_\bR = \Lk_\bR \oplus \Lp_\bR$ is a Cartan decomposition, meaning that $\Lu = \Lk_\bR \oplus \mathbf{i} \, \Lp_\bR$ is a compact real form of $\Lg$ (where $\mathbf{i} = \sqrt{-1}$). We will also write $\Lt_\bR = \Lt \cap \Lg_\bR$, $\Lc_\bR = \Lc \cap \Lg_\bR$, and $\La_\bR = \La \cap \Lg_\bR$. Then we have natural identifications: $X_* (A)_\bR \cong \La_\bR$ and $X_* (T)_\bR \cong \mathbf{i} \, \Lc_\bR \oplus \La_\bR$.

Pick a $G$-invariant symmetric bilinear form $\nu_\Lg$ on $\Lg$, which restricts to a pseudo-Euclidean structure on $\Lg_\bR$, so that $\Lg_\bR = \Lk_\bR \oplus \Lp_\bR$ is an orthogonal decomposition, and we have $\nu_\Lg |_{\Lk_\bR} < 0$ and $\nu_\Lg |_{\Lp_\bR} > 0$. Note that, by construction, the form $\nu_\Lg$ is non-degenerate and $\theta$-invariant. Note also that if $G$ is semisimple, we can take $\nu_\Lg$ to be the Killing form. By construction, the form $\nu_\Lg$ restricts to a Euclidean structure on $X_* (T)_\bR$. This exhibits the Weyl groups $W$ and $W_\La$ as Coxeter groups acting on the Euclidean spaces $X_* (T)_\bR$ and $X_* (A)_\bR$, respectively. From now on, we will use the above identification to write $\La_\bR$ for $X_* (A)_\bR$.

\begin{remark}\label{choice-of-nu}
Since many of our constructions will depend on the choice of the bilinear form $\nu_\Lg$, we note that this choice comes from a convex, and therefore contractible, set of possibilities.
\end{remark}

In a number of places in this paper, we will need to fix a metric on $\Lp$. To that end, note that the bilinear form $\nu_\Lg$ gives rise to a Hermitian structure $\langle \; , \, \rangle$ on $\Lp$, defined by $\langle x , y \rangle = \nu_\Lg (x, y)$ for all $x, y \in \Lp_\bR$. We will measure distances and angles in $\Lp$ with respect to this Hermitian structure.

Recall the root systems $\Phi_\rr$ and $\Phi_\ri$. Pick positive systems $\Phi_\rr^+ \subset \Phi_\rr$ and $\Phi_\ri^+ \subset \Phi_\ri$ (for example, we can choose a positive system $\Phi^+ \subset \Phi$ and let $\Phi_\rr^+ = \Phi_\rr \cap \Phi^+$, $\Phi_\ri^+ = \Phi_\ri \cap \Phi^+$). Let $\rho^\rr$ and $\rho^\ri$ be the half sums of the positive real roots and the positive imaginary roots, defined with respect to $\Phi_\rr^+$ and $\Phi_\ri^-$, respectively. With these choices we define:
\beqn
\Phi_\rc = \{ \alpha \in \Phi \; | \; \alpha \text{ complex  and  perpendicular to $\rho^\rr$ and $\rho^\ri$} \}.
\eeqn
Note that the condition that $\alpha$ is complex in the above definition is redundant. Note further that $\Phi_\rc$ is also a sub-root system of $\Phi$, and that $\theta$ restricts to an involution on $\Phi_\rc$. Let us write: 
\beqn
\text{$W_\rr$ (resp. $W_\ri$, $W_\rc$) $=$ the Coxeter group associated to $\Phi_\rr$ (resp. $\Phi_\ri$, $\Phi_\rc$).}
\eeqn
Let $W^\theta \subset W$ (resp. $W_\rc^\theta \subset W_\rc$) be the subgroup consisting of the elements that commute with $\theta$. Then, according to \cite[Proposition 3.12]{V}, we have:
\beqn
W^\theta \ = \ W_\rc^\theta \ltimes (W_\rr \times W_\ri). 
\eeqn
Moreover, we have an isomorphism:
\beq\label{little-Weyl-group}
W_\La \ \cong \ W_\rc^\theta \ltimes W_\rr \, ,
\eeq
given by restricting elements of the RHS to $\La$ (see, for example, \cite[Proposition 4.16]{V}). 

The group $W_\rr$ (as a subgroup of $W_\La$) is generated by the reflections $s_{\bar\alpha}$ for the real roots $\alpha \in \Phi_\rr$. The structure of $W_\rc^\theta$ is described as follows. By~\cite[Lemma 3.1, p$. \; 958$]{V}, we can write $\Phi_\rc$ as a disjoint union $\Phi_\rc^1 \sqcup \Phi_\rc^2$, such that:
\beq\label{roots-in-phic}
\begin{gathered}
\text{the roots in $\Phi_\rc^1$ are orthogonal to the roots in $\Phi_\rc^2$ and}
\\
\text{$\theta: \Phi_\rc^1 \to \Phi_\rc^2$ is an isomorphism of root systems.}
\end{gathered}
\eeq
In particular, we have:
\beq\label{perp-roots}
\alpha \text{ and } \theta \alpha \text{ are perpendicular}, \; \alpha \in \Phi_\rc \, .
\eeq
By~\eqref{roots-in-phic}, we have:
\beq\label{Wctheta}
W_\rc^\theta \text{ is generated by the products } s_\alpha s_{\theta\alpha} \, , \; \alpha \in \Phi_\rc \, .
\eeq
Using~\eqref{restricted-roots}, one can readily check that:
\beq\label{restriction-of-complex-roots}
\text{for } \alpha \in \Phi_\rc, \;\; s_{\bar\alpha} = s_\alpha s_{\theta \alpha} |_{\La} \in W_\La \, .
\eeq

We now pick a Weyl chamber $\La_\bR^+ \subset \La_\bR$ for the little Weyl group $W_\La$, as follows. The choice of the positive system $\Phi_\rr^+ \subset \Phi_\rr$ gives us a choice of a Weyl chamber $\La_{\bR, \rr}^+ \subset \La_\bR$, i.e., a connected component of $\La_\bR - \cup_{\alpha \in \Phi_\rr} H_{\bar\alpha}$, where the $H_{\bar\alpha} \subset \La$ is the root hyperplane corresponding to $\bar \alpha \in \Sigma \subset \La_\bR^* \cong X^* (A)_\bR$. We take $\La_\bR^+$ to be one of the connected components of the intersection $\La_{\bR, \rr}^+ \cap \La^{rs}$ (cf.~\eqref{little-Weyl-group}).

Let $S \subset W_\La$ be the set of simple reflections associated to the Weyl chamber $\La_\bR^+$, i.e., the reflections in the walls of $\La_\bR^+$. Using the semidirect product decomposition \eqref{little-Weyl-group}, we observe that if $w \in W_\La$ fixes the Weyl chamber $\La_{\bR, \rr}^+$, then we have $w \in W_\rc^\theta$. Thus, by our choice of the Weyl chamber $\La_\bR^+$, each simple reflection $s \in S$ belongs to either $W_\rr$ or $W_\rc^\theta$. Let $s \in W_\La$ be a reflection, i.e., $s = s_{\bar\alpha}$ for some $\bar\alpha \in \Sigma$ (where $s_{\bar\alpha}$ denotes the reflection corresponding to the restricted root $\bar\alpha \in \Sigma$). Then, by~\eqref{little-Weyl-group}, we have:
\beq
\label{reflections}
\text{either $\exists \; \alpha \in \Phi_\rr$ such that $s = s_{\bar\alpha} \;\;$ or 
$\;\; \exists \; w \in W_\La$ such that $w s w^{-1} \in W_\rc^\theta \, .$}
\eeq

For each reflection $s \in W_\La$, we define:
\beq\label{eqn-Phi-s}
\Phi_s = \{ \alpha \in \Phi \; | \; \bar\alpha \neq 0, \; s_{\bar\alpha} = s \},
\eeq
and
\beq\label{eqn-deltas}
\delta (s) = \frac{1}{ 2} \, | \Phi_s | = \frac{1}{2} \, \sum_{\substack{\bar\alpha \in \Sigma \\ s_{\bar\alpha} = s}} \dim \Lg_{\bar\alpha} \, .
\eeq
If $\delta(s) = 1$, we write $\alpha_s \in \Phi_\rr^+$ for the unique positive real root such that $s = s_{\bar\alpha_s}$.

\subsection{The equivariant fundamental group \texorpdfstring{$\pi_1^K (\Lp^{rs}, \basepta)$}{Lg}}
\label{subsec-fund-group}
Pick a basepoint $\basepta \in \La_\bR^+ \subset \Lp^{rs}$. In this subsection, we describe the equivariant fundamental group $\pi_1^K (\Lp^{rs}, \basepta)$.

\begin{remark}\label{rmk-basepoint-indep}
Note that the Weyl chamber $\La_\bR^+ \subset \Lp^{rs}$ is contractible. Therefore, the fundamental groups $\pi_1^K (\Lp^{rs}, \genpta)$, for different choices of $\genpta \in \La_\bR^+$, are all canonically isomorphic to each other.
\end{remark}

Given an algebraic group $\cG$, we will write $\cG^0$ for the identity component of $\cG$. If $\cG$ acts on a variety $X$, we can think of the equivariant fundamental group $\pi_1^{\cG} (X)$ in a classical way as $\pi_1 ((E \cG \times X) / \cG)$. Furthermore, we choose our conventions so that $\pi_1^\cG (pt) = \cG / \cG^0$. In particular, if $X$ is connected, we identify $\cG$-equivariant local systems on $X$ with (left) representations of $\pi_1^{\cG} (X)$.

\pagebreak

Let:
\beqn
I = Z_K (\La) / Z_K (\La)^0,
\eeqn
be the component group of $Z_K (\La)$. We have:
\beqn
I = C / C^0, 
\eeqn
where $C$ is defined in~\eqref{definition-of-C}. Let us write $A [2]$ for the group of order 2 elements of the torus $A$. Then we have $A [2] \subset C$ and $C = C^0 \cdot A [2]$ (see, for example, \cite[Proposition 1]{KR}). Thus, there is a surjection:
\beq\label{surjectionOnI}
A [2] \to C / C^0 = I.
\eeq
It follows that $I$ is an elementary abelian $2$-group, i.e., a finite product of copies of $\bZ / 2$, and the little Weyl group $W_\La$ acts on $I$ via conjugation.

Consider the adjoint quotient map:
\beqn
f: \Lp \to \Lp \inv K \cong \La \slash W_\La \, .
\eeqn
It restricts to a fibration $f: \Lp^{rs} \to \La^{rs} \slash W_\La$. We set $\barbasepta = f (\basepta)$. Recall that:
\beqn
B_{W_\La} \coloneqq \pi_1 (\La^{rs} \slash W_\La, \barbasepta),
\eeqn
is the braid group attached to the Coxeter group $W_\La$. Let us write: 
\beqn
X_{\barbasepta} = f^{-1} (\barbasepta) = \text{the semisimple $K$-orbit through $\basepta \,$.}
\eeqn
We have $X_{\barbasepta} \cong K / Z_K(\basepta)$, $Z_K(\basepta) = Z_K(\La)$, and therefore:
\beqn
\pi^K_1 (X_{\barbasepta}, \basepta) \cong I.
\eeqn

Note that the group $K$ is not necessarily connected. To deal with this disconnectedness, we write:
\beqn
\begin{split}
I^0 & = \pi^{K^0}_1 (X_{\barbasepta}, \basepta) = Z_{K^0} (\La) / Z_{K^0} (\La)^0 \\
& = \on{Im} (\pi_1 (X_{\barbasepta}, \basepta) \to \pi^K_1 (X_{\barbasepta}, \basepta)) \subset I.
\end{split}
\eeqn
We then have:
\beqn
I / I^0 \cong K / K^0. 
\eeqn
Here we have used the fact that $X_{\barbasepta}$ is connected, and hence it is also a homogenous space for $K^0$. Furthermore, we have: 
\beq\label{eqn-little-Weyl-redefined}
W_\La = N_K (\La) / Z_K (\La) = N_{K^0} (\La) / Z_{K^0} (\La).
\eeq
Let us write:
\beqn
\widetilde B_{W_\La} = \pi_1^K (\Lp^{rs}, \basepta), \;\; \widetilde B^0_{W_\La} = \pi_1^{K^0} (\Lp^{rs}, \basepta) = \on{Im} (\pi_1 (\Lp^{rs}, \basepta) \to \pi_1^K (\Lp^{rs}, \basepta)).
\eeqn
Then we have $\widetilde B_{W_\La} / \widetilde B^0_{W_\La} \cong K / K^0$. We thus have the following commutative diagram:
\beqn
\begin{CD}
1 @>>> I^0 @>>> \widetilde B^0_{W_\La} @>>> B_{W_\La} @>>> 1
\\
@. @VVV @VVV @| @.
\\
1 @>>> I @>>> \widetilde B_{W_\La} @>>> B_{W_\La} @>>> 1
\\
@. @VVV @VVV @. @.
\\
@.  K / K^0 @= K / K^0 \, . @. @. 
\end{CD}
\eeqn

To gain further insight into the structure of $\widetilde B_{W_\La}$, define:
\beqn
\widetilde W_\La = N_K (\La) / Z_K (\La)^0.
\eeqn
We then have the following commutative diagram:
\beq
\label{mainDiagram}
\begin{CD}
1 @>>> I @>>> \widetilde B_{W_\La} @>{\tilde q}>> B_{W_\La} @>>> 1 \;\,
\\
@. @| @VV{\tilde p}V @VV{p}V @.
\\
1 @>>> I @>>> \widetilde W_\La @>{q}>> W_\La @>>> 1 \, .
\end{CD}
\eeq
To define the maps in the diagram above, note that we have:
\beqn
\widetilde B_{W_\La} = \pi_1^K (\Lp^{rs}, \basepta) \cong \pi_1^{N_K (\La)} (\La^{rs}, \basepta) \cong \pi_1^{\widetilde W_{\La}} (\La^{rs}, \basepta).
\eeqn
Also, we have:
\beqn
B_{W_\La} = \pi_1 (\La^{rs} / W_\La, \barbasepta)  \cong \pi_1^{W_\La} (\La^{rs}, \basepta). 
\eeqn
The maps $p$ and $\tilde p$ are given by mapping $\La^{rs}$ to a point, the map $q$ is the natural quotient map, and the map $\tilde q$ is induced by $q$. It follows from diagram~\eqref{mainDiagram} and the abelian property of $I$ that the conjugation action of $\widetilde B_{W_\La}$ on itself descends to an action of $B_{W_\La}$ on $I$, and that this action factors through the conjugation action of $W_\La$ on $I$. We will denote these actions of $B_{W_\La}$ and $W_\La$ on $I$ by ``$\; \cdot \;$''; so we have $b \cdot u = p (b) \cdot u \in I$ for all $u \in I$ and $b \in B_{W_\La}$. Note that the action of $W_\La$ preserves the subgroup $I^0 \subset I$ and induces a trivial action on the quotient $I / I^0 \cong K / K^0$. Note also that the kernel of the map $p$ is the pure braid group $P\!B_{W_\La} = \pi_1 (\La^{rs}, \basepta)$.

Moreover, by inspection, the right square of diagram~\eqref{mainDiagram} is Cartesian, and hence we can identify the group $\widetilde B_{W_\fa}$ as a subgroup of $\widetilde W_\La \times B_{W_\La}$ as follows:
\beq
\label{concrete}
\widetilde B_{W_\La} \cong \{ (\widetilde w, \braid) \in \widetilde W_\La \times B_{W_\La} \; | \; q(\widetilde w) = p(\braid) \}.
\eeq

A further remarkable property of diagram \eqref{mainDiagram} is that the top row of this diagram splits, i.e., the map $\tilde q$ admits a right inverse $\tilde r : B_{W_\La} \to \widetilde B_{W_\La}$. To show this, it suffices to split the exact sequence:
\beq\label{shortexsecofpi1s}
1 \to \pi_1 (X_{\barbasepta}) \to \pi_1 (\Lp^{rs}, \basepta) \to B_{W_\La} \to 1 \, .
\eeq
We can do so using a Kostant-Rallis section. More precisely, for each regular nilpotent element $y \in \cN_\Lp$, there exists a linear subvariety (i.e., a parallel translate of a vector subspace) $\Ls \subset \Lp$ containing $y$, such that the embedding $\Ls \hookrightarrow \Lp$ induces an isomorphism $\Ls \to \Lp \inv K \cong \La / W_{\La}$ (see \cite[Theorems 11,12,13]{KR}). Let $\Ls^{rs} = \Ls \cap \Lp^{rs}$. The $K$-orbit $X_{\barbasepta}$ intersects the section $\Ls$ at exactly one point $a' \in \Ls^{rs}$. The restriction $f |_{\Ls^{rs}} : \Ls^{rs} \to \La^{rs} / W_{\La}$ induces an isomorphism of fundamental groups:
\beq\label{pi1isom1}
B_{W_\La} = \pi_1 (\La^{rs} / W_{\La}, \barbasepta) \cong \pi_1 (\Ls^{rs}, a').
\eeq
Pick a continuous path $\Gamma : [0, 1] \to X_{\barbasepta}$, with $\Gamma (0) = \basepta$ and $\Gamma (1) = a'$; it defines an isomorphism:
\beq\label{pi1isom2}
\pi_1 (\Lp^{rs}, \basepta) \cong \pi_1 (\Lp^{rs}, a').
\eeq
By composing isomorphisms \eqref{pi1isom1} and \eqref{pi1isom2} with the map of fundamental groups induced by the embedding $\Ls^{rs} \hookrightarrow \Lp^{rs}$, we obtain a homomorphism $B_{W_\La} \to \pi_1 (\Lp^{rs}, \basepta)$ which splits the exact sequence \eqref{shortexsecofpi1s}. Note that this splitting homomorphism depends on the triple $[y, \Ls, \Gamma]$. It gives rise to a homomorphism:
\beq\label{krsplitting}
\tilde r [y, \Ls, \Gamma] : B_{W_\La} \to \widetilde B_{W_\La} \, ,
\eeq
which splits the top row of diagram~\eqref{mainDiagram}. We will refer to the homomorphism $\tilde r [y, \Ls, \Gamma]$ as the Kostant-Rallis splitting corresponding to the triple $[y, \Ls, \Gamma]$.

\subsection{The subgroups \texorpdfstring{$\protect\widetilde W_s$}{Lg} and the concept of a regular splitting}
\label{subsec-regular-splitting}
The homomorphism $\tilde r [y, \Ls, \Gamma]$ possesses a certain key property with respect to the simple reflections in $W_\La$, which we call regularity, and which we now proceed to describe. We begin by associating to each reflection $s \in W_\La$ certain subgroups $W_s \subset W_\La$, $\widetilde W_s \subset \widetilde W_\La$, and $I_s \subset I$. 

Let $s \in W_\La$ be a reflection. Let $\La_s \subset \La$ denote the hyperplane fixed by $s$. Let: 
\beqn
K_s = Z_K (\La_s) \;\; \text{and} \;\; \Lk_s = \on{Lie} (K_s) \subset \Lk.
\eeqn 
Define:
\beqn
W_s = N_{K^0_s} (\La) / Z_{K^0_s} (\La) \;\; \text{and} \;\; \widetilde W_s = N_{K^0_s} (\La) / Z_{K^0_s} (\La)^0.
\eeqn
We have a natural projection $q_s : \widetilde W_s \to W_s$. Note that $Z_{K_s^0} (\La)^0 = Z_K(\La)^0$. Therefore, we have inclusions $W_s \to W_\La$ and $\widetilde W_s \to \widetilde W_\La$. We will use these inclusions to view $W_s$ (resp. $\widetilde W_s$) as a subgroup of $W_\La$ (resp. $\widetilde W_\La$). We have the following short exact sequence:
\beqn
1 \longrightarrow I_s \coloneqq Z_{K_s^0} (\La) / Z_{K_s^0} (\La)^0 \longrightarrow \widetilde W_s \xrightarrow{\; q_s \;} W_s \longrightarrow 1 \, .
\eeqn
Note that, by construction, we have $I_s \subset I^0 \subset I$.

The groups $W_s$, $\widetilde W_s$, and $I_s$ defined above can be described as follows. Recall the positive real root $\alpha_s$, associated to $s$ when $\delta (s) = 1$, such that $s = s_{\bar\alpha_s}$. 

\begin{lemma}\label{lemma-trichotomy}
We have:
\begin{enumerate}[topsep=-1.5ex]
\item[(i)]   $W_s = \{ 1, s \} \subset W_\La$.

\item[(ii)]  If $s = s_{\bar\alpha}$ for $\alpha \in \Phi_\rr$, then $\check\alpha (-1)$ represents an element of $I_s \subset I$.

\item[(iii)] If $\delta (s) > 1$, then $Z_{K_s^0} (\La)$ is connected, we have $I_s = \{ 1 \}$, and $q_s: \widetilde W_s \to W_s$ is an isomorphism.

\item[(iv)] If $\delta (s) = 1$, then we have $I_s = \langle \check\alpha_s (-1) \rangle$. Furthermore,
\begin{enumerate}
\item if $\check\alpha_s (-1) \in Z_K (\La)^0$, then $Z_{K_s^0} (\La)$ is connected, we have $I_s = \{ 1 \}$, and $q_s : \widetilde W _s \to W_s$ is an isomorphism;

\item if $\check\alpha_s (-1) \notin Z_K (\La)^0$, then we have $I_s \cong \bZ / 2$, $\widetilde W_s \cong \bZ / 4$, and $q_s: \widetilde W_s \to W_s$ is a surjection.
\end{enumerate}
\end{enumerate}
\end{lemma}

A proof of Lemma~\ref{lemma-trichotomy} will be given in Section~\ref{subsec-rank-one}.

\begin{remark}
The groups $I_s = \ker (q_s) \subset \widetilde W_s$, for all reflections $s \in W_\La$, generate the subgroup $I^0 \subset I$; see, for example, \cite[Theorem 7.55]{K}. In \cite{K} the result is stated for semisimple groups, but by passing to the derived group, we can conclude that it also holds for reductive groups.
\end{remark}

Recall the set $S \subset W_\La$ of simple reflections corresponding to the Weyl chamber $\La_\bR^+$. For each $s \in S$, let $\sigma_s \in B_{W_\La}$ be the counter-clockwise braid generator corresponding to $s$.

\begin{definition}\label{defn-regular-splitting}
Let $\tilde r : B_{W_\La} \to \widetilde B_{W_\La}$ be a homomorphism which splits the top row of diagram~\eqref{mainDiagram}. We say that $\tilde r$ is a regular splitting if the composition $r = \tilde p \circ \tilde r : B_{W_\La} \to \widetilde W_\La$ satisfies:
\beqn
r (\sigma_s) \in \widetilde W_s \subset \widetilde W_\La \, ,
\eeqn
for every $s \in S$.
\end{definition}

\begin{prop}\label{prop-kr-is-regular}
The splitting homomorphism $\tilde r [y, \Ls, \Gamma]$ of equation~\eqref{krsplitting} is regular for every choice of the triple $[y, \Ls, \Gamma]$.
\end{prop}

We suspect that, in fact, every splitting homomorphism $\tilde r : B_{W_\La} \to \widetilde B_{W_\La}$ is regular, and every regular splitting is a Kostant-Rallis splitting for some triple $[y, \Ls, \Gamma]$. A proof of Proposition~\ref{prop-kr-is-regular} will be given in Section \ref{subsec-regularity-of-kr}.

A more algebraic description of the set of all regular splittings can be given by demonstrating that they can be constructed ``one simple reflection at a time.'' More precisely, we have the following lemma.

\pagebreak

\begin{lemma}\label{lemma-braid-relations}
Let $s_1, s_2 \in S$ be a pair of simple reflections. For each $i = 1, 2$, pick a lift:
\beqn
\tilde s_i \in q^{-1}_{s_i} (s_i) \subset \widetilde W_{s_i} \, .
\eeqn
Then the elements $\tilde s_1, \tilde s_2 \in \widetilde W_\La$ satisfy the braid relation for $\sigma_{s_1}, \sigma_{s_2} \in B_{W_\La}$.
\end{lemma}

A proof of Lemma~\ref{lemma-braid-relations} will be given in Section~\ref{subsec-splitting}. The lemma enables us to define a splitting homomorphism $\tilde r : B_{W_\fa} \to \widetilde B_{W_\fa}$ as follows. For each simple reflection $s \in S$, pick an element:
\beq\label{choices-tilde-s}
\tilde s \in q^{-1}_s (s) \subset \widetilde W_s \subset \widetilde W.
\eeq
Note that, by Lemma \ref{lemma-trichotomy}, there is either one or two choices for each $\tilde s$. Define a lifting homomorphism:
\beqn
r : B_{W_\La} \to \widetilde W_\La \;\; \text{by} \;\; r : \sigma_s \mapsto \tilde s, \; s \in S.
\eeqn
Lemma~\ref{lemma-braid-relations} ensures that the map $r$ is well defined. Finally, in terms of isomorphism~\eqref{concrete}, define the splitting homomorphism:
\beqn
\tilde r : B_{W_\La} \to \widetilde B_{W_\La} \;\; \text{by} \;\; \sigma_s \mapsto (r(s), \sigma_s), \; s \in S.
\eeqn
By construction, the map $\tilde r$ is a regular splitting in the sense of Definition~\ref{defn-regular-splitting}. Moreover, it is clear from Definition~\ref{defn-regular-splitting} that every regular splitting arises via the above construction from some choice of the $\{ \tilde s \}_{s \in S}$.

Note that, given a splitting homomorphism $\tilde r : B_{W_\La} \to \widetilde B_{W_\La}$, whether regular or not, we obtain an identification:
\beqn
\widetilde B_{W_\La} \cong I \rtimes B_{W_\La}, \;\; (\widetilde w, \braid) \mapsto (\widetilde w \, r (\braid)^{-1}, \braid),
\eeqn
where $r = \tilde p \circ \tilde r : B_{W_\La} \to \widetilde W_\La$, and $B_{W_\La}$ acts on $I$ through the map $p : B_{W_\La} \to W_\La$, via $b : u \mapsto r(b) \, u \, r(b)^{-1}$, $b\in B_{W_\La}$, $u\in I$. In order to state our main theorem in Section \ref{sec-statement}, we now fix a regular splitting homomorphism:
\beq\label{eqn-tilde-r}
\tilde r : B_{W_\La} \to \widetilde B_{W_\La}.
\eeq
For example, we can use Proposition \ref{prop-kr-is-regular} and take $\tilde r$ to be a Kostant-Rallis splitting.

\section{Statement of the main theorem}
\label{sec-statement}

Recall the $G$-invariant and $\theta$-invariant non-degenerate symmetric bilinear form $\nu_\Lg$ on $\Lg$, which was fixed in Section \ref{subsec-structure-of-Wa}. Using this form, we identify $\Lp^*$ with $\Lp$. We write $\fF$ for the topological Fourier transform functor, which induces an equivalence:
\beqn
\fF : \on{Perv}_{K} (\Lp)_{\bC^*\text{-conic}} \to \on{Perv}_{K} (\Lp)_{\bC^*\text{-conic}} \, ,
\eeqn
on the category of $K$-equivariant $\bC^*$-conic perverse sheaves on $\Lp$. See \cite[Definition 3.7.8]{KS} for a definition of this functor up to a shift.

In \cite{G2, G1, G4}, the first named author explains how to analyze the nearby cycle sheaf $P = P_f$ with constant coefficients for the adjoint quotient map $f : \Lp \to \Lp \inv K$. In particular, it is shown in \cite{G2, G4} that $P$ is a $\bC^*$-conic perverse sheaf on the nilcone $\cN_\Lp = f^{-1} (0)$, which is independent of the curve one chooses to perform the nearby cycle construction. It is also shown that the topological Fourier transform $\fF P$ is the $\IC$-extension of a local system on $\Lp^{rs}$, which is further explicitly described.

We now proceed to generalize the results of \cite{G2, G4} on the sheaf $P$ to the case of nearby cycle sheaves with twisted coefficients. There are two technical distinctions between the settings of \cite{G2, G4} and of the present paper. First, in the present paper, we allow $K$ to be disconnected and work with $K$-equivariant sheaves and local systems. Second, we work with coefficients in $\baseRing$, while \cite{G2, G1, G4} work with coefficients in $\bC$. The latter distinction does not materially affect any of the discussion.

\subsection{Nearby cycle sheaves with twisted coefficients}
\label{subsec-twisted}
In this subsection, we introduce the nearby cycle sheaves $P_\chi$ for $\chi \in \hat I$, and the monodromy action of the braid group $B_{W_\La}$ on $P_\chi$ for $\chi = 1$. Let $\baseRing^*$ be the group of invertible elements of $\baseRing$, and let $\hat I \coloneqq \on{Hom} (I, \baseRing^*)$ be the set of characters of $I$. Since $I$ is an elementary abelian $2$-group, and $\baseRing$ is an integral domain, we have:
\beqn
\hat I = \on{Hom} (I, \{ \pm 1\}).
\eeqn
Note that there is only one character if $\baseRing$ is of characteristic $2$. Fix a character $\chi \in \hat I$. Recall that the equivariant fundamental group of $X_{\barbasepta}$ is given by:
\beqn
\pi_1^K (X_{\barbasepta}, \basepta) \cong I.
\eeqn
Therefore, the character $\chi$ gives rise to a rank one $K$-equivariant local system $\cL_\chi$ on $X_{\barbasepta}$, with $(\cL_\chi)_{\basepta} = \baseRing$. We apply base change to the adjoint quotient map $f: \Lp \to \La / W_{\La}$, to form a family:
\beqn
\cZ_{\barbasepta} = \{ (x, c) \in \Lp \times \bC \; | \; f (x) = c \, \barbasepta \} \to \bC,
\eeqn
where the $\bC$-action on $\La / W_{\La}$ is induced by the action on $\La$, so that $f (c \, \genpta) = c \, f (\genpta)$ for all $c \in \bC$ and $\genpta \in \La$. Let $\cZ_{\barbasepta}^{rs} = \{ (x, c) \in \cZ_{\barbasepta} \; | \; c \neq 0 \}$, and let $p_{\barbasepta} : \cZ_{\barbasepta}^{rs} \to X_{\barbasepta}$ be the map $p_{\barbasepta} : (x, c) \mapsto c^{-1} \, x$. By construction, the family $\cZ_{\barbasepta} \to \bC$ is $K$-equivariant. Let us write $\psi_{\barbasepta}$ for the nearby cycle functor with respect to this family. We set:
\beq\label{eqn-P-chi}
P_\chi = \psi_{\barbasepta} \circ \, p_{\barbasepta}^* \, (\cL_\chi) [-] \in \on{Perv}_K (\cN_\Lp).
\eeq
Here and henceforth $[-]$ denotes an appropriate cohomological shift, so that the resulting sheaf is perverse. Note that every object of $\on{Perv}_K (\cN_\Lp)$ is $\bC^*$-conic, since the nilcone $\cN_\Lp$ consists of finitely many $K$-orbits, and each $K$-orbit $\cO \subset \cN_\Lp$ is $\bC^*$-invariant.

Note that the construction of the sheaf $P_\chi$ depends on the choices of the general fiber $X_{\barbasepta} = f^{-1} (\barbasepta)$ and of the basepoint $\basepta \in X_{\barbasepta}$. We will next discuss how this construction varies as we vary the fiber $X_{\barbasepta}$. We begin with the case where $\chi = 1$ is the trivial character and the local system $\cL_1$ is trivial. We consider the following family:
\beq\label{familyF}
\xymatrix{ \cZ = \myatop{\{ (x, \bargenpta, c) \in \Lp \times (\La^{rs} / W_\La) \times \bC}{\; | \; f (x) = c \, \bargenpta \}}
\ar[rr]^-{F} \ar[rd] && \La^{rs} / W_\La \times \bC \ar[dl] \\& \La^{rs} / W_\La \, ,}
\eeq
where all the maps are projections. Let $F_2 : \cZ \to \bC$ be the map $(x, \bar a, c) \mapsto c$, and let $\psi_{F_2}$ be the corresponding nearby cycle functor. Note that we have:
\beqn
F^{-1} (\bargenpta, c) = c \, X_{\bargenpta} \times \{ \bargenpta \} \times \{ c \} \text{\;\; and \;\;} F_2^{-1} (0) = \cN_\Lp \times \La^{rs} / W_\La \times \{ 0 \}.
\eeqn
We now set:
\beq\label{sheafcP}
\cP = \psi_{F_2} (\baseRing_\cZ) [-] \in \on{Perv}_K (\cN_\Lp \times (\La^{rs} / W_\La)).
\eeq
Let $\cZ_0 = F_2^{-1} (0)$ and let $\cZ^{rs} = \cZ - \cZ_0$. Write $\cN_\Lp / K$ for the set of $K$-orbits in $\cN_\Lp$. Let $\cS_\cZ$ be the stratification of $\cZ$ given by:
\beqn
\cZ = \cZ^{rs} \, \cup \, \bigcup_{\cO \in \cN_\Lp / K} \cO \times \La^{rs} / W_\La \times \{ 0 \},
\eeqn
and let $\cS_{\cZ_0}$ be the corresponding stratification of $\cZ_0$. Note that $\cS_{\cZ_0}$ is a product of an orbit stratification of $\cN_\Lp$ and a trivial stratification of $\La^{rs} / W_\La$. Thus, it is a Whitney stratification. We fully expect that $\cS_\cZ$ is a Whitney stratification as well, i.e., that Whitney conditions hold for the pair $(\cZ^{rs}, \cO \times \La^{rs} / W_\La \times \{ 0 \})$, for every $\cO \in \cN_\Lp / K$. However, this seems to require a separate argument, and is not necessary for our applications.

\begin{prop}\label{prop-asubf2}
For every $\cO \in \cN_\Lp / K$, the pair of strata:
\beqn
(\cZ^{rs}, \cO \times \La^{rs} / W_\La \times \{ 0 \}),
\eeqn
satisfies Thom's $A_{F_2}$ condition.
\end{prop}

We defer the proof of Proposition \ref{prop-asubf2} to Section \ref{subsec-proof-of-asubf2}, where we also recall the definition of the $A_{F_2}$ condition (see Definition~\ref{defn-asubf2}). A general reference for this condition is \cite[Section 11]{Ma}. The significance of Proposition \ref{prop-asubf2} is in the following corollary.

\begin{cor}\label{cor-constructible}
The sheaf $\cP$ is constructible with respect to the stratification $\cS_{\cZ_0}$.
\end{cor}

\begin{proof}
Consider the ambient manifold $M = \Lp \times (\La^{rs} / W_\La) \times \bC \supset \cZ$. For every $c \in \bC^*$, let $\Lambda_c = T^*_{F_2^{-1} (c)} M \subset T^* M$ be the conormal bundle to the general fiber of $F_2$. Let $SS (\cP) \subset T^* M$ denote the micro-support of $\cP$. For this argument, we are only interested in $SS (\cP)$ as a Lagrangian subvariety without multiplicities. By a theorem of Ginzburg, see \cite[Theorem 5.5]{Gi} and the remark that follows that theorem, we have:
\beqn
SS (\cP) = \lim_{c \to 0} \Lambda_c \, .
\eeqn
By Proposition \ref{prop-asubf2}, we have $\displaystyle \lim_{c \to 0} \Lambda_c \subset \Lambda_{\cS_{\cZ_0}}$, where $\Lambda_{\cS_{\cZ_0}} \subset T^* M$ is the conormal variety to the stratification $\cS_{\cZ_0}$. The corollary follows.
\end{proof}

\begin{remark}
In the proof of Corollary \ref{cor-constructible}, we have cited \cite{Gi}, which works in the setting of $\cD$-modules. However, the corollary can also be readily established by a standard transversality argument in the context of constructible sheaves. Roughly speaking, it follows from the fact that, given a point $z \in \cZ_0$ and a pair of small numbers $0 < \epsilon \ll \delta \ll 1$, the boundary $\partial \rB_{z, \delta}$ of the $\delta$-ball around $z$ meets the fiber $F_2^{-1} (c)$ transversely for all $c \in \bC^*$ with $|c| < \epsilon$.
\end{remark}

Let $\on{Perv}_K (\cZ_0, \cS_{\cZ_0})$ denote the category of $K$-equivariant perverse sheaves on $\cZ_0$, constructible with respect to $\cS_{\cZ_0}$. For every $\bargenpta \in \La^{rs} / W_\La$, let $j_{\bargenpta} : \cN_\Lp \to \cZ_0$ be the inclusion $x \mapsto (x, \bargenpta, 0)$, and let $j^*_{\bargenpta} : \on{Perv}_K (\cZ_0, \cS_{\cZ_0}) \to \on{Perv}_K (\cN_\Lp)$ be the corresponding perverse (i.e., properly shifted) restriction functor. Using Corollary~\ref{cor-constructible}, for each $\bargenpta \in \La^{rs} / W_\La$, we can consider the sheaf:
\beqn
\cP_{\bargenpta} = j^*_{\bargenpta} (\cP) \in \on{Perv}_K (\cN_\Lp).
\eeqn
Note that we have:
\beqn
\cP_{\barbasepta} \cong P_1 \, .
\eeqn 
The sheaves $\{ \cP_{\bargenpta} \in \on{Perv}_K (\cN_\Lp) \}$ form a local system over $\La^{rs} / W_\La$. Thus, we obtain a monodromy representation:
\beq\label{eqn-mu-chi-eq-one}
\mu : B_{W_\La} \to \on{Aut} (\Pone),
\eeq
where $\Pone = P_1$. We will refer to this action as the monodromy in the family $f$.

\subsection{Monodromy in the family with twisted coefficients}
\label{subsec-twisted-monodromy}
In this subsection, we introduce the natural counterparts of the monodromy representation \eqref{eqn-mu-chi-eq-one} for a general character $\chi \in \hat I$. These constructions are not needed to state our main results, but will be used in their proofs. The situation for a general character $\chi$ is more complex than for $\chi = 1$. The local system $\cL_\chi$ does not, in general, extend to all of $\Lp^{rs}$. However, we can make the following construction. Let us write $W_{\La, \chi}$ for the stabilizer of $\chi$ in $W_\La$, i.e.,
\beqn
W_{\La, \chi} \coloneqq \on{Stab}_{W_\La} (\chi).
\eeqn
Let us also write:
\beqn
\widetilde W_{\La, \chi} \coloneqq q^{-1} (W_{\La, \chi}) \subset \widetilde W_\La \, ,
\eeqn
\beqn
B_{W_\La}^\chi \coloneqq p^{-1} (W_{\La, \chi}) \subset B_{W_\La} \, , \;\;  \widetilde B_{W_\La}^\chi \coloneqq \tilde{q}^{-1} (B_{W_\La}^\chi) \subset \widetilde{B}_{W_\La} \, .
\eeqn
Thus, we obtain the following commutative diagram:
\beqn
\begin{CD}
1 @>>> I @>>> \widetilde B_{W_\La}^\chi @>{\tilde q}>> B_{W_\La}^\chi @>>> 1 \;\,
\\
@. @| @VV{\tilde p}V @VV{p}V @.
\\
1 @>>> I @>>> \widetilde W_{\La, \chi} @>q>> W_{\La, \chi} @>>> 1 \, .
\end{CD}
\eeqn

Consider the projection $\La / W_{\La, \chi} \to \La / W_{\La}$, and form a Cartesian commutative diagram as follows:
\beq\label{diagram-ptilde-chi}
\begin{CD}
\tilde \Lp_\chi @>{g}>> \Lp
\\
@V{f_\chi}VV @VV{f}V
\\
\La / W_{\La, \chi} @>>> \La / W_\La \, .
\end{CD}
\eeq
Let $\tilde \Lp_\chi^{rs} = f_\chi^{-1} (\La^{rs} / W_{\La, \chi}) \subset \tilde \Lp_\chi$. Let $\hatbasepta \in \La / W_{\La, \chi}$ be the image of $\basepta \in \La \subset \Lp$,
and let $\tildebasepta = (\basepta, \hatbasepta) \in \tilde \Lp_\chi^{rs}$. Then $g: \tilde \Lp_\chi^{rs} \to \Lp^{rs}$ is a $K$-equivariant covering map, and we have:
\beqn
\pi_1^K (\tilde \Lp_\chi^{rs}, \tildebasepta) = \widetilde B_{W_\La}^\chi \subset \widetilde B_{W_\La} \, .
\eeqn

The splitting homomorphism $\tilde r$ of equation \eqref{eqn-tilde-r} determines a canonical extension $\hat\chi : \widetilde B_{W_\La}^{\chi} \to \{ \pm 1 \}$ of the character $\chi : I \to \{ \pm 1 \}$. More precisely, $\hat\chi$ is the unique extension such that:
\beq\label{eqn-extension-of-chi}
\hat\chi \circ \tilde r : B_{W_\La}^{\chi} \to \{ \pm 1 \} \;\; \text{is the trivial character}.
\eeq
Let $\hat \cL_\chi$ be the rank one $K$-equivariant local system on $\tilde \Lp_\chi^{rs}$, corresponding to the character $\hat\chi$.
Note that we have $(\hat \cL_\chi)_{\tildebasepta} = \baseRing$. Recall that $\hatbasepta = f_\chi (\tildebasepta) \in \La^{rs} / W_{\La, \chi}$, and let $X_{\hatbasepta} = f_\chi^{-1} (\hatbasepta)$. Then we have natural identifications:
\beqn
X_{\hatbasepta} \cong X_{\barbasepta} \;\;\;\; \text{and} \;\;\;\; \hat \cL_\chi |_{X_{\hatbasepta}} \cong \cL_\chi \, .
\eeqn
As in the case of $\chi = 1$, we construct a family:
\beqn
\xymatrix{ \cZ_\chi = \myatop{\{ (\tilde x, \hatgenpta, c) \in \tilde \Lp_\chi \times (\La^{rs} / W_{\La, \chi}) \times \bC}{\; | \; f_\chi (\tilde x) = c \, \hatgenpta \}}
\ar[rr]^-{F} \ar[rd] && \La^{rs} / W_{\La, \chi} \times \bC \ar[dl] \\& \La^{rs} / W_{\La, \chi} \, ,}
\eeqn
define the second component function $F_2 : \cZ_\chi \to \bC$, and let $\cZ_\chi^{rs} = F_2^{-1} (\bC^*)$. We continue to write $\hat \cL_\chi$ for the rank one $K$-equivariant local system on $\cZ_\chi^{rs}$ obtained by pullback from $\tilde \Lp_\chi^{rs}$. We now form:
\beqn
\cP_\chi = \psi_{F_2} (\hat \cL_\chi ) [-] \in \on{Perv}_K (\cN_\Lp \times (\La^{rs} / W_{\La, \chi})).
\eeqn
As in the case of $\chi = 1$, we have:
\beqn
(\cP_\chi)_{\hatbasepta} \cong P_\chi \, ,
\eeqn
and we obtain a monodromy representation:
\beq\label{eqn-mu}
\mu : B_{W_\La}^\chi = \pi_1 (\La^{rs} / W_{\La, \chi}, \hatbasepta) \to \on{Aut} (P_\chi).
\eeq
We will refer to this action as the monodromy in the family $f_\chi$. Note that it depends on the choice of the regular splitting \eqref{eqn-tilde-r}. By construction, the action \eqref{eqn-mu} for $\chi = 1$ coincides with the action \eqref{eqn-mu-chi-eq-one}.

In fact, there is a further monodromy structure associated to the character $\chi$, which we call extended monodromy, or extended monodromy in the family. For every $w \in W_\La$, let $w \cdot \chi \in \hat I$ be the character $u \mapsto \chi (w^{-1} \cdot u)$, $u \in I$. Then, for every $w \in W_\La$ and $b \in B_{W_\La}$, there is an extended monodromy transformation:
\beq\label{ext-mon-specific}
\mu (b) : P_{w \cdot \chi} \to P_{(p(b) w) \cdot \chi} \, .
\eeq
These transformations arise from the following construction.

Consider the natural map $\tilde g : \cZ_\chi \to \cZ$, given by:
\beqn
(\tilde x, \hatgenpta, c) \mapsto (g (\tilde x), \bargenpta, c),
\eeqn
where $\bargenpta \in \La / W_\La$ is the image of $\hatgenpta \in \La / W_{\La, \chi}$. Consider the pushforward local system $\tilde g_* \hat \cL_\chi$ on $\cZ^{rs}$. We can now form:
\beqn
\widetilde \cP_\chi = \psi_{F_2} (\tilde g_* \hat \cL_\chi) [-] \in \on{Perv}_K (\cN_\Lp \times (\La^{rs} / W_\La)).
\eeqn
The sheaf $\widetilde \cP_\chi$ is again constructible with respect to the stratification $\cS_{\cZ_0}$. For every left coset $\bar w \in W_\La / W_{\La, \chi}$, let $\bar w^{-1}$ be the corresponding right coset. We write $\bar w^{-1} \hatbasepta \in \La / W_{\La, \chi}$ for the image of $w^{-1} \basepta \in \La$, for some $w \in \bar w$. Note that this image is independent of the choice of $w \in \bar w$. Note also that the element $\bar w^{-1} \hatbasepta \in \La / W_{\La, \chi}$ depends on the choice of the lift $\basepta \in \La$ of $\hatbasepta \in \La / W_{\La, \chi}$. We then have:
\beqn
(\widetilde \cP_\chi)_{\barbasepta} \; \cong \; \bigoplus_{\bar w \in W_\La / W_{\La, \chi}} (\cP_\chi)_{\bar w^{-1} \hatbasepta} \, .
\eeqn
For each $\bar w \in W_\La / W_{\La, \chi}$, let $X_{\bar w^{-1} \hatbasepta} = f_\chi^{-1} (\bar w^{-1} \hatbasepta)$ and let:
\beqn
\bar w^{-1} \tildebasepta = (\basepta, \bar w^{-1} \hatbasepta) \in X_{\bar w^{-1} \hatbasepta} \subset \tilde \Lp_\chi \, .
\eeqn
We have a natural identification:
\beq\label{identify-fibers}
\hat \cL_{\bar w^{-1} \tildebasepta} \cong \hat \cL_{\tildebasepta} = \baseRing.
\eeq
To see this, pick an arbitrary braid $b \in B_{W_\La}$ with $p (b) \in \bar w$, and consider the image $\tilde r (b) \in \widetilde B_{W_\La} = \pi_1^K (\Lp^{rs}, \basepta)$. Isomorphism~\eqref{identify-fibers} is the holonomy operator defined by the element $\tilde r (b)$ via the $K$-equivariant covering map $g : \tilde \Lp_\chi^{rs} \to \Lp^{rs}$; it is well defined by the property~\eqref{eqn-extension-of-chi} of the extension $\hat\chi$. By tracing through the definitions, we find that isomorphism~\eqref{identify-fibers} gives rise to an isomorphism:
\beqn
(\cP_\chi)_{\bar w^{-1} \hatbasepta} \cong P_{\bar w \cdot \chi} \, ,
\eeqn
where $\bar w \cdot \chi \in \hat I$ is the character $u \mapsto \chi (w^{-1} \cdot u)$, $u \in I$, $w \in \bar w$. Thus, we obtain an extended monodromy representation:
\beq\label{ext-mon}
\mu: B_{W_\La} \to \on{Aut} \Bigg( \bigoplus_{\bar w \in W_\La / W_{\La, \chi}} P_{\bar w \cdot \chi} \Bigg).
\eeq

In order to see how this extended monodromy representation restricts to the individual summands in the RHS of \eqref{ext-mon}, we need to understand the holonomy $h_\chi$ of the covering map $\La^{rs} / W_{\La, \chi} \to \La^{rs} / W_\La \,$. This holonomy can be readily deduced from the holonomy of the larger covering $\La^{rs} \to \La^{rs} / W_\La \,$. Namely, for every $b \in B_{W_\La}$ and $\bar w \in W_\La / W_{\La, \chi}$, we have:
\beqn
h_\chi (b) : \bar w^{-1} \hatbasepta \mapsto (\bar w^{-1} p(b)^{-1}) \, \hatbasepta = (p(b) \bar w)^{-1} \, \hatbasepta \, .
\eeqn
Thus, the representation~\eqref{ext-mon} restricts to the individual summands as follows:
\beqn
\mu (b) : P_{\bar w \cdot \chi} \to P_{(p(b) \bar w) \cdot \chi} \, ,
\eeqn
for every $b \in B_{W_\La}$. This defines the transformations~\eqref{ext-mon-specific}. We will make crucial use of this extended monodromy structure.

\subsection{The group \texorpdfstring{$W_{\La, \chi}^0$}{Lg} and the Hecke algebra \texorpdfstring{$\cH_{W_{\La, \chi}^0}$}{Lg}}
\label{subsec-heck-alg}
To state our main result describing the Fourier transform of the  sheaf $P_\chi$ (Theorem \ref{thm-main}), we begin by defining a Coxeter subgroup $W_{\La, \chi}^0 \subset W_\La$.  An interpretation of the Coxeter group $W_{\La, \chi}^0$ making use of the dual group of $G$ is given in \cite[Sections 3.3-3.4]{VX}.

Recall the integer $\delta (s)$ associated to each reflection $s \in W_\La$ (see~\eqref{eqn-deltas}), and the positive real root $\alpha_s$ associated to each $s$ with $\delta (s) = 1$. For each real root $\alpha \in \Phi_\rr$, we have $\check\alpha (-1) \in Z_K(\La)$, and we view $\check\alpha (-1)$ as an element in $I$ via the natural projection map.

\begin{definition}\label{defn-W0}
We define:
\beqn
\begin{gathered}
W_{\La, \chi}^0 = \, \text{the subgroup of $W_{\La}$ generated by all reflections $s \in W_\La$ such that} \\
\text{ either $\delta (s) > 1$, or $\delta (s) = 1$ and $\chi (\check\alpha_s (-1)) = 1$.}
\end{gathered}
\eeqn
\end{definition}

By construction, $W_{\La,\chi}^0$ is a Coxeter group. A proof of the following lemma will be given in Section \ref{subsec-WandW0}.

\begin{lemma}\label{lemma-W0}
The group $W_{\La, \chi}^0$ is a subgroup of $W_{\La, \chi}$.
\end{lemma}

We now define a Hecke algebra $\cH_{W_{\La, \chi}^0}$ associated to the Coxeter group $W_{\La, \chi}^0$ as follows. The Weyl chamber $\La_\bR^+$ for $W_\La$ is contained in a unique Weyl chamber $\La_{\bR, \chi}^+ \subset \La_\bR$ for $W_{\La, \chi}^0$. Let $S_\chi = \{ s_{\bar\alpha_1}, \dots, s_{\bar\alpha_\nSchi} \}$ be the set of simple reflections for $W_{\La, \chi}^0$, corresponding to the Weyl chamber $\La_{\bR, \chi}^+$. We note that, in general, $S_\chi$ is not a subset of a set of simple reflections for $W_\La$. The Hecke algebra $\cH_{W_{\La, \chi}^0}$ is generated by the symbols $\{ T_i \}$ associated to the simple reflections $\{ s_{\bar\alpha_i} \}$. The generators $\{ T_i \}$ are subject to the braid relations for the reflections $\{ s_{\bar\alpha_i} \}$, plus the relations:
\beqn
(T_i - 1) (T_i + q_i) \ = \ 0 \, ,
\eeqn 
where the parameters $\{ q_i \}$ are defined by:
\beqn
q_i = (-1)^{\delta (s_{\bar\alpha_i})} \, .
\eeqn
We note that it is more common to replace $T_i$ by $-T_i$. The Hecke algebra $\cH_{W_{\La, \chi}^0}$ is free of rank $| W_{\La, \chi}^0 |$ over $\baseRing$.

\subsection{A character of the component group \texorpdfstring{$K / K^0$}{Lg}}
\label{subsec-char}
Recall the compact form $K_\bR$ of $K$, introduced in Section~\ref{subsec-structure-of-Wa}. Let us write $\groupM_\bR = Z_{K_\bR} (\La)$. We have:
\beq
K / K^0 = K_\bR / K_\bR^0 = \groupM_\bR / (K_\bR^0 \cap \groupM_\bR).
\eeq
Let $\Lm_\bR = \on{Lie} (\groupM_\bR)$. The adjoint action of $K_\bR$ on $\Lk_\bR$ induces an action of $\groupM_\bR$ on the top exterior power $\wedge^{top} (\Lk_\bR / \Lm_\bR)$, and we obtain a homomorphism:
\beqn
\tau': \groupM_\bR \to \on{Aut} ( \wedge^{top} (\Lk_\bR / \Lm_\bR)) = \bR^*. 
\eeqn
Since the group $\groupM_\bR$ is compact, we have $\on{Im} (\tau') \subset \{ \pm 1 \}$. The quotient space $K_\bR / \groupM_\bR = K_\bR^0 / (K_\bR^0 \cap \groupM_\bR)$ is orientable; see, for example, \cite[Corollary 1.1.10]{Ko} and the proof of Proposition~\ref{prop-fundamental-class} below. It follows that $\tau' |_{K_\bR^0 \cap \groupM_\bR} = 1$, and we obtain a character:
\beqn
\tau: I \to I / I^0 \cong K / K^0 \to \{ \pm 1 \}.
\eeqn
Recall that $W_\La$ acts trivially on $I / I^0$, and hence it acts trivially on $\tau$. Thus, we can use the splitting homomorphism $\tilde r$ of equation \eqref{eqn-tilde-r} to extend $\tau$ to a character:
\beq\label{character-tau}
\tau: \widetilde B_{W_\La} \to \{ \pm 1\},
\eeq
such that $\tau \circ \tilde r (b) = 1$ for every $b \in B_{W_\La}$.

\subsection{Statement of the theorem}
\label{subsec-main-thm}
We begin by constructing a $\baseRing [\widetilde B_{W_\La}]$-module $\cM_\chi$ associated to a character $\chi \in \hat I$.

Pick a basepoint $l \in \La_\bR^+$. All fundamental groups in this subsection will be taken relative to $l$, and we use Remark \ref{rmk-basepoint-indep} to identify $\widetilde B_{W_\La}$ with $\pi_1^K (\Lp^{rs}, l)$, etc. Recall the maps $p: B_{W_\La} \to W_\La$ and $\tilde q : \widetilde B_{W_\La} \to B_{W_\La}$ of diagram \eqref{mainDiagram}. Let us write:
\beqn
B_{W_\La}^{\chi, 0} \coloneqq p^{-1} (W_{\La, \chi}^0) \;\;\;\; \text{and} \;\;\;\;
\widetilde B_{W_\La}^{\chi, 0} \coloneqq {\tilde q}^{-1} (B_{W_{\La}}^{\chi, 0}) \ \cong \ I \rtimes B_{W_\La}^{\chi, 0} \, .
\eeqn
Let $\La_\chi^{rs} \subset \La$ be the complement of the hyperplane arrangement corresponding to the Coxeter group $W_{\La, \chi}^0$. We have $\La^{rs} \subset \La_\chi^{rs}$. Let $B_{W_{\La, \chi}^0} = \pi_1 (\La_\chi^{rs} / W_{\La, \chi}^0, l)$ be the braid group of $W_{\La, \chi}^0$. Note that the basepoint $l$ lies in $\La^{rs} \subset \La_\chi^{rs}$ rather than in the quotient $\La_\chi^{rs} / W_{\La, \chi}^0$; but the above fundamental group notation is unambiguous. We then have the following commutative diagram:
\beqn
\xymatrix{
1 \ar[r] & \pi_1 (\La^{rs}) \ar[r] & B_{W_\La} = \pi_1(\La^{rs} / W_\La) \ar[r] & W_\La \ar[r] & 1 \;\,
\\
1 \ar[r] & \pi_1 (\La^{rs}) \ar[r] \ar@{=}[u] \ar@{->>}_{\varphi_1}[d] & B_{W_\La}^{\chi, 0} = \pi_1(\La^{rs} / W_{\La,\chi}^0) \ar[r] \ar@{^(->}[u] \ar@{->>}_{\varphi}[d] & W_{\La, \chi}^0 \ar[r] \ar@{^(->}[u] \ar@{=}[d] & 1 \;\,
\\
1 \ar[r] & \pi_1 (\La_\chi^{rs}) \ar[r] & B_{W_{\La, \chi}^0} = \pi_1(\La_\chi^{rs} / W_{\La, \chi}^0) \ar[r] & W_{\La, \chi}^0 \ar[r] & 1 \, ,
}
\eeqn
where the vertical maps from the second row to the third are induced by the inclusion $\La^{rs} \subset \La_\chi^{rs}$. Thus, we obtain the following morphisms of rings:
\beqn
\baseRing [B_{W_{\La, \chi}^0}] \xleftarrow{\;\;\; \widetilde \varphi \;\;\;} \baseRing [\widetilde B_{W_\La}^{\chi, 0}] \xhookrightarrow{\;\;\;\;\;\;\;\;} \baseRing [\widetilde B_{W_\La}] \, ,
\eeqn
where the morphism $\widetilde \varphi$ is induced by the composition $\varphi \circ \tilde q : \widetilde B_{W_\La}^{\chi, 0} \to B_{W_{\La, \chi}^0}$.

Recall the Hecke algebra $\cH_{W_{\La,\chi}^0}$ with specific parameters introduced in Section~\ref{subsec-heck-alg}. Let:
\beqn
\eta_\chi : \baseRing [B_{W_{\La, \chi}^0}] \to \cH_{W_{\La, \chi}^0} \, ,
\eeqn
be the canonical quotient map. More precisely, let $\{ \sigma_1, \dots, \sigma_\nSchi \} \subset B_{W_{\La, \chi}^0}$ be the counter-clockwise braid generators associated to the simple reflections $\{ s_{\bar\alpha_1}, \dots, s_{\bar\alpha_\nSchi} \}$ of Section \ref{subsec-heck-alg}. Then the map $\eta_\chi$ is defined by:
\beqn
\eta_\chi : \sigma_i \mapsto T_i, \;\; i = 1, \dots, \nSchi.
\eeqn
We can now view the Hecke algebra $\cH_{W_{\La, \chi}^0}$ as a $\baseRing [\widetilde B_{W_\La}^{\chi, 0}]$-module via the composition:
\beqn
\eta_\chi \circ \widetilde \varphi : \baseRing [\widetilde B_{W_\La}^{\chi, 0}] \to \cH_{W_{\La, \chi}^0} \, .
\eeqn

Recall that the character $\chi : I \to \{ \pm 1 \}$ extends uniquely to a character $\hat\chi : \widetilde B_{W_\La}^\chi \to \{ \pm 1 \}$, such that $\hat\chi \circ \tilde r (b) = 1$ for every $\braid \in B_{W_\La}^\chi$ (see \eqref{eqn-extension-of-chi}). Note that Lemma~\ref{lemma-W0} implies that we have $\widetilde B_{W_\La}^{\chi, 0} \subset \widetilde B_{W_\La}^\chi$. Let $\baseRing_\chi$ denote a copy of $\baseRing$, viewed as a $\baseRing [\widetilde B_{W_\La}^{\chi, 0}]$-module via the restriction of $\hat\chi$ to $\widetilde B_{W_\La}^{\chi, 0}$.

Consider the tensor product $\baseRing_\chi \otimes \cH_{W_{\La, \chi}^0}$, taken over $\baseRing$, and view it as a $\baseRing [\widetilde B_{W_\La}^{\chi, 0}]$-module by combining the $\baseRing [\widetilde B_{W_\La}^{\chi, 0}]$-module structures on $\baseRing_\chi$ and $\cH_{W_{\La, \chi}^0}$. We now induce this module to $\baseRing [\widetilde B_{W_\La}]$, and define:
\beq\label{eqn-M-chi}
\cM_\chi \ = \ \left( \baseRing [\widetilde B_{W_\La}] \otimes_{\baseRing [\widetilde B_{W_\La}^{\chi, 0}]} (\baseRing_\chi \otimes \cH_{W_{\La, \chi}^0}) \right) \otimes \baseRing_\tau \, ,
\eeq
where $\baseRing_\tau$ is a copy of $\baseRing$, viewed as a $\baseRing [\widetilde B_{W_\La}]$-module via the character $\tau$ of equation~\eqref{character-tau}. Note that $\cM_\chi$ is free of rank $|W_\La|$ over $\baseRing$. We interpret the $\baseRing [\widetilde B_{W_\La}]$-module $\cM_\chi$ as a $K$-equivariant
local system on $\Lp^{rs}$, whose fiber over the basepoint $l \in \La^{rs}$ is equal to $\cM_\chi$, and whose holonomy is given by the $\baseRing [\widetilde B_{W_\La}]$-module structure.

\begin{thm}
\label{thm-main}
Consider the perverse sheaf $P_\chi$ defined in~\eqref{eqn-P-chi}. Its Fourier transform is given by:
\beqn
\fF P_\chi \cong \on{IC} (\Lp^{rs}, \cM_\chi),
\eeqn
where the RHS is the IC-extension of the local system $\cM_\chi$ defined in~\eqref{eqn-M-chi} above.
\end{thm}

\subsection{The Case \texorpdfstring{$\chi = 1$}{Lg}}
Note that Theorem \ref{thm-main} describes the perverse sheaf $P_\chi$, but not the monodromy action:
\beqn
\mu : B_{W_\La}^\chi \to \on{Aut} (P_\chi),
\eeqn
of equation \eqref{eqn-mu}. Consider now the special case $\chi =1$. In this case, we can readily describe both the sheaf $P_\chi$ and the action $\mu$. In fact, the statement of Theorem \ref{thm-main} for $\chi = 1$ is a partial paraphrase of \cite[Theorem 6.1]{G2} (see also \cite[Section 4]{G4}). Here we adapt the full statement of \cite[Theorem 6.1]{G2} to our setting. From now on, we write $\Pone = P_1$.

Assume that $\chi = 1$, and write $\cH_1 = \cH_{W_{\La, \chi}^0}$ for short. Note that we have $W_{\La, \chi}^0 = W_\La$, so we have $S_\chi = S$, and the Hecke algebra $\cH_1$ is naturally a quotient of the group algebra $\baseRing [B_{W_\La}]$. Write:
\beqn
\eta = \eta_1 : \baseRing [B_{W_\La}] \to \cH_1 \, ,
\eeqn
for the quotient map. Next, define an involution $\omega : \cH_1 \to \cH_1$ by:
\beqn
\omega : T_i \mapsto (-1)^{\delta (s_{\bar\alpha_i})} \cdot T_i \, , \;\; i = 1, \dots, \nSchi.
\eeqn
Let $\on{End}_\baseRing (\cH_1)$ be the endomorphism ring of $\cH_1$ as a $\baseRing$-module, and let:
\beqn
L_{\cH_1}, R_{\cH_1} : \cH_1 \to \on{End}_\baseRing (\cH_1),
\eeqn
be the $\baseRing$-algebra homomorphisms given by:
\beqn
L_{\cH_1} (T_i) : h \mapsto T_i \, h \;\;\;\; \text{and} \;\;\;\; R_{\cH_1} (T_i) : h \mapsto h \, T_i \, , \;\;\;\; i = 1, \dots, \nSchi.
\eeqn

Define a $\baseRing [\widetilde B_{W_\La}]$-module $\cM$ as follows. Let $\cM = \cH_1$ as $\baseRing$-modules, and let $\baseRing [\widetilde B_{W_\La}]$ act on $\cM$ via the homomorphism:
\beqn
(L_{\cH_1} \circ \eta \circ \tilde q) \otimes \tau : \widetilde B_{W_\La} \to \on{End}_\baseRing (\cM).
\eeqn
It is not hard to see that we have $\cM \cong \cM_1$ as $\baseRing [\widetilde B_{W_\La}]$-modules, where $\cM_1$ is the module defined in Section~\ref{subsec-main-thm} for $\chi = 1$. We interpret the module $\cM$ as a $K$-equivariant local system on $\Lp^{rs}$, whose fiber over the basepoint $l \in \La^{rs}$ is equal to $\cM$, and whose holonomy is given by the $\baseRing [\widetilde B_{W_\La}]$-module structure. Write $\mu_l : B_{W_\La} \to \on{Aut} ((\fF \Pone)_l)$ for the action induced by the monodromy in the family $\mu : B_{W_\La} \to \on{Aut} (\Pone)$ of equation \eqref{eqn-mu-chi-eq-one}.

\pagebreak

\begin{thm}\label{thm-chi-eq-one}
In the situation of Theorem \ref{thm-main}, assume that $\chi = 1$. 
\begin{enumerate}[topsep=-1.5ex]
\item[(i)]  We have $\fF \Pone \cong \on{IC} (\Lp^{rs}, \cM)$.

\item[(ii)] The monodromy action $\mu_l : B_{W_\La} \to \on{Aut} (\cM)$ is given by $\mu_l = R_{\cH_1} \circ \omega \circ \eta$.
\end{enumerate}
\end{thm}

\begin{remark}
For $\chi \in \hat I$, let $\cH_\chi^\circ$ be the opposite of the Hecke algebra $\cH_{W_{\La, \chi}^0}$. Theorem \ref{thm-chi-eq-one} implies that the monodromy action $\mu$ of equation \eqref{eqn-mu-chi-eq-one} generates the endomorphism ring $\on{End} (\Pone)$, and we have $\on{End} (\Pone) \cong \cH_1^\circ$. At the time of this writing, we do not see a similarly nice description of $\on{End} (P_\chi)$ for a general character $\chi \in \hat I$. However, in situations where $W_{\La, \chi}^0 = W_{\La, \chi}$, Theorem \ref{thm-main} implies that $\on{End} (\Pone) \cong \cH_\chi^\circ$, and is generated by the monodromy action $\mu$ of equation \eqref{eqn-mu}. In situations where $W_{\La, \chi}^0 \subset W_{\La, \chi}$ is a proper subgroup, we expect that $\on{End} (P_\chi)$ is free of rank $|W_{\La, \chi}|$ over $\baseRing$, and is still generated by the action $\mu$.
\end{remark}

\section{An example: the pair \texorpdfstring{$(SL(2), SO(2))$}{Lg}}
\label{sec-example}

In this section, we describe the nearby cycle sheaves $\{ P_\chi \}_{\chi \in \hat I}$, and the actions of the monodromy in the family on these sheaves, in the case of the symmetric pair $(SL(2), SO(2))$ and $\baseRing = \bC$. In particular, we illustrate the dependence of the monodromy in the family \eqref{eqn-mu} on the choice of the splitting homomorphism $\tilde r : B_{W_\La} \to \widetilde B_{W_\La}$. In addition to illustrating Theorems \ref{thm-main} and \ref{thm-chi-eq-one}, this example will be used in the proof of Theorem \ref{thm-main}.

The reason for assuming that $\baseRing = \bC$ is to ensure that the category $\on{Perv}_K (\cN_\Lp)$ is Artinian, so we can discuss simple objects and composition series. In this example, we could have chosen $\baseRing$ to be any other field of characteristic $\neq 2$. Proposition \ref{prop-example} below will describe the structure of the sheaves $\{ P_\chi \}_{\chi \in \hat I}$ as objects in $\on{Perv}_K (\cN_\Lp)$ for $\baseRing = \bC$.

Assume that $G = SL(2)$ and $K = SO(2)$. Then we have $\Lp \cong \bC^2$, $\La \cong \bC$, $W_\La = \bZ / 2$, and $f : \Lp \to \La / W_{\La}$ is a non-degenerate quadric. The nilcone $\cN_\Lp$ is a pair of lines through the origin, and it consists of three $K$-orbits:
\beqn
\cN_\Lp = \bigcup_{i = 0}^2 \cO_i \, , 
\eeqn
where $\cO_0 = \{ 0 \}$, and $\cO_1 \cong \cO_2 \cong \bC^*$. We have $\pi_1^K (\cO_1) \cong \pi_1^K (\cO_2) \cong \bZ / 2$.
For $i = 1, 2$, let $\cE_0 [\cO_i]$ and $\cE_1 [\cO_i]$ be the rank one local systems on $\cO_i$ with holonomy $1$ and $-1$, respectively.  Note that $\cE_0 [\cO_i] \cong \baseRing_{\cO_i}$. Similarly, let $\cE_0 [\cO_0] = \baseRing_{\cO_0}$.

Turning now to diagram \eqref{mainDiagram}, it takes the following form:
\beq\label{mainDiagram-example}
\begin{CD}
1 @>>> \bZ / 2 @>>> \bZ \times_{\bZ / 2} \bZ / 4 @>{\tilde{q}}>> \bZ @>>> 1 \,
\\
@. @| @VV{\tilde p}V @VV{p}V @.
\\
1 @>>> \bZ / 2 @>>> \bZ / 4 @>{q}>> \bZ / 2 @>>> 1 \, .
\end{CD}
\eeq
There are two characters $\chi_0, \chi_1 : I \to \{ \pm 1\}$, where $\chi_0 = 1$ is the trivial character. We write $\cL_{\chi_0}, \cL_{\chi_1}$ for the corresponding local systems on the general fiber $X_{\barbasepta}$. We have:
\beqn
W_{\La, \chi_0} = W_{\La, \chi_1} = W_\La \, .
\eeqn
Consequently, we have $\tilde \Lp^{rs}_{\chi_0} = \tilde \Lp^{rs}_{\chi_1} = \Lp^{rs}$. Let $s \in W_\La$ be the non-trivial element. There are two choices for the regular splitting homomorphism $\tilde r : B_{W_\La} \to \widetilde B_{W_\La}$, corresponding to the two elements of the preimage $q^{-1} (s)$. We denote these two choices by $\tilde r_1, \tilde r_2 : B_{W_\La} \to \widetilde B_{W_\La}$. To distinguish between $\tilde r_1$ and $\tilde r_2$, consider the corresponding extensions $\hat \cL_{\chi_1, 1}$ and $\hat \cL_{\chi_1, 2}$ of $\cL_{\chi_1}$ to $\Lp^{rs}$. Note that $\pi_1 (\Lp^{rs}, \basepta) \cong \bZ^2$. For $i = 1, 2$, let $\gamma_i \in \pi_1 (\Lp^{rs}, \basepta)$ be the element represented by a small counter-clockwise loop linking the orbit $\cO_i$. We assume that the holonomy of $\hat \cL_{\chi_1, i}$ around $\gamma_i$ is trivial ($i = 1, 2$). Note that the trivial local system $\cL_{\chi_0}$ extends to the trivial local system $\hat \cL_{\chi_0}$ on $\Lp^{rs}$, regardless of which of the splitting homomorphisms $\tilde r_1, \tilde r_2$ is used.

Let $P_{\chi_0}, P_{\chi_1}$ be the nearby cycle sheaves corresponding to the characters $\chi_0, \chi_1$, as in Section~\ref{subsec-twisted}. The sheaf $P_{\chi_0}$ is equipped with a monodromy transformation:
\beq\label{eqn-mu0}
\mu_0 = \mu_0 (\sigma_s) : P_{\chi_0} \to P_{\chi_0} \, ,
\eeq
whereas the sheaf $P_{\chi_1}$ is equipped with a pair of monodromy transformations:
\beq\label{eqn-mu1}
\mu_{1, 1}, \mu_{1, 2} : P_{\chi_1} \to P_{\chi_1} \, ,
\eeq
where $\mu_{1, i} = \mu_{1, i} (\sigma_s)$ corresponds to the splitting homomorphism $\tilde r_i$ ($i = 1, 2$; see discussion following \eqref{eqn-mu}).

\begin{prop}\label{prop-example}
Assume that $\baseRing = \bC$.  The properties of the sheaves $P_{\chi_0}, P_{\chi_1} \in \on{Perv}_K (\cN_\Lp)$ can be summarized as follows. \\
\indent
(i) The sheaf $P_{\chi_0}$ has four simple constituents: 
\beqn
\text{two copies of} \;\; \IC (\cO_0, \cE_0 [\cO_0]), \;\; \IC (\cO_1, \cE_0 [\cO_1]), \;\; \text{and} \;\; \IC (\cO_2, \cE_0 [\cO_2]).
\eeqn
The monodromy transformation $\mu_0 \in \on{End} (P_{\chi_0})$ generates the endomorphism ring, and satisfies:
\beqn
(\mu_0 - 1)^2 = 0.
\eeqn
We have $\on{Im} (\mu_0 - 1) \cong \on{Coker} (\mu_0 - 1) \cong \IC (\cO_0, \cE_0 [\cO_0])$. Thus, we have:
\beqn
\on{End} (P_{\chi_0}) = \bC [\mu_0] / (\mu_0 - 1)^2.
\eeqn
\indent
(ii) We have:
\beqn
P_{\chi_1} \cong \IC (\cO_1, \cE_1 [\cO_1]) \oplus \IC (\cO_2, \cE_1 [\cO_2]).
\eeqn
In terms of the above direct sum decomposition, we can describe the monodromy transformations \eqref{eqn-mu1} in matrix form, as follows:
\beqn
\mu_{1, 1} = 
\begin{pmatrix}
1   &  0   \\
0   &  -1
\end{pmatrix}, \;\;\;\;
\mu_{1, 2} = 
\begin{pmatrix}
-1  &  0   \\
0   &  1
\end{pmatrix}.
\eeqn
Thus, we have:
\beqn
\on{End} (P_{\chi_1}) = \bC [\mu_{1, 1}] / (\mu_{1, 1}^2 - 1) = \bC [\mu_{1, 2}] / (\mu_{1, 2}^2 - 1).
\eeqn
\end{prop}

\begin{proof}
Part (i) is a special case of \cite[Examples 3.7, 3.8]{G2}. For part (ii), we can readily check that the restrictions of $P_{\chi_1}$ to the open orbits $\cO_1$ and $\cO_2$ are given by:
\beqn
P_{\chi_1} |_{\cO_i} \cong \cE_1 [\cO_i] [1], \;\; i = 1, 2.
\eeqn
The actions of the monodromy transformations $\mu_{1, 1}$ and $\mu_{1, 2}$ on these restrictions are likewise readily computed from the definitions. Let $h \in \Lp^*$ be a generic covector at the origin for the $K$-orbit stratification of $\cN_\Lp$, and let:
\beqn
M_{(0, h)} : \on{Perv}_K (\cN_\Lp) \to \{ \text{$\bC$-vector spaces} \},
\eeqn
be the corresponding Morse group functor. By Morse theory, we have:
\beqn
M_{(0, h)} (P_{\chi_1}) \cong \bC^2.
\eeqn
By a standard calculation with perverse sheaves on the complex line, we have:
\beqn
M_{(0, h)} (\IC (\cO_1, \cE_1 [\cO_1])) \cong M_{(0, h)} (\IC (\cO_2, \cE_1 [\cO_2])) \cong \bC.
\eeqn
It follows that the above $\IC$ sheaves are the only simple constituents of $P_{\chi_1}$. The direct sum decomposition claim follows from considering the actions of the monodromy transformations $\mu_{1, 1}$ and $\mu_{1, 2}$.
\end{proof}

The assertions of Proposition \ref{prop-example} concerning the equations for the monodromy in the family translate readily to the case of the general $\baseRing$.

\begin{corollary}\label{cor-example}
The monodromy transformations \eqref{eqn-mu0}, \eqref{eqn-mu1} satisfy:
\beqn
(\mu_0 - 1)^2 = 0, \quad (\mu_{1, 1}^2 - 1) = (\mu_{1, 2}^2 - 1) = 0,
\eeqn
for every coefficient ring $\baseRing$ as in Section \ref{sec-preliminaries}.
\end{corollary}

\begin{proof}
This follows from Proposition \ref{prop-example}, plus the observation that all Morse groups of the sheaves $P_{\chi_0}, P_{\chi_1}$ are free $\baseRing$-modules.  Thus, the case $\baseRing = \bC$ implies the case $\baseRing = \bZ$, and the case $\baseRing = \bZ$ implies the general case.
\end{proof}

\section{Proofs of key preliminary results}
\label{sec-proofs-prelim}

In this section, we collect the proofs of the key preliminary results that we used in order to state Theorems \ref{thm-main} and \ref{thm-chi-eq-one}. In Section~\ref{subsec-rank-one}, we discuss rank one symmetric pairs associated to a reflection $s \in W_\La$, and give a proof of Lemma~\ref{lemma-trichotomy} which describes the subgroup $\widetilde W_s \subset \widetilde W_\La$. In Section \ref{subsec-WandW0}, we give a proof of Lemma \ref{lemma-W0}, establishing the inclusion $W_{\La, \chi}^0 \subset W_{\La, \chi}$. In Section \ref{subsec-regularity-of-kr} we give a proof of Proposition \ref{prop-kr-is-regular}, establishing the regularity of Kostant-Rallis splittings. In Section \ref{subsec-splitting}, we give a proof of Lemma~\ref{lemma-braid-relations}, which provides a full classification of regular splittings. In Section \ref{subsec-proof-of-asubf2}, we give a proof of Proposition \ref{prop-asubf2}, which provides an underpinning for the monodromy in the family structures introduced in Sections \ref{subsec-twisted} and \ref{subsec-twisted-monodromy}.

\subsection{Rank one symmetric pairs associated to a reflection}
\label{subsec-rank-one}
We define the (semisimple) rank of the pair $(G, K)$ by:
\beqn
\on{rank} (G, K) = \dim \La - \dim Z (\Lg) \cap \Lp.
\eeqn
Fix a reflection $s \in W_\La$, as in Section~\ref{subsec-regular-splitting}. In this subsection, we introduce rank one symmetric pairs $(G_s, K_s)$, $(\bar G_s, \bar K_s)$, $(\bar G'_s, \bar K'_s)$ associated to $s$, and use them to give a proof of Lemma~\ref{lemma-trichotomy}. A number of key arguments in the rest of this paper will also make use of these rank one symmetric pairs.

Begin by setting $G_s = Z_G (\La_s)$, and recall that we write $K_s = Z_K (\La_s)$. Note that the group $G_s$ is connected. We let $\Lg_s = \on{Lie} (G_s)$ and $\Lp_s = \Lg_s \cap \Lp$, so that $\Lg_s = \Lk_s \oplus \Lp_s$. Recall the subset $\Phi_s \subset \Phi$ of equation~\eqref{eqn-Phi-s}. We have:
\beqn
\Lg_s = Z_{\Lk} (\La) \oplus \La \oplus \bigoplus_{\alpha \in \Phi_s} \Lg_\alpha \, .
\eeqn
Let $\La_s^\p \subset \La$ be the orthogonal complement to $\La_s$, i.e., the $-1$-eigenspace of the reflection $s$. We define:
\beqn
\bar \Lg_s = Z_{\Lk} (\La) \oplus \La_s^\p \oplus \bigoplus_{\alpha \in \Phi_s} \Lg_\alpha \subset \Lg_s \, .
\eeqn
By construction, $\bar \Lg_s$ is the Lie algebra of a connected closed subgroup $\bar G_s \subset G_s$. We let $\bar K_s = \bar G_s \cap K$ and note that $\on{Lie} (\bar K_s) = \Lk_s$. We also let $\bar \Lp_s = \bar \Lg_s \cap \Lp$, so that $\bar \Lg_s = \Lk_s \oplus \bar \Lp_s$. Note that $\La_s^\p \subset \bar \Lp_s$ is a Cartan subspace for the pair $(\bar G_s, \bar K_s)$, and we have:
\beq\label{eqn-barLp-s}
\dim \bar \Lp_s = \delta (s) + 1, \;\;\;\; \bar \Lp_s = \La_s^\p \oplus [\Lk_s, \La], \;\;\;\; \Lp_s = \La \oplus [\Lk_s, \La].
\eeq

For some purposes, it will be convenient to further reduce the dimension of the group $\bar G_s$. Consider the centralizer $Z_{\bar K_s} (\bar \Lp_s)$ and let $Z_{\bar K_s} (\bar \Lp_s)^0 \subset Z_{\bar K_s} (\bar \Lp_s)$ be the identity component. Note that $Z_{\bar K_s} (\bar \Lp_s)^0$ is a normal subgroup of $\bar G_s$. We let:
\beqn
\bar G'_s = \bar G_s / Z_{\bar K_s} (\bar \Lp_s)^0 \;\;\;\; \text{and} \;\;\;\; \bar K'_s = \bar K_s / Z_{\bar K_s} (\bar \Lp_s)^0 \, .
\eeqn

Note that the little Weyl group of each of the symmetric pairs $(G_s, K_s)$, $(\bar G_s, \bar K_s)$, $(\bar G'_s, \bar K'_s)$ is equal to $\bZ / 2 = \{ 1, s \}$.

\begin{proof}[Proof of Lemma~\ref{lemma-trichotomy}]
For part (i), by applying equation~\eqref{eqn-little-Weyl-redefined} to the pair $(G_s, K_s)$, we see that the group $W_s$ is the little Weyl group of this pair, which is $\bZ / 2 = \{ 1, s \}$, as required.

For part (ii), let:
\beq\label{eqn-subsl2}
\Lg_{\pm \alpha} = \La_s^\p \oplus \Lg_{\alpha} \oplus \Lg_{-\alpha} \subset \bar \Lg_s
\;\;\;\; \text{and} \;\;\;\;
\Lk_{\pm \alpha} = \Lg_{\pm \alpha} \cap \Lk \subset \Lk_s \, .
\eeq
By construction, the subspace $\Lg_{\pm \alpha}$ is a Lie algebra isomorphic to $\Ls\Ll (2)$. Fix a group homomorphism $\varphi : SL (2) \to \bar G_s$, such that $d \varphi (\Ls\Ll (2)) = \Lg_{\pm \alpha}$ and $d \varphi (\Ls\Lo (2)) = \Lk_{\pm \alpha}$. Note that we have $\varphi (-1) = \check\alpha (-1)$. Since $-1 \in SO (2) \subset SL (2)$ and $\varphi (SO (2)) \subset K_s^0$, we can conclude that $\check\alpha (-1) \in Z_{K_s^0} (\La)$. The claim follows.

For part (iii), assume that $\delta (s) >1$, and let:
\beqn
\bar f_s : \bar \Lp_s \to \bar \Lp_s \inv \bar K_s = \bar \Lp_s \inv \bar K_s^0 \cong \La_s^\p / W_s \, ,
\eeqn
be the quotient map. We have an isomorphism $\La_s^\p / W_s \cong \bC$, which is canonical up to scalar multiplication. In terms of this isomorphism, the map $\bar f_s$ is given by a non-degenerate homogeneous quadratic polynomial. Since $\dim \bar \Lp_s = \delta (s) + 1 > 2$ (see equation~\eqref{eqn-barLp-s}), we can conclude that the general fiber of $\bar f_s$ is simply connected. Note that $K_s^0 = \bar K_s^0$. It follows that the centralizer $Z_{K_s^0} (\La) = Z_{\bar K_s^0} (\La_s^\p)$ is connected. The claim follows.

For part (iv), assume that $\delta (s) = 1$, and define subalgebras $\Lg_{\pm \alpha_s} \subset \bar \Lg_s$, $\Lk_{\pm \alpha_s} \subset \Lk_s$ as in equation~\eqref{eqn-subsl2}. Fix a homomorphism $\varphi : SL (2) \to \bar G_s$ such that $d \varphi (\Ls\Ll (2)) = \Lg_{\pm \alpha_s}$ and $d \varphi (\Ls\Lo (2)) = \Lk_{\pm \alpha_s}$, as in the proof of part (ii) above. As before, we have $\varphi (-1) = \check\alpha_s (-1)$. Furthermore, at the level of Lie algebras, we have:
\beqn
\bar \Lg_s = Z_{\Lk_s} (\bar \Lp_s) \oplus \Lg_{\pm \alpha_s} \, .
\eeqn
Therefore, the homomorphism $\varphi$ extends to a surjection:
\beqn
\widetilde\varphi : Z_{\bar K_s} (\bar \Lp_s)^0 \times SL (2) \to \bar G_s \, ,
\eeqn
such that $\widetilde\varphi (Z_{\bar K_s} (\bar \Lp_s)^0 \times SO (2)) = K_s^0$. It follows that we have:
\beqn
Z_{K_s^0} (\La) = \widetilde\varphi (Z_{\bar K_s} (\bar \Lp_s)^0 \times \{ \pm 1 \}).
\eeqn
The claim that $I_s = \langle \check\alpha_s (-1) \rangle$ follows, and with it part (iv)(a) of the Lemma. In the situation of part (iv)(b), we can conclude that the surjection $\widetilde\varphi$ is an isomorphism, and we have:
\beqn
(\bar G'_s, \bar K'_s) \cong (SL(2), SO(2)).
\eeqn
The claim now follows from an easy calculation for the symmetric pair $(SL(2), SO(2))$, as in diagram~\eqref{mainDiagram-example}.
\end{proof}

As in the proof above, we will write $\bar f_s : \bar \Lp_s \to \La_s^\p / W_s$ for the quotient map. We will also make use of the quotient map:
\beqn
f_s : \Lp_s \to \Lp_s \inv K_s = \Lp_s \inv K_s^0 \cong \La / W_s \, .
\eeqn
We have the following commutative diagram:
\beq\label{diagramFs}
\begin{CD}
\bar \Lp_s @>>> \Lp_s @>>> \Lp
\\
@V{\bar f_s}VV @V{f_s}VV @VV{f}V
\\
\La_s^\p / W_s @>>> \La / W_s @>>> \La / W_\La \, ,
\end{CD}
\eeq
where the bottom right horizontal arrow is a projection map, while the other three horizontal arrows are inclusion maps. Moreover, using the identifications $\Lp_s = \bar \Lp_s \times \La_s$ and $\La / W_s = (\La_s^\p / W_s) \times \La_s$, we can write:
\beq\label{eqn-Fs-product}
f_s (x, a) = (\bar f_s (x), a), \text{ for all } \, x \in \bar \Lp_s \, , \; a \in \La_s \, .
\eeq

\subsection{The stabilizer group \texorpdfstring{$W_{\La, \chi}$}{Lg} and the Coxeter subgroup \texorpdfstring{$W_{\La, \chi}^0$}{Lg}}
\label{subsec-WandW0}
Let us fix a character $\chi \in \hat I $. In this subsection, we analyze the stabilizer group $W_{\La, \chi}$ and prove Lemma~\ref{lemma-W0}, i.e., the containment $W_{\La, \chi}^0 \subset W_{\La, \chi}$. 

\begin{lemma}\label{lemma-stab}
We have: 
\begin{enumerate}[topsep=-1.5ex]
\item[(i)] $W_\rc^\theta \subset W_{\La, \chi}$.

\item[(ii)] If $s = s_{\bar\alpha} \in W_\La$, $\alpha \in \Phi_\rr$, and $\chi(\check\alpha (-1)) = 1$, then $s \in W_{\La, \chi}$.
\end{enumerate}
\end{lemma}

\begin{proof}
Recall that $A[2]$ denotes the set of order 2 elements in $A$. We have a bijection: 
\beqn
X_* (A) / 2 X_* (A) \to A[2], \;\; \text{given by} \;\; x \mapsto x (-1).
\eeqn
We regard an element of $X_* (A)$ as an element of $X_* (T)$ via the inclusion $A \to T$. Note that for every $x \in X_* (A)$, we have $\theta(x) = -x$.

We now prove part (i). Let $\alpha \in \Phi_\rc$ and $x \in X_* (A)$. Then $\alpha$ and $\theta \alpha$ are perpendicular (see~\eqref{perp-roots}). Using~\eqref{restriction-of-complex-roots}, we have: 
\beqn
s_{\bar\alpha} (x) = s_{\alpha} s_{\theta \alpha} (x) = x - \la \alpha, x \ra \check\alpha - \la \theta \alpha, x \ra \theta \check\alpha = x - \la \alpha, x \ra \check\alpha + \la \alpha, x \ra \theta \check\alpha \, .
\eeqn
It follows that:
\beqn
s_{\alpha} s_{\theta\alpha} \left( x (-1) \right) = x (-1) \left( \check\alpha (-1) \, \theta \check\alpha (-1) \right)^{\la \alpha, x \ra}.
\eeqn
We claim that $\check\alpha (-1) \, \theta \check\alpha (-1) \in C^0 = (T^\theta)^0$. This can be seen as follows. Let us write $\check\alpha (-1) = t_1 t_2$, where $t_1 \in C$ and $t_2 \in A$. Then we have $\check\alpha (-1) \, \theta \check\alpha (-1) = t_1^2$, and since $I = C / C^0$ is a 2-group, we have $t_1^2 \in C^0$. In view of~\eqref{Wctheta} and~\eqref{surjectionOnI}, we conclude that $W_\rc^\theta$ acts trivially on $I$, and thus is automatically contained in $W_{\La, \chi}$.

The proof of part (ii) is similar. Let $\alpha \in \Phi_\rr$ and $x \in X_* (A)$. We have:
\beqn
s_{\bar\alpha} \, (x (-1)) = x (-1) \, \check\alpha (-1)^{\la \alpha, x \ra}.
\eeqn
If $\chi (\check\alpha (-1)) = 1$, then we have:
\beqn
\chi (s_{\bar\alpha} \, (x (-1))) = \chi (x (-1)) \;\; \text{for every} \;\; x \in X_* (A).
\eeqn
The claim follows.
\end{proof}

\begin{corollary} 
We have:
\beqn
W_{\La, \chi} \; = \; W_\rc^\theta \ltimes W_{\rr, \chi} \, , \;\; \text{where} \;\; W_{\rr, \chi} \; = \; W_\rr \cap W_{\La, \chi} \, .
\eeqn
\end{corollary}

\begin{proof}
This follows from Lemma \ref{lemma-stab} (i) and equation~\eqref{little-Weyl-group}.
\end{proof}

\begin{corollary}\label{cor-delta-gt-one}
Let $s \in W_\La$ be a reflection. If $\delta (s) > 1$, then $s \in W_{\La, \chi}$.
\end{corollary}

\begin{proof}
In view of equation~\eqref{reflections} and Lemma~\ref{lemma-stab} (i), we only need to consider the case when $s = s_{\bar\alpha}$, for $\alpha \in \Phi_\rr$. In this case, if $\delta (s) > 1$, then Lemma~\ref{lemma-trichotomy} implies that $\check\alpha (-1) = 1$ in $I$. Thus by Lemma~\ref{lemma-stab} (ii), we have $s \in W_{\La, \chi}$.
\end{proof}

Lemma~\ref{lemma-W0} now follows from Lemma~\ref{lemma-stab} (ii) and Corollary~\ref{cor-delta-gt-one}. We note the following two corollaries of Lemma~\ref{lemma-W0}.

\begin{corollary}\label{cor-sInW0}
Let $s \in W_\La$ be a reflection. Then we have:
\beqn
s \in W_{\La, \chi}^0 \;\;\;\; \text{if and only if} \;\;\;\; \chi |_{I_s} = 1.
\eeqn
\end{corollary}

\begin{proof}
By Lemma \ref{lemma-trichotomy} and the definition of the group $W_{\La, \chi}^0$ (Definition~\ref{defn-W0}), the Coxeter group $W_{\La, \chi}^0$ is generated by all the reflections $s \in W_\La$ such that $\chi |_{I_s} = 1$. Combining this with \cite[Exercise 1.14]{H}, we have $s \in W_{\La, \chi}^0$ if and only if there exist a reflection $t \in W_{\La, \chi}^0$ and an element $g \in W_{\La, \chi}^0$, such that $s = g \, t \, g^{-1}$ and $\chi |_{I_t} = 1$. Recall that we write ``$\; \cdot \;$'' for the action of $W_\La$ on $I$. By the definition of $I_s \subset I$, we have $I_s = g \cdot I_t$. The corollary now follows from Lemma~\ref{lemma-W0}.
\end{proof}

Recall the regular splitting homomorphism $\tilde r : B_{W_\La} \to \widetilde B_{W_\La}$ which was fixed in \eqref{eqn-tilde-r}, and let $r = \tilde p \circ \tilde r : B_{W_\La} \to \widetilde W_\La$.

\begin{corollary}\label{cor-char-and-lifts}
Let $s \in S$ and $\chi \in \hat I$. Then we have:
\beqn
\chi( r (\sigma_s^2)) = 1 \;\;\;\; \text{if and only if} \;\;\;\; s \in W_{\La, \chi}^0 \, .
\eeqn
\end{corollary}

\begin{proof}
This follows from Definitions \ref{defn-regular-splitting} and \ref{defn-W0}, Lemma~\ref{lemma-trichotomy}, and Corollary~\ref{cor-sInW0}.
\end{proof}

\subsection{Regularity of the Kostant-Rallis splitting}
\label{subsec-regularity-of-kr}
In this subsection, we give a proof of Proposition \ref{prop-kr-is-regular}. We will make use of the following construction. Pick a reflection $s \in S$. Let $\baseptas \in \La_s$ be the orthogonal projection of the basepoint $\basepta \in \La_\bR^+$ onto the hyperplane $\La_s \subset \La$, taken with respect to the bilinear form $\nu_\Lg$. Let $\barbaseptas = f (\baseptas) \in \La / W_\La$, and let $X_{\barbaseptas} = f^{-1} (\barbaseptas)$. Let $\Lp^{reg} \subset \Lp$ be the set of regular points for $f$, and let $X_{\barbaseptas}^{reg} = X_{\barbaseptas} \cap \Lp^{reg}$. Pick a point $y_s \in X_{\barbaseptas}^{reg}$. Let $N \subset \Lp^{reg}$ be a normal slice to $X_{\barbaseptas}^{reg}$ through $y_s$. By this we mean that $N$ is a locally closed smooth submanifold of $\Lp^{reg}$, with $\dim_\bR N = \dim_\bR \La$, which is transverse to $X_{\barbaseptas}^{reg}$ at $y_s$.

Before proceeding further, we briefly indicate the idea of the proof of Proposition \ref{prop-kr-is-regular}. Consider a small free loop $\gamma_3$ (this notation will be made clear shortly) in the normal slice $N$, linking the hypersurface $f^{-1} (f (\La_s)) \subset \Lp$ in the counter-clockwise direction. Let us write $\widetilde B_{W_\La} / I^0$ for the set of $I^0$-orbits in $B_{W_\La}$ under the conjugation action. The loop $\gamma_3$ represents a canonical element of $\widetilde B_{W_\La} / I^0$; namely the image of $\tilde r [y, \Ls, \Gamma] \, (\sigma_s) \in \widetilde B_{W_\La}$. Moreover, this element depends on the pair $[y_s, N]$ only through the connected component of $y_s \in X_{\barbaseptas}^{reg}$. This enables us, when computing the image of $\tilde r [y, \Ls, \Gamma] \, (\sigma_s)$ in $\widetilde B_{W_\La} / I^0$, to replace the Kostant-Rallis slice $\Ls \subset \Lp$ by a normal slice $N$ to $X_{\barbaseptas}^{reg}$ which is contained in the subspace $\Lp_s = Z_\Lp (\La_s)$ of Section~\ref{subsec-rank-one}. The proposition follows, once we note that conjugation by elements of $I^0$ preserves the subgroup $\widetilde W_s \subset \widetilde W_\La$.

To make the above paragraph precise, recall the Hermitian structure $\langle \; , \, \rangle$ on $\Lp$ introduced in Section~\ref{subsec-structure-of-Wa}. We will measure distances with respect to this Hermitian structure. For each $\epsilon > 0$, let $\on{B} [\baseptas, \epsilon] \subset \La$ be the open $\epsilon$-ball around $\baseptas \in \La_s$, and let $\on{B} [\barbaseptas, \epsilon] = f (\on{B} [\baseptas, \epsilon]) \subset \La / W_\La$. We will say that $\epsilon > 0$ is sufficiently small for the normal slice $N \subset \Lp^{reg}$ through $y_s$ if there exists a smooth map $g_N : \on{B} [\barbaseptas, 2 \epsilon] \to N$, such that $g_N (\barbaseptas) = y_s$ and $f \circ g_N = \on{Id}_{\on{B} [\barbaseptas, 2 \epsilon]}$. Note that, in this case, the map $g_N$ is determined uniquely by the triple $[y_s, N, \epsilon]$.

Fix an $\epsilon \in (0, \on{dist} (\basepta, \baseptas)]$ which is sufficiently small for $N$. Let $\genpta = \genpta [\baseptas, \epsilon] \in \La_\bR^+$ be the unique point of the segment $[\basepta, \baseptas]$ satisfying \linebreak $\on{dist} (\genpta, \baseptas) = \epsilon$. Let $\bargenpta = f (\genpta)$ and $X_{\bargenpta} = f^{-1} (\bargenpta)$. Next, pick a continuous path $\Gamma : [0, 1] \to X_{\bargenpta}$, with $\Gamma (0) = \genpta$ and $\Gamma (1) = g_N (\bargenpta)$. The quadruple $[y_s, N, \epsilon, \Gamma]$ determines an element $\gamma = \gamma [y_s, N, \epsilon, \Gamma] \in \widetilde B_{W_\La} = \pi_1^K (\Lp^{rs}, \basepta)$ as follows.  The element $\gamma$ is given by composing five paths $\gamma_1, \dots, \gamma_5 : [0, 1] \to \Lp^{rs}$, where:
\begin{enumerate}[topsep=-1.5ex]
\item[$\bullet$]     $\gamma_1 : [0, 1] \to \La_\bR^+$ is the straight line path from $\basepta$ to $\genpta$;

\item[$\bullet$]     $\gamma_2 = \Gamma : [0, 1] \to X_{\bargenpta}$;

\item[$\bullet$]     $\gamma_3 : [0, 1] \to N$ is given by $\gamma_3 (t) = (g_N \circ f) (\baseptas + \exp (t \pi \mathbf{i}) \cdot (\genpta - \baseptas))$;

\item[$\bullet$]     $\gamma_4 : [0, 1] \to X_{\bargenpta}$ is given by $\gamma_4 (t) = \gamma_2 (1 - t)$;

\item[$\bullet$]     $\gamma_5 : [0, 1] \to \La_\bR^+$ is given by $\gamma_5 (t) = \gamma_1 (1 - t)$.
\end{enumerate}
Let $\Lp^{reg}_s = \Lp^{reg} \cap f^{-1} (f (\La_s))$. Then $\Lp^{reg}_s$ is a hypersurface in $\Lp^{reg}$, and $y_s \in \Lp^{reg}_s$ is a smooth point of this hypersurface. The path $\gamma_3$ is just a closed loop in $\Lp^{rs} \subset \Lp^{reg} - \Lp^{reg}_s$, which links the hypersurface $\Lp^{reg}_s$ near the point $y_s$. The paths $\gamma_1, \gamma_2, \gamma_4, \gamma_5$ are needed to turn $\gamma_3$ into an element of $\pi_1^K (\Lp^{rs}, \basepta)$.

The following lemma clarifies the dependence of $\gamma [y_s, N, \epsilon, \Gamma] \in \widetilde B_{W_\La}$ on the path $\Gamma$.

\begin{lemma}\label{lemma-conjByI}
In the above construction, suppose $\Gamma_1, \Gamma_2 : [0, 1] \to X_{\bargenpta}$ are two continuous paths from $\genpta$ to $g_N (\bargenpta)$. Then there exists an element $u \in I^0$ such that:
\beqn
\gamma [y_s, N, \epsilon, \Gamma_2] = u \; \gamma [y_s, N, \epsilon, \Gamma_1] \; u^{-1} \, .
\eeqn
\end{lemma}

\begin{proof}
We have a natural identification $I^0 = \pi^{K^0}_1 (X_{\barbasepta}, \basepta) = \pi_1^{K^0} (X_{\bargenpta}, \genpta)$. The element $u \in I^0$ is obtained by composing the paths $\Gamma_1$ and $\Gamma_2^{-1}$ in a suitable order.
\end{proof}

By Lemma~\ref{lemma-conjByI}, we can unambiguously omit the path $\Gamma$ from the notation, to obtain an element $\gamma [y_s, N, \epsilon] \in \widetilde B_{W_\La} / I^0$.

\begin{proof}[Proof of Proposition~\ref{prop-kr-is-regular}]
We fix a triple $[y, \Ls, \Gamma]$ as in equation~\eqref{krsplitting}. We also fix a reflection $s \in S$. Let $y_{s, 0} \in X_{\barbaseptas}$ be the unique point of the intersection $\Ls \cap X_{\barbaseptas}$. Let $N_0 = \Ls$, i.e., let us view the Kostant-Rallis slice $\Ls \subset \Lp^{reg}$ as a normal slice to $X_{\barbaseptas}^{reg}$ through $y_{s, 0}$. Take $\epsilon_0 = \on{dist} (\basepta, \baseptas)$. Note that $\epsilon_0$ (and, in fact, every $\epsilon > 0$) is sufficiently small for $N_0$. By examining the definitions, we have:
\beqn
\tilde r [y, \Ls, \Gamma] \, (\sigma_s) = \gamma [y_{s, 0}, N_0, \epsilon_0, \Gamma].
\eeqn
Thus, we need to show that:
\beq\label{gammaInWs1}
\tilde p (\gamma [y_{s, 0}, N_0, \epsilon_0, \Gamma]) \in \widetilde W_s \, . 
\eeq

The map $\tilde p : \widetilde B_{W_\La} \to \widetilde W_\La$ of diagram~\eqref{mainDiagram} induces a map of sets of $I^0$-orbits under the conjugation action, which we also denote by $\tilde p :  \widetilde B_{W_\La} / I^0 \to \widetilde W_\La / I^0$. Moreover, it follows from the definitions in Section \ref{subsec-regular-splitting} that the subgroup $\widetilde W_s \subset \widetilde W_{\La}$ is invariant under the conjugation action of $I^0$. Thus, we can restate equation~\eqref{gammaInWs1} as follows:
\beq\label{gammaInWs2}
\tilde p (\gamma [y_{s, 0}, N_0, \epsilon_0]) \in \widetilde W_s / I^0 \, . 
\eeq

The idea of the proof of equation~\eqref{gammaInWs2} is to exploit the fact that the element $\gamma [y_s, N, \epsilon] \in \widetilde B_{W_\La} / I^0$ remains unchanged as we vary the arguments $[y_s, N, \epsilon]$ continuously. Since the choice of $\epsilon$ in the construction of $\gamma [y_s, N, \epsilon]$ comes from a path-connected set, namely an interval in $\bR$, we can conclude that $\gamma [y_s, N, \epsilon]$ is independent of $\epsilon$. Thus, we can unambiguously write $\gamma [y_s, N, \epsilon] = \gamma [y_s, N]$. Likewise, since any two normal slices to $X_{\barbaseptas}^{reg} \subset \Lp^{reg}$ through $y_s$ can be included into a smooth family of such normal slices over the interval $[0, 1]$, we can conclude that $\gamma [y_s, N] = \gamma [y_s]$ is independent of $N$. Furthermore, the element $\gamma [y_s] \in \widetilde B_{W_\La} / I^0$ only depends on the connected component of $y_s$ in $X_{\barbaseptas}^{reg}$. We will now make use of the above observations to replace the pair $[y_{s, 0}, N_0]$ by a pair $[y_{s, 1}, N_1]$, such that $\gamma [y_{s, 0}, N_0] = \gamma [y_{s, 1}, N_1]$, but the RHS is easier to relate to the subgroup $\widetilde W_s \subset \widetilde W_\La$.

Recall that we let $\Lp_s = Z_\Lp (\La_s)$ and write $f_s : \Lp_s \to \La / W_s$ for the quotient map (see diagram~\eqref{diagramFs}). Consider the intersection $X_{\barbaseptas} \cap \Lp_s$. In view of diagram~\eqref{diagramFs} and equation~\eqref{eqn-Fs-product}, the connected component of this intersection that passes through $\baseptas$ can be described as $f_s^{-1} (f_s (\baseptas)) = \baseptas + \bar f_s^{-1} (0)$. It follows from the restricted root space decomposition~\eqref{eqn-restricted-root-decomposition} that we have:
\beq\label{eqn-normal-slice}
\Lp = [\Lk, \baseptas] \oplus \Lp_s \, .
\eeq
Thus, the subspace $\Lp_s \subset \Lp$ is a normal slice through $\baseptas$ to the closed orbit $K \cdot \baseptas = K^0 \cdot \baseptas \subset \Lp$. Note that $K^0 \cdot \baseptas$ is the unique closed $K^0$-orbit in $X_{\barbaseptas}$. It follows that every $K^0$-orbit in $X_{\barbaseptas}$ meets the connected component $\baseptas + \bar f_s^{-1} (0)$ of $X_{\barbaseptas} \cap \Lp_s$. Thus, we can pick a point $y_{s, 1} \in K^0 \cdot y_{s, 0} \cap (\baseptas + \bar f_s^{-1} (0))$. By construction, the points $y_{s, 0}$ and $y_{s, 1}$ lie in the same connected component of $X_{\barbaseptas}^{reg}$.

Let $N_1 \subset \Lp_s$ be a normal slice through $y_{s, 1}$ to the orbit $K_s^0 \cdot y_{s, 1}$. Note that $f_s (\baseptas) \in \La / W_s$ is a regular point for the projection map $\La / W_s \to \La / W_\La$. It follows by examining the right square of diagram~\eqref{diagramFs} that $N_1$ is also a normal slice through $y_{s, 1}$ to $X_{\barbaseptas}^{reg} \subset \Lp^{reg}$. By construction, the element $\gamma [y_{s, 1}, N_1] \in \widetilde B_{W_\La} / I^0$ can be represented by a closed loop in $\Lp^{rs} \cap \Lp_s$. It follows that:
\beqn
\tilde p (\gamma [y_{s, 0}, N_0]) = \tilde p (\gamma [y_{s, 1}, N_1]) \in \widetilde W_s / I^0 \, ,
\eeqn
as required.
\end{proof}

\subsection{A classification of regular splittings}
\label{subsec-splitting}
In this subsection, we give a proof of Lemma~\ref{lemma-braid-relations}. To that end, we recall a standard construction. Following \cite[Lemma 8.1.4]{S}, for each $\alpha \in \Phi$, let us choose an isomorphism:
\beqn
u_\alpha : \bG_a \to U_\alpha \subset G, 
\eeqn
 so that for each $\alpha \in \Phi$, we have:
\begin{subequations}
\beq
t \, u_\alpha (x) \, t^{-1} = u_\alpha (\alpha (t) \, x) \text{ for all } t \in T \text{ and } x \in \bC;
\eeq
\beq\label{eqn-nalpha1}
n_\alpha \coloneqq u_\alpha (1) \, u_{-\alpha} (-1) \, u_\alpha (1) \in N_G (T) \text{ and has image } s_\alpha \text{ in } W;
\eeq
\beq\label{eqn-nalpha2}
u_\alpha (a) \, u_{-\alpha} (-1/a) \, u_\alpha (a) = \check\alpha (a) \, n_\alpha \text{ for all } a \in \bC^*;
\eeq
\end{subequations}
where $s_\alpha \in W$ is the reflection corresponding to $\alpha$ and $\check\alpha \in X_* (T)$ is the coroot corresponding to $\alpha$. We have:
\beq\label{nalpha}
n_\alpha^2 = \check\alpha (-1), \;\; n_{-\alpha} = n_\alpha^{-1} \, , \, \text{ and}
\eeq
\beq\label{eqn1}
\begin{gathered}
\text{for each $u \in U_\alpha - \{ 1 \}$, there is a unique $u' \in U_{-\alpha} - \{ 1 \}$,} \\\text{such that $u \, u' \, u \in N_G (T)$.}
\end{gathered}
\eeq

Note that for each $\alpha \in \Phi$, there exists a $c_\alpha \in \bC^*$ such that: 
\beq\label{eqn2}
\theta (u_\alpha (x)) = u_{\theta \alpha} (c_\alpha \, x) \text{ for all } x \in \bC.
\eeq
By combining equations \eqref{eqn-nalpha2} and \eqref{eqn1}, and by applying $\theta$ to both sides of equation~\eqref{eqn2}, we see that:
\beqn
c_{-\alpha} = c_{\theta \alpha} = 1 / c_\alpha \, .
\eeqn
By letting $\tilde u_\alpha (x) = u_\alpha (x / \sqrt{c_\alpha})$ (and choosing the square roots so that $\sqrt{c_\alpha} \, \sqrt{c_{\theta \alpha}} = 1$), we can and will assume that the isomorphisms $\{ u_\alpha \}$ are chosen so that:
\beq\label{choice-of-constants}
\theta (u_\alpha (x)) = u_{\theta \alpha} (x).
\eeq
Let us write $X_\alpha = d_0 u_\alpha (1) \in \Lg_\alpha$. Then we have $\theta (X_\alpha) = X_{\theta \alpha}$.

Let us also record that:
\beq\label{conjugation-relation}
n_\alpha \, \check\beta (a) \, n_\alpha^{-1} = s_\alpha \, \check\beta(a), \text{ for all } \alpha, \beta \in \Phi \text{ and } a \in \bC^*.
\eeq
This follows directly from equation~\eqref{eqn-nalpha1}.

\begin{proof}[Proof of Lemma~\ref{lemma-braid-relations}]
Without loss of generality, we can assume that $G$ is semisimple. As discussed at the end of Section~\ref{subsec-structure-of-Wa}, we can order the set $S$ of simple reflections, writing $S = \{ s_1, \ldots, s_\nS \}$, so that $s_i = s_{\bar\alpha_i}$ where $\alpha_i \in \Phi_\rr$, for $i = 1, \ldots, m$, and $s_i = s_{\bar\beta_i} = s^{}_{\beta_i} s^{}_{\theta\beta_i} |_\La$ where $\beta_i \in \Phi_\rc^1$ (see equations~\eqref{roots-in-phic} and \eqref{restriction-of-complex-roots}), for $i = m+1, \ldots, \nS$. Here we use the convention that $m = 0$ if $\Phi_\rr = \emptyset$ and $m = \nS$ if $\Phi_\rc = \emptyset$. For $\bar\alpha \in \Sigma \subset X^* (A)$, let us write $\check{\bar\alpha} \in X_* (A)$ for the corresponding coroot. We can and will assume that:
\beq\label{angles}
\langle \bar\alpha_i, \check{\bar\alpha}_j \rangle \leq 0, \ 
\langle \bar\beta_i, \check{\bar\beta}_j \rangle \leq 0, \
\langle \bar\alpha_i, \check{\bar\beta}_j \rangle \leq 0, \ 
\langle \bar\beta_i, \check{\bar\alpha}_j \rangle\leq 0.
\eeq

For each $\alpha \in \Phi$, let $n_\alpha \in N_G (T)$ be as in equation~\eqref{eqn-nalpha1}. We define:
\beq\label{defn-of-bar-n}
\begin{gathered}
n_{\bar\alpha} = \check{\alpha} (\mathbf{i}) \, n_\alpha \text{ or } \check{\alpha} (-\mathbf{i}) \, n_\alpha \text{ if } \alpha \in \Phi_\rr \, , \\ 
n_{\bar\beta} = \check{\beta} (\mathbf{i}) \, n_\beta \, (-\theta \check\beta (\mathbf{i})) \, n_{-\theta \beta} = \check{\beta} (\mathbf{i}) \, \theta \check\beta (-\mathbf{i}) \, n_\beta \, n_{-\theta\beta} \text{ if } \beta \in \Phi_\rc^1 \, .
\end{gathered}
\eeq
We claim that $n_{\bar\alpha}, n_{\bar\beta} \in N_K (\La) \cap Z_K (\La_s)^0$. To see this, first note that, by the definition of $\Phi_\rc \subset \Phi$ and by equation~\eqref{roots-in-phic}, we have $\beta \pm \theta \beta' \notin \Phi$ for all $\beta, \beta' \in \Phi_\rc^1$. More precisely, if we had $\gamma = \beta \pm \beta' \in \Phi$, then we must have $\gamma \in \Phi_\rc$, but $\gamma$ is not orthogonal to either $\beta$ or $\beta'$, which contradicts equation~\eqref{roots-in-phic}. It follows that:
\beq\label{beta-commutations}
n_{\beta} \, n_{-\theta \beta'} = n_{-\theta \beta'} \, n_{\beta} \, \text{ for } \, \beta, \beta' \in \Phi_\rc^1 \, .
\eeq
Combining equation~\eqref{beta-commutations} with equations \eqref{eqn-nalpha1}, \eqref{nalpha},  \eqref{choice-of-constants}, and \eqref{conjugation-relation}, we find that: \beqn
\begin{gathered}
\text{$\theta (n_\alpha) = \check\alpha (-1) \, n_\alpha$ for $\alpha \in \Phi_\rr \,$,} \\
\text{$\theta (n_\beta \, n_{-\theta \beta}) = \check\beta (-1) \, \theta \check\beta (-1) \, n_\beta \, n_{-\theta \beta}$ for $\beta \in \Phi_\rc \,$.}
\end{gathered}
\eeqn
Based on this, one can readily check that $n_{\bar\alpha}, n_{\bar\beta} \in K$. It follows that $n_{\bar\alpha}, n_{\bar\beta} \in N_K(\La)$. It remains to show that $n_{\bar\alpha}, n_{\bar\beta}$ belong to the identity component $Z_K(\La_s)^0$ of $Z_K(\La_s)$. This follows from the fact that the entire construction of the elements $n_{\bar\alpha}, n_{\bar\beta} \in G$ can be lifted to the simply-connected cover $G_{sc}$ of $G$, where the corresponding $\theta$-fixed point group $K_{sc} = G_{sc}^\theta \subset G_{sc}$ is connected. By~\eqref{conjugation-relation}, we have $n_{\bar\alpha}^2 = \check\alpha (-1)$. It follows that the two expressions for $n_{\bar\alpha}$ in~\eqref{defn-of-bar-n} are the inverses of each other. Moreover, we have:
\beq\label{the-groups-Ws-and-bar-n}
\widetilde W_s = \langle n_{\bar\alpha} \rangle \text{ if } s = s_{\bar\alpha} \, , \; \alpha \in \Phi_\rr \, , \;\; \widetilde W_s = \{1, n_{\bar\beta} \} \text{ if } s = s_{\bar\beta} \, , \; \beta \in \Phi_\rc^1 \, .
\eeq

We now show that the elements $\{ n_{\bar\alpha_i}, \, n_{\bar\beta_j} \}_{i, j} \subset K$ pairwise satisfy braid relations (as elements of $K$). To begin with, using~\eqref{restricted-roots} and~\eqref{roots-in-phic}, one can readily check that: 
\begin{subequations}
\beq\label{pairing-1}
\langle \bar\alpha_i, \check{\bar\alpha}_j \rangle = \langle \alpha_i, \check\alpha_j \rangle,
\eeq
\beq\label{pairing-2}
\langle \bar\beta_i, \check{\bar\beta}_j \rangle = \langle \beta_i, \check\beta_j \rangle,
\eeq
\beq\label{pairing-3}
\langle \bar\alpha_i, \check{\bar\beta}_j \rangle = 2 \, \langle \alpha_i, \check\beta_j \rangle
\;\; \text{and} \;\;
\langle \bar\beta_j, \check{\bar\alpha}_i \rangle = \langle \beta_j, \check\alpha_i \rangle.
\eeq
\end{subequations}
Equation~\eqref{pairing-3} implies that:
\beq\label{pairing-4}
\begin{gathered}
\text{either} \;\; \langle \alpha_i, \check\beta_j \rangle = \langle \beta_j, \check\alpha_i \rangle = 0 \;\; \text{and} \;\; \langle \alpha_i, -\theta \check\beta_j \rangle = \langle -\theta \beta_j, \check\alpha_i \rangle = 0, \\
\text{or} \;\; \langle \alpha_i, \check\beta_j \rangle = \langle \beta_j, \check\alpha_i \rangle = -1 \;\; \text{and} \;\; \langle \alpha_i, -\theta \check\beta_j \rangle = \langle -\theta \beta_j, \check\alpha_i \rangle = -1.
\end{gathered}
\eeq

In view of~\eqref{angles},~\eqref{beta-commutations},~\eqref{pairing-1}--\eqref{pairing-3}, and~\eqref{pairing-4}, by applying \cite[Proposition 9.3.2]{S}, we can conclude that the elements $\{ n_{\alpha_i}, \, n_{\beta_j}, \, n_{-\theta \beta_j} \}_{i, j} \subset G$ pairwise satisfy braid relations.

The braid relations for the $\{ n_{\bar\alpha_i}, \, n_{\bar\beta_j} \}_{i, j}$ can now be verified by a direct computation using the definitions in~\eqref{defn-of-bar-n}, equation~\eqref{conjugation-relation}, and the braid relations for the $\{ n_{\alpha_i}, \, n_{\beta_j}, \, n_{-\theta \beta_j} \}_{i, j}$. More precisely, one can show that the pair $n_{\bar\alpha_i}, \, n_{\bar\alpha_j}$ (resp. $n_{\bar\beta_i}, \, n_{\bar\beta_j}$) satisfies the same braid relation as the pair $n_{\alpha_i}, \, n_{\alpha_j}$ (resp. $n_{\beta_i}, \, n_{\beta_j}$). In the former case of~\eqref{pairing-4}, the pair $n_{\bar\alpha_i}, \, n_{\bar\beta_j}$ satisfies the same braid relation as the pair $n_{\alpha_i}, \, n_{\beta_j}$, i.e., $n_{\bar\alpha_i} \, n_{\bar\beta_j} = n_{\bar\beta_j} \, n_{\bar\alpha_i}$. In the latter case of~\eqref{pairing-4}, the pair $n_{\bar\alpha_i}, \, n_{\bar\beta_j}$ satisfies the braid relation $n_{\bar\alpha_i} \, n_{\bar\beta_j} \, n_{\bar\alpha_i} \, n_{\bar\beta_j} = n_{\bar\beta_j} \, n_{\bar\alpha_i} \, n_{\bar\beta_j} \, n_{\bar\alpha_i}$, while the pair $n_{\alpha_i}, \, n_{\beta_j}$ satisfies the braid relation $n_{\alpha_i} \, n_{\beta_j} \, n_{\alpha_i} = n_{\beta_j} \, n_{\alpha_i} \, n_{\beta_j}$. We spell out this last case in more detail. Let us write $\alpha = \alpha_i$, $\beta = \beta_j$ for brevity, and let us choose $n_{\bar\alpha} = \check{\alpha} (\mathbf{i})$ in~\eqref{defn-of-bar-n}, the alternative being similar. By repeatedly applying~\eqref{conjugation-relation}, we find:
\beqn
n_{\bar\alpha} \, n_{\bar\beta} = \check\alpha (-\mathbf{i}) \, \check\beta (\mathbf{i}) \, \theta \check\beta (-\mathbf{i}) \, n_\alpha \, n_\beta \, n_{-\theta \beta} \, ,
\eeqn
\beqn
n_{\bar\beta} \, n_{\bar\alpha} = \check\alpha (\mathbf{i}) \, \check\beta (-1) \, \theta \check\beta (-1) \, n_\beta \, n_{-\theta \beta} \, n_\alpha \, ,
\eeqn
\beqn
n_{\bar\alpha} \, n_{\bar\beta} \, n_{\bar\alpha} \, n_{\bar\beta} = \check\beta (-\mathbf{i}) \, \theta \check\beta (\mathbf{i}) \, n_\alpha \, n_\beta \, n_{-\theta \beta} \, n_\alpha \, n_\beta \, n_{-\theta \beta} \, ,
\eeqn
\beqn
n_{\bar\beta} \, n_{\bar\alpha} \, n_{\bar\beta} \, n_{\bar\alpha} = \check\beta (-\mathbf{i}) \, \theta \check\beta (\mathbf{i}) \, n_\beta \, n_{-\theta \beta} \, n_\alpha \, n_\beta \, n_{-\theta \beta} \, n_\alpha \, .
\eeqn
The desired braid relation now follows from the braid relations:
\beqn 
n_\alpha \, n_\beta \, n_\alpha = n_\beta \, n_\alpha \, n_{\beta} \, , \;\; n_\alpha \, n_{-\theta \beta} \, n_\alpha = n_{-\theta \beta} \, n_\alpha \, n_{-\theta \beta} \, , \;\; n_\beta \, n_{-\theta \beta} = n_{-\theta \beta} \, n_\beta \, .
\eeqn
The lemma follows from equation~\eqref{the-groups-Ws-and-bar-n} and the above discussion.
\end{proof}

\subsection{Thom's \texorpdfstring{$A_{F_2}$}{Lg} condition}
\label{subsec-proof-of-asubf2}
In this subsection, we give a proof of Proposition \ref{prop-asubf2}. Fix an orbit $\cO \in \cN_\Lp / K$, and write $\cZ_\cO = \cO \times (\La^{rs} / W_\La) \times \{ 0 \} \subset \cZ_0$ for short. We recall the statement of Thom's $A_{F_2}$ condition for the pair $(\cZ^{rs}, \cZ_\cO)$. For every $z \in \cZ$, let $\bar z = F_2 (z)$ and $\cZ_{\bar z} = F_2^{-1} (\bar z)$. 

\begin{definition}\label{defn-asubf2}
We say that the pair $(\cZ^{rs}, \cZ_\cO)$ satisfies Thom's $A_{F_2}$ condition, if for every sequence
$z_1, z_2, \ldots \in \cZ^{rs}$ with $\, \lim z_i = z \in \cZ_\cO$, we have:
\beqn
\text{if there exists a limit} \;\; \Delta = \lim T_{z_i} \cZ_{\bar z_i} \subset T_z \cZ, \; \text{then} \;\; \Delta \supset T_z \cZ_\cO \, .
\eeqn
\end{definition}

The symbol $T_z \cZ$ in the above definition denotes the Zariski tangent space. It is clear that the $A_{F_2}$ condition is local in the parameter space $\La^{rs} / W_\La$.  It will be convenient to consider the pullback of diagram~\eqref{familyF} via the covering map $\La^{rs} \to \La^{rs} / W_\La$. More precisely, let:
\beqn
\tcZ = \cZ \times_{(\La^{rs} / W_\La)} \La^{rs} = \{ (x, \genpta, c) \in \Lp \times \La^{rs} \times \bC \; | \; f (x) = c \, f (\genpta) \}.
\eeqn
Let $\tF_2 : \tcZ \to \bC$ be the map $(x, \genpta, c) \mapsto c$, let $\tcZ_0 = \tF_2^{-1} (0)$, and let $\tcZ^{rs} = \tcZ - \tcZ^{0}$. By analogy with the stratification $\cS_\cZ$ of $\cZ$, we have a stratification $\cS_\tcZ$ of $\tcZ$ given by:
\beqn
\tcZ = \tcZ^{rs} \, \cup \, \bigcup_{\cO \in \cN_\Lp / K} \tcZ_\cO \, ,
\eeqn
where $\tcZ_\cO = \cO \times \La^{rs} \times \{ 0 \}$. Proposition \ref{prop-asubf2} is then equivalent to the following.

\begin{prop}\label{prop-asubtf2}
For the every $\cO \in \cN_\Lp / K$, the pair of strata $(\tcZ^{rs}, \tcZ_\cO)$ satisfies Thom's $A_{\tF_2}$ condition.
\end{prop}

The statement of the $A_{\tF_2}$ condition is an obvious paraphrase of Definition \ref{defn-asubf2}. In order to prove Proposition \ref{prop-asubtf2}, we restate it in metric terms. Let $\tM = \Lp \times \La^{rs} \times \bC$. Recall the Hermitian metric $\langle \; , \, \rangle$ on $\Lp$ introduced in Section~\ref{subsec-structure-of-Wa}. By restriction, it gives a Hermitian metric on $\La^{rs}$. We combine these Hermitian metrics on $\Lp$ and $\La^{rs}$ with the standard Hermitian metric on $\bC$, to obtain a flat Hermitian product metric on $\tM$. We will use $\langle \; , \, \rangle$ to denote this product metric, or one of its three constituent components, depending on the context. All norms and angles in this subsection will be taken with respect to $\langle \; , \, \rangle$. By analogy with the map $F_2 : \cZ \to \bC$, for every $z \in \tcZ$, we let $\bar z = \tF_2 (z)$ and $\tcZ_{\bar z} = \tF_2^{-1} (\bar z)$. Fix a point $z = (y, a, 0) \in \tcZ_\cO$, where $y \in \cO$ and $a \in \La^{rs}$. The following is a restatement of Proposition~\ref{prop-asubtf2}.

\begin{prop}\label{prop-asubtf2-metric}
For every non-zero vector $v \in T_z \tcZ_\cO \subset T_z \tM$ and every $\epsilon > 0$, there exists an open neighborhood $U \subset \tM$ of $z$, such that for every $z_1 \in U \cap \tcZ^{rs}$, there exists a non-zero vector $v_1 \in T_{z_1} \tcZ_{\bar z_1} \subset T_{z_1} \tM$, such that we have: $\angle (v, v_1) < \epsilon$.
\end{prop}

Note that the angle $\angle (v, v_1)$ is well defined, because the metric $\langle \; , \, \rangle$ on $\tM$ is flat. Proposition~\ref{prop-asubtf2} follows immediately from Proposition~\ref{prop-asubtf2-metric}, by inspecting the definition of the $A_{\tF_2}$ condition. We now proceed to prove Proposition~\ref{prop-asubtf2-metric}. The Hermitian metric $\langle \; , \, \rangle$ on $\Lp$ induces a Hermitian metric on the dual space $\Lp^*$. Moreover, if we view the restriction $\nu_\Lp = \nu_\Lg |_{\Lp}$ as a vector space isomorphism $\nu_\Lp : \Lp \to \Lp^*$, then this isomorphism is an isometry. We will need the following fact.

\begin{lemma}\label{lemma-shortest}
Every $h \in \nu_\Lp (\La) \subset \Lp^*$ is the shortest point of its $K$-orbit.
\end{lemma}

\begin{proof}
This follows from \cite[Theorem 0.1]{KN}.
\end{proof}

\begin{proof}[Proof of Proposition~\ref{prop-asubtf2-metric}]
Note that we have $T_z \tcZ_\cO \cong T_y \cO \oplus T_a \La^{rs}$. Therefore, we can write $v = (v_y, v_a)$, where $v_y \in T_y \cO$ and $v_a \in T_a \La^{rs}$. By linearity, it suffices to consider the cases $v = (v_y, 0)$ and $v = (0, v_a)$. The first of these two cases follows immediately from the fact that $T_y \cO \subset T_z \tcZ_\cO$ is the tangent space to the $K$-orbit $K \cdot z \subset \tcZ$, and the $K$-action on $\tcZ$ preserves the function $\tF_2 : \tcZ \to \bC$.

Let us now tackle the case $v = (0, v_a)$. In this case, we take:
\beqn
U = U [v, \epsilon] = \{ z_1 \in \tM \; | \; |\tF_2 (z_1)| < \epsilon \}.
\eeqn
Pick a point $z_1= (x_1, a_1, c_1) \in U \cap \tcZ^{rs}$, where $x_1 \in \Lp^{rs}$, $a_1 \in \La^{rs}$, and $c_1 = \tF_2 (z_1) \in \bC$.
We let $v_1 = (v_{x_1}, v_a, 0) \in T_{z_1} \tcZ_{\bar z_1}$, where we view $v_a \in T_a \La^{rs} = T_a \La$ as an element of $T_{a_1} \La^{rs}$ and $v_{x_1} \in T_{x_1} \Lp$ remains to be constructed.  

Note that, by construction, we have $\bar x_1 = f (x_1) = c_1 \, f (a_1)$. Let:
\beqn
X_{\bar x_1} = f^{-1} (\bar x_1) \subset \Lp \;\; \text{and} \;\; a_2 = c_1 \, a_1 \in X_{\bar x_1} \cap \La.
\eeqn
Note that we have:
\beqn
T_{a_2} \La = (T_{a_2} X_{\bar x_1})^\p \subset T_{a_2} \Lp,
\eeqn
and define:
\beqn
N_{x_1} \coloneqq (T_{x_1} X_{\bar x_1})^\p \subset T_{x_1} \Lp,
\eeqn
where the orthogonal complements are taken with respect to $\langle \; , \, \rangle$. Define an isomorphism:
\beqn
\eta = ((d_{x_1} f) |_{N_{x_1}})^{-1} \, \circ \, d_{a_2} f \, : \, T_{a_2} \La \to N_{x_1} \, .
\eeqn

We set $v_{x_1} = \eta (c_1 \, v_a)$ and $v_1 = (v_{x_1}, v_a, 0) \in T_{z_1} \cZ_{\bar z_1}$. We then have:
\beqn
\angle (v, v_1) < \tan (\angle (v, v_1)) = \frac{\left\lVert v_{x_1} \right\rVert}{\left\lVert v_a \right\rVert} = | c_1 | \cdot \frac{\left\lVert \eta (v_a) \right\rVert}{\left\lVert v_a \right\rVert} \, .
\eeqn
By the choice of the open set $U$, we have $| c_1 | < \epsilon$. Therefore, to show that $\angle (v, v_1) < \epsilon$ it will suffice to show that the operator norm $\left\lVert \eta \right\rVert$, computed using $\langle \; , \, \rangle$, satisfies:
\beq\label{normeta}
\left\lVert \eta \right\rVert \leq 1.
\eeq

To see this, consider the conormal bundle $\Lambda = T^*_{X_{\bar x_1}} \Lp \subset T^* \Lp$. Let $\Lambda_{a_2}$ and $\Lambda_{x_1}$ be the fibers of $\Lambda$ over $a_2$ and $x_1$, respectively. We view $\Lambda_{a_2}, \Lambda_{x_1}$ as linear subspaces of $\Lp^*$. By the $G$-invariance of $\nu_\Lg$, in the notation of Lemma~\ref{lemma-shortest}, we then have $\Lambda_{a_2} = \nu_\Lp (\La) \subset \Lp^*$. Further, we have isometries:
\beq\label{isometries}
(T_{a_2} \La)^* \cong \Lambda_{a_2} \;\; \mbox{and} \;\; (N_{x_1})^* \cong \Lambda_{x_1} \, .
\eeq
Let $g \in K$ be an element such that $g \, a_2 = x_1$. Let $g^* : \Lp^* \to \Lp^*$ be the vector space adjoint of the action of $g$ on $\Lp$. We have $g^* (\Lambda_{x_1}) = \Lambda_{a_2}$. Moreover, in terms of the isometries~\eqref{isometries}, we have $g^* |_{\Lambda_{x_1}} = \eta^*$. By Lemma~\ref{lemma-shortest}, we have $\left\lVert g^* |_{\Lambda_{x_1}} \right\rVert < 1$. Inequality~\eqref{normeta} follows.
\end{proof}

This completes our proof of Proposition~\ref{prop-asubf2}.

\section{Proof of Theorem \ref{thm-chi-eq-one}}
\label{sec-proof-of-thm-chi-eq-one}

Our Theorem \ref{thm-chi-eq-one} is very close to, and can be readily derived from, \cite[Theorem 6.1]{G2} (see also \cite[Section 4]{G4}). Indeed, there are only three material distinctions between these two theorems. Two of them were mentioned at the start of Section~\ref{sec-statement}. Namely, in the present paper, we work $K$-equivariantly with a possibly disconnected group $K$, and we use $\baseRing$ as the coefficient ring, while \cite{G2} uses $\bC$. And third, in Theorem \ref{thm-chi-eq-one}, microlocal monodromy acts on the algebra $\cH_1$ via left multiplications, and monodromy in the family acts via right multiplications, while this convention is reversed in \cite[Theorem 6.1]{G2}. The second distinction is not significant, as the Morse groups at the origin we will study are free over $\baseRing$. The third distinction is purely notational, reflecting the fact that, in the present paper, microlocal monodromy is the primary object of study.

In this section, we will give a proof of Theorem \ref{thm-chi-eq-one}, relying on the general geometric results of \cite{G2} and \cite{G1} on the Fourier transform of the nearby cycles, but not on the main results of \cite{G2} in \cite[Section 3]{G2} or the specific discussion of symmetric spaces in \cite[Section 6]{G2}. There are two reasons for doing this. First, our proof of Theorem \ref{thm-chi-eq-one} will prepare the ground, and introduce most of the necessary notions, for the proof of Theorem \ref{thm-main} in Section~\ref{sec-proof}. Second, the paper \cite{G2}, whose setting is the class of polar representations of Dadok and Kac (see \cite{DK}), contains a gap in the proofs of \cite[Lemmas 4.2, 4.3]{G2}. This gap has been fixed, without affecting the main results of \cite{G2}, in the recent document \cite{G4}. The fix required a substantial additional argument; see \cite[Section 3]{G4}. However, this issue is easy to deal with in the case of symmetric spaces. The substantial additional argument is only needed for general polar representations, where complex reflection groups which are not Coxeter groups arise. Thus, in this section, we will use the results of \cite[Sections 1-2]{G2} and \cite{G1}, but not the parts of \cite{G2} which depend on the additional argument of \cite[Section 3]{G4}.

\subsection{Fourier transform of the nearby cycles}
\label{subsec-fourier}
Let $\cS$ be the stratification of the nilcone $\cN_\Lp$ by $K$-orbits. For each $y \in \Lp$, let $h_y : \Lp \to \bC$ be the linear function given by $h_y : * \mapsto \nu_\Lg (*, y)$, where $\nu_\Lg$ is the symmetric bilinear form introduced in Section~\ref{subsec-structure-of-Wa}. Also, write:
\beqn
\xi_y = \on{Re} (h_y) : \Lp \to \bR \;\;\;\; \text{and} \;\;\;\; \zeta_y = \on{Im} (h_y) : \Lp \to \bR.
\eeqn

\begin{lemma}\label{lemma-gencov}
For every $y \in \Lp^{rs}$, the pair $(0, h_y) \in T^* \Lp = \Lp \times \Lp^*$ is a generic covector for the stratification $\cS$, in the sense of stratified Morse theory. 
\end{lemma}

\begin{proof}
If $\cO \subset \cN_\Lp$ is a nilpotent $K$-orbit, then we have:
\beqn
T^*_\cO \Lp = \{ (x, h_y) \in \Lp \times \Lp^* \; | \; [x, y] = 0, \; x \in \cO \}. 
\eeqn
The lemma follows, as the stratification $\cS$ is $\bC^*$-conic, and $x = 0$ is the only element of $\cN_\Lp$ which commutes with a regular semisimple element. 
\end{proof}

Recall that we write $P_1 = \Pone$. By Lemma \ref{lemma-gencov}, we can consider the Morse local system $M (\Pone)$ of the perverse sheaf $\Pone$ on the open set $\Lp^{rs}$. More precisely, we set:
\beq\label{defn-Morse-group}
M_y (\Pone) = M_{(0, h_y)} (\Pone) = \bH^0 (\cN_\Lp, \{ x \in \cN_\Lp \; | \; \xi_y (x) \leq -1 \}; \Pone).
\eeq
Note that we can use the number $-1$ in the above definition, because the sheaf $\Pone$ is constructible with respect to the $\bC^*$-conic stratification $\cS$. Note also that some authors use the opposite convention for Morse groups, writing $\xi_y (x) \geq 1$ instead of $\xi_y (x) \leq -1$. The two conventions can be identified canonically, using a choice of a $\sqrt{-1}$. Finally, note that $\bH^i (\cN_\Lp, \{ x \in \cN_\Lp \; | \; \xi_y (x) \leq -1 \}; \Pone) = 0$ for $i \neq 0$.

\begin{lemma}\label{lemma-fourier-morse}
The restriction of the perverse sheaf $\fF \Pone$ to the open set $\Lp^{rs} \subset \Lp$ is given by the Morse local system $M (\Pone)$. More precisely, we have:
\beqn
\fF \Pone |_{\Lp^{rs}} \cong M (\Pone) [\dimLp].
\eeqn
\end{lemma} 

\begin{proof}
The lemma follows by applying \cite[Proposition 3.7.12 (ii)]{KS} to a small convex open neighbourhood of a point $y \in \Lp^{rs}$, then using Lemma~\ref{lemma-gencov} and the conical property of $\cS$.
\end{proof}

By the general geometric results of \cite[Section 2]{G2} and \cite{G1}, we now have the following.

\begin{prop}\label{prop-fourier}
We have:
\beqn
\fF \Pone \cong \on{IC} (\Lp^{rs}, M (\Pone)).
\eeqn
\end{prop}

\begin{proof}
The paper \cite{G1} uses $\bC$ as the coefficient ring, but the proof of \cite[Theorem 1.1]{G1} goes through with minimal changes with coefficients in $\baseRing$. Namely, the proof of \cite[Theorem 1.1]{G1} proceeds by establishing the $\on{IC}$ bounds on the stalks of $\fF \Pone$ by a direct geometric argument, then invoking Verdier duality to establish the $\on{IC}$ bounds on the costalks. This use of Verdier duality is not available over $\baseRing$. However, the bounds on the costalks can be established by a direct argument which is entirely parallel (and in some sense dual) to that provided for the stalks. The proposition thus follows from the analog of \cite[Theorem 1.1]{G1} with coefficients in $\baseRing$, combined with Lemma~\ref{lemma-fourier-morse} and \cite[Proposition 2.17]{G2}, which verifies the main hypothesis of \cite[Theorem 1.1]{G1}.
\end{proof}

Proposition \ref{prop-fourier} indicates that the main content of Theorem \ref{thm-chi-eq-one}, specific to symmetric spaces, is in identifying the Morse local system $M (\Pone)$ and the action of the monodromy in the family on $M (\Pone)$.

\begin{remark}\label{rmk-Z-to-k}
Recall that the module $\cM$ of Theorem \ref{thm-chi-eq-one} is free over $\baseRing$. Moreover, if we write $\cM [\bZ]$ for the corresponding module over $\bZ$, then we have $\cM \cong \cM [\bZ] \otimes_\bZ \baseRing$. In view of Proposition \ref{prop-fourier}, it follows that the claim of Theorem \ref{thm-chi-eq-one} for $\baseRing = \bZ$ implies the claim for general $\baseRing$.
\end{remark}

\subsection{Picard-Lefschetz classes}
\label{subsec-PLclasses}
Recall the basepoints $\basepta, l \in \La_\bR^+$, where $\La_\bR^+$ is the Weyl chamber specified in Section~\ref{subsec-structure-of-Wa}. Let us write $d = \dim_\bC X_{\barbasepta}$. 

\begin{lemma}\label{lemma-nearby-morse}
The Morse group $M_l (\Pone)$ can be identified as follows. We have:
\beqn
M_l (\Pone) \cong H^d (X_{\barbasepta}, \{ x \in X_{\barbasepta} \; | \; \xi_l (x) \leq - \xi_0 \}; \baseRing),
\eeqn
where $\xi_0 \gg 1$ is a sufficiently large real number.
\end{lemma}

\begin{proof}
This follows from the definition of the Morse group $M_l (\Pone)$ (see equation~\eqref{defn-Morse-group}) and the definition of the nearby cycle functor $\psi_{\barbasepta}$ in Section~\ref{subsec-twisted}. See also \cite[Lemma 3.1(i)]{G3}.
\end{proof}

\begin{lemma}\label{lemma-morse-dual}
We have:
\beqn
M_l (\Pone) \cong H_d (X_{\barbasepta}, \{ x \in X_{\barbasepta} \; | \; \xi_l (x) \geq \xi_0 \}; \baseRing),
\eeqn
where $\xi_0 \gg 1$ is a sufficiently large real number.
\end{lemma}

\begin{proof}
This follows from Lemma \ref{lemma-nearby-morse}, plus Poincare-Lefschetz duality on $X_{\barbasepta}$, plus the generic property of $h_l$ (see Lemma \ref{lemma-gencov}); cf. \cite[Lemma 2.7]{G2}. We use the complex structure to orient $X_{\barbasepta}$.
\end{proof}

Let $h_{l, \barbasepta} = h_l |_{X_{\barbasepta}} : X_{\barbasepta} \to \bC$. Write $Z_l \subset X_{\barbasepta}$ for the critical locus of $h_{l, \barbasepta}$. For each $w \in W_\La$, let $e_w = w \, \basepta \in \La$.

\begin{lemma}\label{lemma-crit-points} We have:
\begin{enumerate}[topsep=-1.5ex]
\item[(i)]    $Z_l = \{ e_w \; | \; w \in W_\La \};$

\item[(ii)]   each $e_w \in Z_l$ is a Morse critical point of $h_{l, \barbasepta} \, ;$

\item[(iii)]  $\zeta_l (e_w) = 0$ for all $w \in W_\La \, .$
\end{enumerate}
\end{lemma}

\begin{proof}
As in the proof of Lemma \ref{lemma-gencov}, we have:
\beq\label{conormalX}
T^*_{X_{\barbasepta}} \Lp = \{ (x, h_y) \in \Lp \times \Lp^* \; | \; [x, y] = 0, \; x \in X_{\barbasepta} \}.
\eeq
Part (i) of the lemma follows from equation~\eqref{conormalX}. Consider the projection to the second factor $\pi_2 : T^*_{X_{\barbasepta}} \Lp \to \Lp^*$. It is not hard to check that the differential $d_{(x, h_y)} \pi_2$ is nondegenerate whenever $y \in \Lp^{rs}$. Part (ii) of the lemma follows. Part (iii) of the lemma follows from the fact that we have $l, e_w \in \La_\bR$ for every $w \in W$, and the fact that $\nu_\Lg$ is real on $\La_\bR$.
\end{proof}

Let $C_l = h_l (Z_l)$ be the set of critical values of $h_{l, \barbasepta}$. By Lemma \ref{lemma-crit-points}(iii), we have $C_l \subset \bR \subset \bC$. We can not assert that the critical values $\{ h_l (e_w) \}_{w \in W_\La}$ are all distinct. But we can claim that the order of the critical values on the real line is consistent with the Bruhat order on $W_\La$. More precisely, let $\prec$ denote the Bruhat order on $W_\La$ arising from the choice of the Weyl chamber $\La_\bR^+$, so that $1 \prec w$ for every $w \in W_\La - \{ 1 \}$. Also, let $\BruhatLength : W_\La \to \bN$ be the length function associated to the set of simple reflections $S$, so that $\BruhatLength (1) = 0$ and $\BruhatLength (s) = 1$ for every $s \in S$. Then we have the following.

\begin{lemma}\label{lemma-bruhat}
For every $w_1, w_2 \in W_\La$ with $w_1 \prec w_2$, we have $h_l (e_{w_1}) = \xi_l (e_{w_1}) > \xi_l (e_{w_2}) = h_l (e_{w_2})$.
\end{lemma}

\begin{proof}
Let us begin by recalling the definition of the Bruhat order. We have $w_1 \prec w_2$ if and only if there exists a sequence of reflections $t_1, \ldots, t_m \in W_\La$, such that $w_2 = w_1 t_1 \cdots t_m$ and $\BruhatLength (w_1 t_1 \cdots t_i) > \BruhatLength (w_1 t_1 \cdots t_{i-1})$ for every $i = 1, \ldots, m$. Thus, it suffices to prove the lemma for $w_1 \prec w_2$ with $w_2 = w_1 t$ for some reflection $t \in W_\La$. Let $w_2 = s_1 \cdots s_k$ be a reduced expression, where $s_i \in S$ (and repetitions are allowed). Then, by the Strong Exchange Condition, there exists a unique $1 \leq i \leq k$ such that $t = s_k \cdots s_{i+1} s_i s_{i+1} \cdots s_k$ and $w_1 = s_1 \cdots \hat s_i \cdots s_k$.

Let $\bar\alpha_i \in \Sigma \subset \La_\bR^*$ be a restricted root such that $s_i = s_{\bar\alpha_i}$ and $\bar\alpha_i (\genpta) > 0$ for every $\genpta \in \La_\bR^+$. Let $\check{\bar\alpha}_i \in \La_\bR$ be the corresponding coroot, so that
$s \, \genpta = \genpta - \bar\alpha_i (\genpta) \, \check{\bar\alpha}_i$ for every $\genpta \in \La_\bR$. Define $\bar\alpha = (s_k \cdots s_{i+1}) \, \bar\alpha_i$ and $\check{\bar\alpha} = (s_k \cdots s_{i+1}) \, \check{\bar\alpha}_i$. Then we have $t \, \genpta = \genpta - \bar\alpha (\genpta) \, \check{\bar\alpha}$ for every $a \in \La_\bR$. Further, we have:
\beqn
\begin{split}
\xi_l (e_{w_1}) & = \langle w_1 \, \basepta, l \rangle
                            = \langle \basepta, w_1^{-1} \, l \rangle
                            = \langle t \, \basepta, w_2^{-1} \, l \rangle \\
                        & = \langle \basepta - \bar\alpha (\basepta) \, \check{\bar\alpha}, w_2^{-1} \, l \rangle
                            = \xi_l (e_{w_2}) - \bar\alpha (\basepta) \, \langle w_2 \, \check{\bar\alpha}, l \rangle.
\end{split}
\eeqn
Note that $\bar\alpha (\basepta) = (s_k \cdots s_{i+1} \, \bar\alpha_i) (\basepta) > 0$, because the word $s_k \cdots s_i$ is reduced and $\basepta \in \La_\bR^+$. Similarly, we have $\langle w_2 \, \check{\bar\alpha}, l \rangle = \langle s_1 \cdots s_i \, \check{\bar\alpha}_i, l \rangle = - \langle s_1 \cdots s_{i-1} \, \check{\bar\alpha}_i, l \rangle < 0$, because the word $s_1 \cdots s_i$ is reduced and $l \in \La_\bR^+$. The lemma follows.
\end{proof}

Using Lemma \ref{lemma-bruhat}, we can identify how large the number $\xi_0$ in Lemma \ref{lemma-morse-dual} needs to be.

\begin{lemma}\label{lemma-large-xi}
The homology group in the RHS of the isomorphism of Lemma \ref{lemma-morse-dual} is independent of $\xi_0$ for $\xi_0 > \xi (e_1)$.
\end{lemma}

\begin{proof}
This is an application of Morse theory, using Lemmas \ref{lemma-gencov} and \ref{lemma-bruhat}.
\end{proof}

We will analyze the relative homology group $M_l (\Pone) \cong H_d (X_{\barbasepta}, \{ x \in X_{\barbasepta} \; | \; \xi_l (x) \geq \xi_0 \}; \baseRing)$ of Lemma~\ref{lemma-morse-dual} using the techniques of the classical Picard-Lefschetz theory. Thus, we now recall the standard construction of Picard-Lefschetz classes in the Morse group $M_l (\Pone)$. Fix a $\xi_0 > \xi (e_1)$, as in Lemma \ref{lemma-large-xi}. Pick a critical point $e \in Z_l$. Let $\cH [e] : T_e X_{\barbasepta} \to \bC$ be the Hessian of $h_{l, \barbasepta}$ at $e$. We view $\cH [e]$ as a complex-valued quadratic form. Note that, by Lemma \ref{lemma-crit-points} (ii), the form $\cH [e]$ is non-degenerate. Next, pick a smooth path $\gamma : [0, 1] \to \bC$ such that:
\begin{enumerate}[topsep=-1.5ex]
\item[(P1)]  $\gamma (0) = h_l (e)$;

\item[(P2)]  $\gamma' (t) \neq 0$ for all $t \in [0, 1]$, and $\gamma' (0) = \mathbf{i}$;

\item[(P3)]  $\gamma (1) = \xi_0$;

\item[(P4)]  $\gamma (t) \notin C_l$ for all $t \in (0, 1)$;

\item[(P5)]  $\gamma (t_1) \neq \gamma (t_2)$ for all $t_1, t_2 \in [0, 1]$ with $t_1 \neq t_2$.
\end{enumerate}
Recall that we write $\mathbf{i} = \sqrt{-1}$. Let $\cH_v [e] = \on{Re} (- \mathbf{i} \cdot \cH_e) : T_e X_{\barbasepta} \to \bR$; it is a non-degenerate real-valued quadratic form. Here, the subscript $v$ stands for ``vertical'', referring to the direction of $\gamma' (0)$. Recall the Hermitian structure $\langle \; , \, \rangle$ on $\Lp$ introduced in Section~\ref{subsec-structure-of-Wa}. Let $T_v [e] \subset T_e X_{\barbasepta}$ be the positive eigenspace of the real quadratic form $\cH_v [e]$ relative to $\langle \; , \, \rangle$, where by ``positive eigenspace'' we mean the direct sum of all the eigenspaces corresponding to positive eigenvalues. Note that we have $\dim_\bR T_v [e] = \dim_\bC X_{\barbasepta} = d$, since the real-valued quadratic form $\cH_v [e]$ is the real part of a complex-valued quadratic form. Finally, pick an orientation $o$ of the real vector space $T_v [e]$.

The data $[e , \gamma, o]$ defines a Picard-Lefschetz homology class:
\beqn
PL [e , \gamma, o] \in H_d (X_{\barbasepta}, \{ x \in X \; | \; \xi_l (x) \geq \xi_0 \}; \baseRing),
\eeqn
as follows. The class $PL [e , \gamma, o]$ is represented by a smoothly embedded disk:
\beqn
\kappa : (\rD^d, \partial \rD^d) \to (X_{\barbasepta}, \{ \xi_l (x) \geq \xi_0 \}),
\eeqn
such that:
\begin{enumerate}[topsep=-1.5ex]
\item[(i)]    $h_l \circ \kappa (\rD^d - \{ 0 \}) = \gamma ((0, 1]) \subset \bC$;

\item[(ii)]   $h_l \circ \kappa \, (\partial \rD^d) = \{ \xi_0 \}$;

\item[(iii)]  $\kappa (0) = e$;   

\item[(iv)]  $d \kappa (T_0 \rD^d) = T_v [e]$;

\item[(v)]   the orientation of the relative cycle $\kappa (\rD^d)$ is given by $o$.
\end{enumerate}
We use Lemma \ref{lemma-morse-dual} to view the class $PL [e , \gamma, o]$ as an element of $M_l (\Pone)$. It is a standard fact that the class $PL [e , \gamma, o]$ depends on the path $\gamma$ only through the smooth homotopy class of $\gamma$ within the class of all paths satisfying properties (P1)-(P5).

\subsection{A canonical basis up to sign for \texorpdfstring{$M_l (\Pone)$}{Lg}}
\label{subsec-basis-up-to-sign}
In this subsection, we introduce a canonical basis up to sign $\{ \pm \basis_w \}_{w \in W_\La}$ for the Morse group $M_l (\Pone)$. Here, the word ``canonical'' refers to the fact that the basis $\{ \pm \basis_w \}$ will, in a certain sense, be independent of the choices of the basepoints $\basepta, l \in \La_\bR^+$.

For each $w \in W_\La$, we choose a path $\gamma_w : [0, 1] \to \bC$, satisfying properties (P1)-(P5) of Section~\ref{subsec-PLclasses} for $e = e_w$, and such that we further have:
\begin{enumerate}[topsep=-1.5ex]
\item[(P6)]  $\zeta_l (\gamma_w (t)) > 0$ for all $t \in (0, 1)$.
\end{enumerate}
Note that, in view of Lemma~\ref{lemma-crit-points} (iii), property (P6) determines the path $\gamma_w$ uniquely up to smooth homotopy within the class of all paths satisfying properties (P1)-(P5).

For each $w \in W_\La$, we set:
\beqn
\pm \basis_w = PL [e_w, \gamma_w] \in M_l (\Pone).
\eeqn
Note that, in this definition, we do not specify an orientation of the positive eigenspace $T_v [e_w]$. As a result, each element $\basis_w$ is only defined up to sign.

\begin{lemma}\label{lemma-basis}
The $\baseRing$-module $M_l (\Pone)$ is free, and the elements $\{ \pm \basis_w \}_{w \in W_\La}$ form a basis of $M_l (\Pone)$ up to sign.
\end{lemma}

\begin{proof}
This is a standard application of Picard-Lefschetz theory, using Lemmas \ref{lemma-gencov}, \ref{lemma-morse-dual}, and \ref{lemma-crit-points}.
\end{proof}

The following remark, which is standard in Picard-Lefschetz theory, explains the sense in which the basis $\{ \pm \basis_w \}$ is canonical. Recall the sheaves $\{ \cP_{\bargenpta} \}$ introduced in Section~\ref{subsec-twisted} for all $\bargenpta \in \La^{rs} / W_\La$. We have $\Pone = \cP_{\barbasepta}$. For every pair $\genpta, l \in \La_\bR^+$, let $M [\genpta, l] = M_l (\cP_{\bargenpta})$, where $\bargenpta = f (\genpta)$. As in Lemma~\ref{lemma-basis}, we have a basis up to sign:
\beqn
\{ \pm \basis_w [\genpta, l] \}_{w \in W_\La} \subset M [\genpta, l].
\eeqn

\begin{remark}\label{rmk-canonical}
(i) The groups $\{ M [\genpta, l] \}_{\genpta, l \in \La_\bR^+}$ form a local system over the set $\La_\bR^+ \times \La_\bR^+$, which is contractible. The holonomy of this local system in the ``$\genpta$-direction'' is given by the monodromy in the family, i.e., by the structure of the perverse sheaf $\cP$ of equation~\eqref{sheafcP}, as in Section~\ref{subsec-twisted}. \\
\indent
(ii) For each $w \in W_\La$, the basis elements $\{ \pm \basis_w [\genpta, l] \}_{\genpta, l \in \La_\bR^+}$ form a section up to sign of the local system $\{ M [\genpta, l] \}_{\genpta, l \in \La_\bR^+} \,$.
\end{remark}

In working with the basis $\{ \pm \basis_w \}_{w \in W_\La}$ of Lemma~\ref{lemma-basis}, Remark \ref{rmk-canonical} will give us the flexibility to assume, whenever convenient, that either $\basepta$ or $l$ is near one of the walls of the Weyl chamber $\La_\bR^+$.

\subsection{Partial description of the monodromy actions on \texorpdfstring{$M_l (\Pone)$}{Lg}}
\label{subsec-partial-description}
Let us write:
\beqn
\lambda_l : \widetilde B_{W_\La} \to \on{Aut} (M_l (\Pone)),
\eeqn
for the microlocal monodromy action arising from the structure of $\Pone$ as a $K$-equivariant perverse sheaf. The content of Theorem \ref{thm-chi-eq-one} is in describing the actions $\lambda_l$ and $\mu_l$ on the Morse group $M_l (\Pone)$. Note that, by construction, these two actions commute with each other. In this subsection, we provide a partial description of these two actions in terms of the basis $\{ \pm \basis_w \}_{w \in W_\La}$. Recall the set of simple reflections $S \subset W_\La$, and the associated length function $\BruhatLength : W_\La \to \bN$. Also, recall the counter-clockwise braid generators $\{ \sigma_s \}_{s \in S} \subset B_{W_\La}$.

\begin{prop}\label{prop-partial-description}
We have:
\begin{enumerate}[topsep=-1ex]
\item[(i)]   for every $w \in W_\La$ and every $s \in S$ such that $\BruhatLength (w) < \BruhatLength (sw)$, 
\beqn
\lambda_l \circ \tilde r \, (\sigma_s) (\pm \basis_w) = \pm \basis_{sw} \, ;
\eeqn

\item[(ii)]  for every $w \in W_\La$ and every $s \in S$ such that $\BruhatLength (w) < \BruhatLength (ws)$, 
\beqn
\mu_l (\sigma_s) (\pm \basis_w) = \pm \basis_{ws} \, .
\eeqn
\end{enumerate}
\end{prop}

\begin{proof}
For part (i), recall the hyperplane $\La_s \subset \La$ fixed by $s$, and let $\overline{\La_\bR^+} \subset \La_\bR$ be the closure of the Weyl chamber $\La_\bR^+$. Pick a point $l_s \in \La_s \cap \overline{\La_\bR^+}$, such that $\on{Stab}_{W_\La} (l_s) = W_s$ ($= \{ 1, s \}$). Using Remark~\ref{rmk-canonical}, we can assume that the basepoint $\basepta \in \La_\bR^+$ is sufficiently generic, so that:
\beqn
h_{l_s} (e_{w_1}) = h_{l_s} (e_{w_2}) \implies w_2 = s w_1 \, , \;\; \text{for all} \;\; w_1, w_2 \in W_\La \, .
\eeqn
Pick a small $\epsilon > 0$. Using Remark~\ref{rmk-canonical}, we can further assume that $l \in \La_\bR^+$ is the unique point such that $l - l_s \in \La_s^\p$ and $\on{dist} (l, l_s) = \epsilon$, where $\La_s^\p \subset \La$ is the orthogonal complement to $\La_s$, as in Section~\ref{subsec-rank-one}.

By the assumptions on the basepoints $\basepta, l \in \La_\bR^+$, the critical values $\{ h_l (e_w) \}_{w \in W_\La}$ are all distinct, and they appear on the real line as a set of $| W_\La | / 2$ pairs of nearby points, corresponding to the right cosets $W_s \backslash W_\La$.

We now define an explicit loop $\Gamma_s : [0, 1] \to \Lp^{rs}$, representing the element $\tilde r (\sigma_s) \in \widetilde B_{W_\La} = \pi_1^K (\Lp^{rs}, l)$. Recall that we write: $r = \tilde p \circ \tilde r : B_{W_\La} \to \widetilde W_\La$. Pick a representative $k_s \in N_K (\La)$ for the element $r (\sigma_s) \in N_K (\La) / Z_K (\La)^0$. By the definition of a regular splitting (Definition~\ref{defn-regular-splitting}), we have $\tilde r (\sigma_s) \in \widetilde B^0_{W_\La}$. It follows that $k_s \in K^0$. Let $\Gamma_{k_s} : [0, 1] \to K^0$ be any continuous path with $\Gamma_{k_s} (0) = 1$ and $\Gamma_{k_s} (1) = k_s$. We define:
\beqn
\Gamma_s (t) = \left\{ \begin{array}{ll}
l_s + \exp (2 t \pi \mathbf{i}) \cdot (l - l_s) \in \La^{rs} & \text{ for } t \in [0, 1/2] \\
\Gamma_{k_s} (2 t - 1) \, s (l) & \text{ for } t \in [1/2, 1].
\end{array}\right.
\eeqn
Note that $s |_{\La_s^\p} = -1$. Therefore, the two branches of the definition of $\Gamma_s$ agree for $t = 1/2$, with $\Gamma_s (1/2) = s (l)$.

Write $l_t = \Gamma_s (t)$ and $h_t = h_{l_t}$, for short. The proposition follows by tracing what happens to the critical values of $h_t |_{X_{\barbasepta}}$, as $t$ traces the interval $[0, 1]$. To that end, let us write $Z_t$ for the set of critical points of $h_t |_{X_{\barbasepta}}$. Note that, for $t \in [0, 1/2]$, we have $Z_t  = Z_l = \{ e_w \}_{w \in W_\La}$, and for $t \in [1/2, 1]$, we have $Z_t = \Gamma_{k_s} (2 t - 1) \, Z_l$.

For every right coset $u \in W_s \backslash W_\La$, let $w_u \in u$ be the unique element such that $\BruhatLength (w_u) < \BruhatLength (s w_u)$. By Lemma~\ref{lemma-bruhat}, we have $h_l (e_{s w_u}) < h_l (e_{w_u})$. Let $C_u \subset \bC$ be the small circle whose diameter is the closed interval:
\beq\label{eqn-interval}
[h_l (e_{s w_u}), h_l (e_{w_u})] \subset \bR.
\eeq
Then, by construction, as $t$ traces the interval $[0, 1/2]$, each pair of critical values $(h_l (e_{w_u}), h_l (e_{s w_u}))$, $u \in W_s \backslash W_\La$, moves counter-clockwise halfway around the circle $C_u$. Further, as $t$ traces the interval $[1/2, 1]$, the critical values of $h_t |_{X_{\barbasepta}}$ remain unchanged. The intervals \eqref{eqn-interval} for different $u \in W_s \backslash W_\La$ are mutually disjoint. Part (i) of the proposition follows.

The proof of part (ii) is similar, with the roles of $\basepta$ and $l$ reversed. We briefly indicate the argument. We pick a generic point $\baseptas \in \overline{\La_\bR^+} \cap \La_s$, and a generic $l \in \La_\bR^+$. We then assume that $\basepta \in \La_\bR^+$ is near $\baseptas$, with $\basepta - \baseptas \in \La_s^\p$. We consider the loop $\Gamma_s : [0, 1] \to \La^{rs} / W_\La$, defined by:
\beqn
\Gamma_s (t) = f (\baseptas + \exp (t \pi \mathbf{i}) \cdot (\basepta - \baseptas)),
\eeqn
which represents $\sigma_s \in B_{W_\La} = \pi_1 (\La^{rs} / W_\La, \basepta)$. The critical values \linebreak $\{ h_l (e_w) \}_{w \in W_\La}$ are organized into pairs, corresponding the left cosets $W_\La / W_s$. As $t$ traces the interval $[0, 1]$, the critical values of $h_l |_{X_{\Gamma_s (t)}}$ move counter-clockwise in pairs, as in the proof of part (i). Part (ii) of the proposition follows.
\end{proof}

\subsection{Picking the signs for the basis \texorpdfstring{$\{ \pm \basis_w \}$}{Lg}}
\label{subsec-picking-signs}
In this subsection, we use Proposition \ref{prop-partial-description} to assign specific signs to the basis elements $\{ \pm \basis_w \}_{w \in W_\La}$. 

By, for example, \cite[Proposition 8.3.3]{S}, there exists a unique map of sets $\LiftToBraids : W_\La \to B_{W_\La}$ such that:
\begin{enumerate}[topsep=-1.5ex]
\item[(i)]    $\LiftToBraids (1) = 1$; $\LiftToBraids (s) = \sigma_s$ for $s \in S$;

\item[(ii)]  $\LiftToBraids (w_1 w_2) = \LiftToBraids (w_1) \LiftToBraids (w_2)$ for every pair $w_1, w_2 \in W_\La$ with $\BruhatLength (w_1 w_2) = \BruhatLength (w_1) + \BruhatLength (w_2)$.
\end{enumerate}
We will write $b_w = \beta(w) \in B_{W_\La}$. Note that, using our previous notation, we have $b_s = \sigma_s$ for a simple reflection $s \in S$. 

It follows from Proposition~\ref{prop-partial-description}, arguing by induction on the length $\BruhatLength (w)$ of an element $w \in W_\La$, that we have:
\beq\label{eqn-cor-partial-description}
\lambda_l \circ \tilde r \, (b_w) (\pm \basis_1) = \pm \basis_w \;\; \text{for every} \;\; w \in W_\La \, .
\eeq

Now, for each $w \in W_\La$, recall the positive eigenspace $T_v [e_w] \subset T_{e_w} X_{\barbasepta}$. Let $\cO [w]$ be the $\bZ / 2$-torsor of orientations of $T_v [e_w]$. Note that, since the Weyl chamber $\La_\bR^+$ is contractible, the torsor $\cO [w]$ is canonically independent of the choice of the basepoints $\basepta, l \in \La_\bR^+$. Pick an orientation $o_1 \in \cO [1]$. We define:
\beqn
\basis_1 = PL [e_1, \gamma_1, o_1] \in M_l (\Pone); \;
\basis_w = \lambda_l \circ \tilde r \, (b_w) \, \basis_1 \in M_l (\Pone) \text{ for } w \in W_\La - \{ 1 \}.
\eeqn
Note that, by equation~\eqref{eqn-cor-partial-description}, for every $w \in W_\La$, there is a unique $o_w \in \cO [w]$ such that:
\beqn
\basis_w = PL [e_w, \gamma_w, o_w] \in M_l (\Pone).
\eeqn
We can now summarize much of the content of Theorem \ref{thm-chi-eq-one} in terms of the basis:
\beqn
\{ \basis_w \}_{w \in W_\La} \subset M_l (\Pone), 
\eeqn
as follows.

\begin{prop}\label{prop-chi-eq-one-basis}
We have:
\begin{enumerate}[topsep=-1.5ex]
\item[(i)]   for every $w \in W_\La$ and every $s \in S$ such that $\BruhatLength (w) < \BruhatLength (sw)$,
\begin{eqnarray*}
\lambda_l \circ \tilde r \, (\sigma_s) \, \basis_w &=& \basis_{sw} \, ;
\\
\lambda_l \circ \tilde r \, (\sigma_s) \, \basis_{sw} &=& - \basis_w + 2 \, \basis_{sw} \, , \;\; \text{if $\delta (s)$ is odd};  
\\
\lambda_l \circ \tilde r \, (\sigma_s) \, \basis_{sw} &=& \basis_w \, , \;\; \text{if $\delta (s)$ is even};
\end{eqnarray*}

\item[(ii)]  for every $w \in W_\La$ and every $s \in S$ such that $\BruhatLength (w) < \BruhatLength (ws)$,
\begin{eqnarray*}
\mu_l (\sigma_s) \, \basis_w &=& \basis_{ws} \, , \;\; \text{if $\delta (s)$ is odd};
\\
\mu_l (\sigma_s) \, \basis_w &=& - \basis_{ws} \, , \;\; \text{if $\delta (s)$ is even};
\\
\mu_l (\sigma_s) \, \basis_{ws} &=& - \basis_w + 2 \, \basis_{ws} \, , \;\; \text{if $\delta (s)$ is odd};
\\
\mu_l (\sigma_s) \, \basis_{ws} &=& - \basis_w \, , \;\; \text{if $\delta (s)$ is even}.
\end{eqnarray*}
\end{enumerate}
\end{prop}

The next several subsections will be devoted to the proof of Proposition \ref{prop-chi-eq-one-basis}.

\subsection{Equations for the monodromy in the family}
\label{subsec-minimal-polynomials}
In this subsection, we provide the following key geometric input into the proof of Proposition \ref{prop-chi-eq-one-basis}.

\begin{prop}\label{prop-minimal-polynomials}
Let $s \in S$ be a simple reflection. The monodromy transformation $\mu (\sigma_s) \in \on{End} (\Pone)$ satisfies:
\begin{eqnarray*}
(\mu (\sigma_s) - 1)^2 &=& 0, \;\;\;\; \text{if $\delta (s)$ is odd};
\\
\mu (\sigma_s)^2 - 1 &=& 0, \;\;\;\; \text{if $\delta (s)$ is even}.
\end{eqnarray*}
\end{prop}

\begin{proof}
The majority of the argument consists of a reduction to rank one, i.e., to the case where $\dim \La = 1$. The key to the argument is to make use of Thom's $A_f$ condition for the adjoint quotient map $f : \Lp \to \La / W_\La$. We use this property of $f$ to show that specializing to the discriminant locus and then to the nilpotent cone gives the same result as specializing to the nilpotent cone directly.

With that in mind, similarly to the beginning of the proof of Proposition~\ref{prop-partial-description}, we pick a generic point $\baseptas \in \La_s \cap \overline{\La_\bR^+}$ and a small $\epsilon > 0$, and we assume (using Remark~\ref{rmk-canonical}) that $\basepta \in \La_\bR^+$ is the unique point such that $\basepta - \baseptas \in \La_s^\p$ and $\on{dist} (\basepta, \baseptas) = \epsilon$. Let $\rD_\gamma = \{ z \in \bC \; | \; |z| < 2 \, \epsilon \}$. Define an analytic arc $\gamma : \rD_\gamma \to \La$ by:
\beqn
\gamma (z) = \baseptas + (z / \epsilon) \cdot (\basepta - \baseptas).
\eeqn
Let $\rD_{\bar\gamma} = \{ z \in \bC \; | \; |z| < 4 \, \epsilon^2 \}$, and let $\bar\gamma : \rD_{\bar\gamma} \to \La / W_\La$ be the unique analytic arc such that:
\beqn
f \circ \gamma \, (z) = \bar\gamma (z^2) \;\; \text{for every} \;\; z \in \rD_\gamma \, .
\eeqn
We base change the family $f: \Lp \to \La / W_\La$ to $\rD_{\bar\gamma}$, to obtain a family:
\beqn
f_{\bar\gamma} : \Lp_{\bar\gamma} = \Lp \times_{\La / W_\La} \rD_{\bar\gamma} \to \rD_{\bar\gamma} \, .
\eeqn
Let $\barbaseptas = f(\baseptas)$ and $X_{\barbaseptas} = f^{-1} (\barbaseptas) = f_{\bar\gamma}^{-1} (0)$. Write $\Lp^{rs}_{\bar\gamma} = f_{\bar\gamma}^{-1} (\rD_{\bar\gamma} - \{ 0 \}) \subset \Lp_{\bar\gamma}$. Form the nearby cycle sheaf:
\beqn
P_{\bar\gamma} = \psi_{f_{\bar\gamma}} \, \baseRing_{\Lp^{rs}_{\bar\gamma}} [-] \in \on{Perv}_K (X_{\barbaseptas}),
\eeqn
which, as usual, we make perverse by an appropriate shift. Let $\mu_{\bar\gamma} : P_{\bar\gamma} \to P_{\bar\gamma}$ be the associated (counter-clockwise) monodromy. Consider the functor:
\beqn
\psi_{\barbaseptas} : \on{Perv}_K (X_{\barbaseptas}) \to \on{Perv}_K (\cN_\Lp),
\eeqn
defined in the same manner as the functor $\psi_{\barbasepta}$ of Section~\ref{subsec-twisted}, now using the family:
\beqn
\cZ_{\barbaseptas} = \{ (x, c) \in \Lp \times \bC \; | \; f (x) = c \, \barbaseptas \} \to \bC.
\eeqn
We claim that:
\beq
\label{factorization}
\psi_{\barbaseptas} (P_{\bar\gamma}) \cong \Pone \;\;\;\; \text{and} \;\;\;\;
\psi_{\barbaseptas} (\mu_{\bar\gamma}) = \mu (\sigma_s).
\eeq

To prove this fact, we proceed as in Section~\ref{subsec-twisted}, and consider the family:
\beqn
F : \cZ_{\bar\gamma} = \{ (x, z, c) \in \Lp \times \rD_{\bar\gamma} \times \bC \; | \; f(x) = c \, \bar\gamma (z) \} \to \rD_{\bar\gamma} \times \bC,
\eeqn
which extends, via pullback to $\rD_{\bar\gamma}$, the appropriate part of the family $\cZ$ of Section~\ref{subsec-twisted} across the codimension one locus in $\La / W_{\La}$. Note that this family is $(K \times \bC^*)$-equivariant, where $\bC^*$ acts trivially on $\rD_{\bar\gamma}$. The space $\cZ_{\bar\gamma}$ is equipped with a natural stratification $\cS_{\cZ_{\bar\gamma}}$, which we now describe.
Recall that we write $\cN_\Lp / K$ for the set of $K$-orbits in $\cN_\Lp$. Let $\rD_{\bar\gamma}^* = \rD_{\bar\gamma} \cap \bC^*$. Consider the decomposition:
\beqn
\cZ_{\bar\gamma} = \bigcup_{i = 0}^3 \, \cZ_{\bar\gamma, i} \, ,
\eeqn
where $\cZ_{\bar\gamma, 0} = F^{-1} (\{ 0 \} \times \{ 0 \})$, $\cZ_{\bar\gamma, 1} = F^{-1} (\{ 0 \} \times \bC^*)$, $\cZ_{\bar\gamma, 2} = F^{-1} (\rD_{\bar\gamma}^* \times \{ 0 \})$, and $\cZ_{\bar\gamma, 3} = F^{-1} (\rD_{\bar\gamma}^* \times \bC^*)$. The stratification $\cS_{\cZ_{\bar\gamma}}$ respects this decomposition. The locus $\cZ_{\bar\gamma, 0} \cong \cN_\Lp$ is stratified by $K$-orbits. The locus $\cZ_{\bar\gamma, 1}$ is stratified by $(K \times \bC^*)$-orbits. The locus $\cZ_{\bar\gamma, 2} = \cN_\Lp \times \rD_{\bar\gamma} \times \{ 0 \}$ is stratified by the products $\cO \times \rD_{\bar\gamma} \times \{ 0 \}$ for all $\cO \in \cN_\Lp / K$. Finally, the locus $\cZ_{\bar\gamma, 3}$ is a single stratum.

Let us write $F_1 : \cZ_{\bar\gamma} \to \rD_{\bar\gamma}$ and $F_2 : \cZ_{\bar\gamma} \to \bC^*$ for the two components of $F$. The family $F$ satisfies the following properties.
\begin{enumerate}[topsep=-1.5ex]
\item[(P0)] For every stratum $S_0 \subset \cZ_{\bar\gamma, 0}$, Thom's $A_F$ condition holds for the pair $(\cZ_{\bar\gamma, 3}, S_0)$.

\item[(P1)] For every stratum $S_1 \subset \cZ_{\bar\gamma, 1}$, Thom's $A_{F_1}$ condition holds for the pair $(\cZ_{\bar\gamma, 3}, S_1)$.

\item[(P2)] For every stratum $S_2 \subset \cZ_{\bar\gamma, 2}$, Thom's $A_{F_2}$ condition holds for the pair $(\cZ_{\bar\gamma, 3}, S_2)$.
\end{enumerate}
Property (P0) follows from the $K$-equivariance of $F$, property (P1) follows from the $(K \times \bC^*)$-equivariance of $F$, and property (P2) is analogous to Proposition~\ref{prop-asubf2}.

Let $R = [0, 4 \, \epsilon^2) \times [0, + \infty) \subset \rD_{\bar\gamma} \times \bC$. Let $\cZ_R = F^{-1} (R) \subset \cZ_{\bar\gamma}$. For each $i = 0, 1, 2, 3$, let $\cZ_{R, i} = \cZ_R \cap \cZ_{\bar\gamma, i}$, and let $j_i : \cZ_{R, i} \to \cZ_R$ be the inclusion map. We now form the following diagram of perverse sheaves on $\cZ_{R, 0} \cong \cN_\Lp$:
\beq\label{roofDiagram}
j_0^* \, (j_1)_* \, j_1^* \, (j_3)_* \, \baseRing_{\cZ_{R, 3}} [-] \; \longleftarrow \; j_0^* \, (j_3)_* \, \baseRing_{\cZ_{R, 3}} [-] \; \longrightarrow \; j_0^* \, (j_2)_* \, j_2^* \, (j_3)_* \, \baseRing_{\cZ_{R, 3}} [-],
\eeq
where all pushforward functors are understood to be derived. Here, the shifts are used to make the sheaves perverse, as usual, and the maps are induced by the inclusions $\cZ_{R, i} \to \cZ_{R, i} \cup \cZ_{R, 3}$, $i = 1, 2$. By property (P1) of the family $F$, the sheaf $j_1^* \, (j_3)_* \, \baseRing_{\cZ_{R, 3}}$ is constructible with respect to the stratification of $\cZ_{\bar\gamma, 1}$, and therefore we have:
\beqn
j_0^* \, (j_1)_* \, j_1^* \, (j_3)_* \, \baseRing_{\cZ_{R, 3}} [-] \cong \psi_{\barbaseptas} (P_{\bar\gamma}).
\eeqn
By property (P2) of the family $F$, the sheaf $j_2^* \, (j_3)_* \, \baseRing_{\cZ_{R, 3}}$ is constructible with respect to the stratification of $\cZ_{\bar\gamma, 2}$, and therefore we have:
\beqn
j_0^* \, (j_2)_* \, j_2^* \, (j_3)_* \, \baseRing_{\cZ_{R, 3}} [-] \cong \Pone.
\eeqn
Finally, by property (P0) of the family $F$, both of the maps in diagram~\eqref{roofDiagram} are isomorphisms.

More precisely, for the last assertion, let $M_{\bar\gamma} = \Lp \times \rD_{\bar\gamma} \times \bC \supset \cZ_{\bar\gamma}$, and fix a Hermitian metric on $M_{\bar\gamma}$. Fix a point $m \in \cZ_{\bar\gamma, 0}$. For $\delta > 0$, let $\rB_{m, \delta} \subset M_{\bar\gamma}$ be the closed $\delta$-ball around $m$. By the Whitney conditions for the stratification of $\cZ_{\bar\gamma, 0}$, there exists an $\delta_0 > 0$, such that the ball $\rB_{m, \delta_0}$ is compact, and for every $\delta \in (0, \delta_0]$, the boundary $\partial \rB_{m, \delta}$ meets the strata of $\cZ_{\bar\gamma, 0}$ transversely. By the $A_F$ condition of property (P0), for every $\delta_1 \in (0, \delta_0]$, there exists a $\zeta \in (0, 4 \, \epsilon^2)$, such that for every $\delta \in [\delta_1, \delta_0]$ and every pair $(z, c) \subset (0, \zeta) \times (0, \zeta) \subset R$, the intersection $\partial \rB_{m, \delta} \cap F^{-1} (z, c)$ is transverse. It follows that the restriction of $F$ to the intersection $\rB_{m, \delta} \cap F^{-1} ((0, \zeta) \times (0, \zeta))$ is a smooth fibration, whose fibers are manifolds with boundary, and the stalk of the sheaf $j_0^* \, (j_3)_* \, \baseRing_{\cZ_{R, 3}}$ at $m$ is given by the cohomology of the total space, or equivalently of any of the fibers, of this fibration. By tracing the definitions, we can conclude that the maps in diagram~\eqref{roofDiagram} induce isomorphisms on the stalks at $m$. The assertion follows.

This proves the first assertion of equation~\eqref{factorization}. For the second assertion we need to consider a version of diagram~\eqref{roofDiagram} parametrized by the circle $S^1 = \{ z \in \bC \; | \; |z| = 1 \}$. Namely, for each $z \in S^1$, we obtain an analog of diagram~\eqref{roofDiagram} by restricting the family $F$ to the region $R_z = (z \cdot [0, 4 \, \epsilon^2)) \times [0, + \infty) \subset \rD_{\bar\gamma} \times \bC$. We omit the details.

Thus, we are reduced to proving the equations for $\mu_{\bar\gamma}$ in place of $\mu (\sigma_s)$, which we do by reducing to rank one. Recall the subspaces $\Lp_s = Z_\Lp (\La_s) = \La \oplus [\Lk_s, \La]$, $\bar \Lp_s = \La_s^\p \oplus [\Lk_s, \La]$ of $\Lp$, and the symmetric pair $(\bar G_s, \bar K_s) \subset (G, K)$, introduced in Section~\ref{subsec-rank-one}. The space $\bar \Lp_s$ is the $(-1)$-eigenspace of this pair. Consider the intersection $X_{\barbaseptas} \cap \Lp_s$. In view of \eqref{diagramFs} and \eqref{eqn-Fs-product}, the connected component of this intersection that passes through $\baseptas$ can be described as $X_{\barbaseptas} \cap (\baseptas + \bar \Lp_s) = \baseptas + \bar f_s^{-1} (0)$ (cf. the discussion above \eqref{eqn-normal-slice}). Note that $K \cdot \baseptas$ is the unique closed $K$-orbit in $X_{\barbaseptas}$. Note also, as in the proof of Proposition~\ref{prop-kr-is-regular} (see equation~\eqref{eqn-normal-slice}), that the subspace $\Lp_s \subset \Lp$ is a normal slice to $K \cdot \baseptas$ at $\baseptas$. Recall that the group $\bar K_s$ acts on the zero-fiber $\bar f_s^{-1} (0)$ with orbits that are $\bC^*$-conic. It follows that, for every $K$-orbit $\cO \subset X_{\barbaseptas}$, the locus $\cO_{\barbaseptas} \coloneqq \cO \cap (\baseptas + \bar \Lp_s)$ is non-empty, and the intersection $\cO \cap \Lp_s \subset \Lp$ is transverse along $\cO_{\barbaseptas}$. Therefore, there is a well-defined perverse (i.e., properly shifted) restriction functor:
\beqn
j^* : \on{Perv}_K (X_{\barbaseptas}) \to \on{Perv}_{\bar K_s} (X_{\barbaseptas} \cap (\baseptas + \bar \Lp_s)),
\eeqn
where $j : X_{\barbaseptas} \cap (\baseptas + \bar \Lp_s) \to X_{\barbaseptas}$ is the inclusion map. Since every $K$-orbit in $X_{\barbaseptas}$ meets the locus $\baseptas + \bar \Lp_s$, the functor $j^*$ is faithful. Thus, it will suffice to prove the required quadratic relations for the restriction:
\beqn
j^* (\mu_{\bar\gamma}) \in \on{End} (j^* (P_{\bar\gamma})).
\eeqn
On the other hand, the sheaf $j^* (P_{\bar\gamma})$ is the nearby cycles in the family:
\beqn
\baseptas + \bar \Lp_s \to (\baseptas + \bar \Lp_s) \inv K_s \cong \La_s^\p / W_s \, ,
\eeqn
and $j^* (\mu_{\bar\gamma})$ is the associated monodromy transformation.

Thus, we are reduced to the case of the symmetric pair $(\bar G_s, \bar K_s)$ of rank one. In this case, we have $\dim \bar \Lp_s = \delta (s) + 1$ (see equation~\eqref{eqn-barLp-s}), the little Weyl group is $W_{\La_s^\p} \cong \bZ / 2$, and the quotient map $\bar f_s$ is a non-degenerate homogeneous quadratic polynomial. The equations for the monodromy in this case are well known; see, for example, \cite[Example 3.7]{G2}.
\end{proof}

\subsection{The fundamental class of \texorpdfstring{$X_{\barbasepta}$}{Lg}}
\label{subsec-fund-class}
In this subsection, we provide another geometric ingredient for the proof of Proposition \ref{prop-chi-eq-one-basis}. Consider the absolute homology group:
\beqn
M^0_l (\Pone) \coloneqq H_d (X_{\barbasepta}; \baseRing).
\eeqn
From the long exact sequence of the pair of Lemma \ref{lemma-morse-dual}, we can see that $M^0_l (\Pone)$ is naturally a submodule of the Morse group $M_l (\Pone)$. Recall the character $\tau : \widetilde B_{W_\La} \to \{ \pm 1 \}$ of Section~\ref{subsec-char}.

\begin{prop}
\label{prop-fundamental-class}
We have:
\begin{enumerate}[topsep=-1.5ex]
\item[(i)]    $M^0_l (\Pone) \cong \baseRing$;

\item[(ii)]   both of the monodromy actions $\lambda_l$ and $\mu_l$ preserve the submodule $M^0_l (\Pone) \subset M_l (\Pone)$;

\item[(iii)]  the microlocal monodromy action $\lambda_l$ restricted to $M^0_l (\Pone)$ is given by $\tau$;

\item[(iv)]  the monodromy in the family action $\mu_l$ restricts to $M^0_l (\Pone)$ as follows:
\beqn
\text{for every $s \in S$, we have:} \;\;\;\; \mu (\sigma_s) |_{M^0_l (\Pone)} = (-1)^{\delta (s) + 1} \, .
\eeqn
\end{enumerate}
\end{prop}

\begin{proof}
Recall the compact form $K_\bR$ of $K$, introduced in Section \ref{subsec-structure-of-Wa}. By a theorem of Mostow, \cite{Mo}, see also the discussion in~\cite{SV}, the orbit $X_{\barbasepta} = K \cdot \basepta$ can be identified with the normal bundle to the orbit:
\beqn
C_{\barbasepta} = K_\bR \cdot \basepta \subset X_{\barbasepta} \, .
\eeqn
In particular, $C_{\barbasepta}$ is a deformation retract of $X_{\barbasepta}$. As the orbit $C_{\barbasepta}$ is connected and orientable, see, for example, \cite[Corollary 1.1.10]{Ko}, we conclude that $M^0_l (\Pone) = H_d (X_{\barbasepta}; \baseRing) = H_d (C_{\barbasepta}; \baseRing) \cong \baseRing$. This proves part (i). Part (ii) follows from the definitions of the two actions. As to part (iii), it is clear that the subgroup $\widetilde B^0_{W_\La} \subset \widetilde B_{W_\La}$, which is the image of the map $\pi_1 (\Lp^{rs}, l) \to \pi_1^K (\Lp^{rs}, l) = \widetilde B_{W_\La}$, acts trivially on $M^0_l (\Pone)$. It remains to consider the action of $\widetilde B_{W_\La} / \widetilde B^0_{W_\La} \cong K / K^0$. Using the notation of Section~\ref{subsec-char}, we have $C_{\barbasepta} \cong K_\bR / \groupM_\bR$. By the construction of the character $\tau$, the component group $K / K^0 \cong K_\bR / K_\bR^0$ then acts on $H_d (C_{\barbasepta}; \baseRing)$ via $\tau$.

It remains to prove part (iv). Consider the family $f: \Lp^{rs} \to \La^{rs} / W_\La$, and pull it back to $\La^{rs}$ to obtain the following diagram:
\beqn
\begin{CD}
\Lp^{rs} \times_{\La^{rs} / W_\La} \La^{rs} @>>> \Lp^{rs}
\\
@V{\tilde f}VV @VV{f}V
\\
\La^{rs} @>>> \La^{rs} / W_\La \, .
\end{CD}
\eeqn
The family $\tilde f$ is trivial, i.e., we have $\Lp^{rs} \times_{\La^{rs} / W_\La} \La^{rs} \cong K / Z_K (\La) \times \La^{rs}$. Thus, the action of $\mu (\sigma_s)$ on the total homology $H_* (X_{\barbasepta})$ is induced by the right action of $s \in W_{\La} = N_K (\La) / Z_K (\La)$ on $X_{\barbasepta} = K / Z_K (\La)$. In particular, the action of $\mu (\sigma_s)$ on $H_d (X_{\barbasepta}; \baseRing) = H_d (C_{\barbasepta}; \baseRing)$ is given by the right action of $s \in N_{K_\bR} (\La) / Z_{K_\bR} (\La) = N_K (\La) / Z_K(\La)$ on $C_{\barbasepta} = K_\bR / Z_{K_\bR} (\La)$. We will now reduce the calculation to the rank one case. Recall the pair $(G_s, K_s)$ of Section~\ref{subsec-rank-one}. Write $K_{s, \bR} = K_s \cap K_\bR$. We have:
\beqn
s \in N_{K_{s, \bR}} (\La) / Z_{K_{s, \bR}} (\La) \subset N_{K_\bR} (\La) / Z_{K_\bR} (\La).
\eeqn
Consider the fibration:
\beqn
K_{s, \bR} / Z_{K_{s, \bR}} (\La) \to C_{\barbasepta} = K_\bR / Z_{K_\bR} (\La) \to K_\bR / K_{s, \bR} \, .
\eeqn
The self-map $R_s : C_{\barbasepta} \to C_{\barbasepta}$ of the total space of this fibration, given by the right multiplication by $s$, commutes with the projection to the base $K_\bR / K_{s, \bR}$. Therefore, the action of $R_s$ on the top homology $H_d (C_{\barbasepta}; \baseRing)$ is given by the action of $R_s$ on the top homology of the fiber $K_{s, \bR} / Z_{K_{s, \bR}} (\La)$. Thus, the problem is reduced to the case of the symmetric pair $(G_s, K_s)$, or equivalently, to the case of the symmetric pair $(\bar G_s, \bar K_s)$ with Cartan subspace of dimension $1$ (see Section~\ref{subsec-rank-one}).

In the case where $\dim \La = 1$, we have $d = \delta (s)$, the compact core $C_{\barbasepta}$ of $X_{\barbasepta}$ is homeomorphic to the sphere $S^d$, and the map $R_s : C_{\barbasepta} \to C_{\barbasepta}$ is conjugate by a homeomorphism to the antipodal map of $S^d$. The degree of this antipodal map is $(-1)^{\delta (s) + 1}$. This concludes the proof of part (iv).
\end{proof}

\subsection{The path \texorpdfstring{$\gamma_{1, s}$}{Lg}}
\label{subsec-near-wall}
In this subsection, we provide the final geometric ingredient for the proof of Proposition \ref{prop-chi-eq-one-basis}. Pick a simple reflection $s \in S$. As in the proofs of Propositions \ref{prop-partial-description} and \ref{prop-minimal-polynomials} above, we pick a generic point $\baseptas \in \La_s \cap \overline{\La_\bR^+}$ and a small $\epsilon > 0$, and we assume, for the duration of this subsection, that $\basepta \in \La_\bR^+$ is the unique point such that $\basepta - \baseptas \in \La_s^\p$ and $\on{dist} (\basepta, \baseptas) = \epsilon$. We also assume that $l = \basepta$.

For every $w \in W_\La - \{ 1, s \}$, we have:
\beq\label{eqn-near-wall}
0 < \xi_l (e_1) - \xi_l (e_s) \ll \xi_l (e_s) - \xi_l (e_w).
\eeq
To see this, consider the limit $\epsilon \to 0$, and use the fact that $\nu_\Lg (w \, \baseptas, \baseptas) < \nu_\Lg (\baseptas, \baseptas)$ for every $w \in W_\La - \{ 1, s \}$. To paraphrase inequality~\eqref{eqn-near-wall}, we can say that, with the given choice of $\basepta, l \in \La_\bR^+$, the critical values $h_l (e_1)$ and $h_l (e_s)$ are close to each other, and far to the right of all the other critical values of the restriction $h_{l, \barbasepta} : X_{\barbasepta} \to \bC$.

We pick a smooth path $\gamma_{1, s} : [0, 1] \to \bC$, satisfying properties (P1)-(P5) of Subsection~\ref{subsec-PLclasses} for $e = e_1$, and such that we further have:
\begin{enumerate}[topsep=-1.5ex]
\item[(P7)]    $\gamma_{1, s} (1/3) = (h_l (e_1) + h_l (e_s)) / 2 \in \bR$;

\item[(P8)]    $\gamma_{1, s} (2/3) = (- h_l (e_1) + 3 \, h_l (e_s)) / 2 \in \bR$;

\item[(P9)]    $\zeta_l (\gamma_{1, s} (t)) > 0$ for all $t \in (0, 1/3) \cup (2/3, 1)$;

\item[(P10)]  $\zeta_l (\gamma_{1, s} (t)) < 0$ for all $t \in (1/3, 2/3)$.
\end{enumerate}
Note that, by Lemma~\ref{lemma-crit-points} (iii) and inequality~\eqref{eqn-near-wall}, properties (P7)-(P10) determine the path $\gamma_{1, s}$ uniquely up to smooth homotopy within the class of all paths satisfying properties (P1)-(P5). Topologically, we can say that the path $\gamma_{1, s}$ differs from the path $\gamma_1$ of Section~\ref{subsec-basis-up-to-sign} only in that it passes the critical value $h_l (e_s)$ on the left, rather than on the right. Therefore:
\beq\label{eqn-pl-formula}
\begin{gathered}
\text{there exists an $m_s \in \bZ$, independent of $\baseRing$, such that} \\
PL [e_1, \gamma_{1, s}, o_1] = \basis_1 + m_s \, \basis_s \in M_l (\Pone).
\end{gathered}
\eeq
The integer $m_s$ can be interpreted, via the Picard-Lefschetz formula, as an intersection number for a pair of vanishing spheres. However, we will not be using this interpretation. The significance of the path $\gamma_{1, s}$ is given by the following lemma.

\begin{lemma}\label{lemma-generator-squared}
We have:
\begin{enumerate}[topsep=-1.5ex]
\item[(i)]    $\lambda_l \circ \tilde r \, (\sigma_s) \, \basis_s = \pm PL [e_1, \gamma_{1, s}, o_1];$

\item[(ii)]   $\mu_l (\sigma_s) \, \basis_s = \pm PL [e_1, \gamma_{1, s}, o_1]$.
\end{enumerate}
\end{lemma}

\begin{proof}
We prove (i), and the proof of (ii) is similar but easier. We proceed as in the proof of Proposition~\ref{prop-partial-description}, and consider the path $\Gamma_s : [0, 1] \to \Lp^{rs}$, representing $\tilde r (\sigma_s) \in \widetilde B_{W_\La} = \pi_1^K (\Lp^{rs}, l)$, defined in that proof. We also write $l_t = \Gamma_s (t)$ and $h_t = h_{l_t}$, for short. Each critical point $e_w$ of $h_{l, \barbasepta}$, $w \in W_\La$, gives rise to a continuous path $\tilde e_w : [0, 1] \to X_{\barbasepta}$, such that $\tilde e_w (0) = e_w$ and $\tilde e_w (t)$ is a critical point of $h_t |_{X_{\barbasepta}}$ for all $t \in [0, 1]$. In our situation, it suffices to only consider the critical points $\tilde e_1 (t)$ and $\tilde e_s (t)$, as the critical values corresponding to these critical points remain far to the right of all other critical values of $h_t |_{X_{\barbasepta}}$ for all $t \in [0, 1]$.

Let $C \subset \bC$ be the small circle whose diameter is the interval:
\beqn
[h_l (e_s), h_l (e_1)] \subset \bR.
\eeqn
As $t$ traces the interval $[0, 1/2]$, the pair of critical values $(h_t (\tilde e_1 (t)), h_t (\tilde e_s (t)))$ moves counter-clockwise halfway around the circle $C$. In particular, the path $\gamma_s$ gets taken to the path $\gamma_{1, s}$. As $t$ traces the interval $[1/2, 1]$, the critical values of $h_t |_{X_{\barbasepta}}$ remain unchanged. 
\end{proof}

\subsection{Proof of Proposition \ref{prop-chi-eq-one-basis}}
\label{subsec-proof-of-prop-chi-eq-one-basis}
We are now prepared to put together the geometric ingredients of Sections \ref{subsec-minimal-polynomials} - \ref{subsec-near-wall} to give a proof of Proposition \ref{prop-chi-eq-one-basis}. We begin with the following analog of Proposition \ref{prop-minimal-polynomials}.

\begin{lemma}\label{lemma-microlocal-minimal-polynomials}
Let $s \in S$ be a simple reflection. The microlocal monodromy transformation $\lambda_l \circ \tilde r \, (\sigma_s) \in \on{End} (M_l (\Pone))$ satisfies:
\begin{eqnarray*}
(\lambda_l \circ \tilde r \, (\sigma_s) - 1)^2 &=& 0, \;\;\;\; \text{if $\delta (s)$ is odd};
\\
(\lambda_l \circ \tilde r \, (\sigma_s))^2 - 1 &=& 0, \;\;\;\; \text{if $\delta (s)$ is even}.
\end{eqnarray*}
\end{lemma}

\begin{proof}
By Lemma \ref{lemma-basis}, it suffices to prove the lemma for $\baseRing = \bZ$, so we proceed with this assumption. Consider the set of right cosets $W_s \backslash W_\La$. For each coset $u \in W_s \backslash W_\La$, consider the rank $2$ submodule $M_l [u] \subset M_l (\Pone)$ generated by the basis elements $\{ \basis_w \}_{w \in u}$. We have:
\beq\label{eqn-decomposition}
M_l (\Pone) = \bigoplus_{u \in W_s \backslash W_\La} M_l [u].
\eeq
Let $u_1 \in W_s \backslash W_\La$ be the coset of $1 \in W_\La$. By Proposition \ref{prop-partial-description}, Lemma \ref{lemma-generator-squared}, and equation~\eqref{eqn-pl-formula}, we know that both the microlocal monodromy operator $\lambda_l \circ \tilde r \, (\sigma_s)$ and the monodromy in the family operator $\mu_l (\sigma_s)$ preserve the submodule $M_l [u_1] \subset M_l (\Pone)$. Furthermore, we have:
\beqn
\lambda_l \circ \tilde r \, (\sigma_s) \, \basis_1 = \pm \mu_l (\sigma_s) \, \basis_1 \, .
\eeqn
We also know that the restrictions $\lambda_l \circ \tilde r \, (\sigma_s) |_{M_l [u_1]}$ and $\mu_l (\sigma_s) |_{M_l [u_1]}$ commute, and that $\basis_1$ is a cyclic element for both. It follows that:
\beq\label{eqn-on-u1}
\lambda_l \circ \tilde r \, (\sigma_s) |_{M_l [u_1]} = \pm \mu_l (\sigma_s) |_{M_l [u_1]} \, .
\eeq

Pick any $u \in W_s \backslash W_\La$, and let $w_u \in u$ be the element such that $\BruhatLength (w_u) < \BruhatLength (s w_u)$. Consider the monodromy transformation $\mu_l (b_{w_u^{-1}}) : M_l (\Pone) \to M_l (\Pone)$. By Proposition \ref{prop-partial-description} (ii), we have:
\beqn
\mu_l (b_{w_u^{-1}}) (M_l [u_1]) = M_l [u].
\eeqn
This implies that the operator $\lambda_l \circ \tilde r \, (\sigma_s)$ preserves the submodule $M_l [u] \subset M_l (\Pone)$, and we have:
\beq\label{eqn-conjugation}
\text{the restriction $\lambda_l \circ \tilde r \, (\sigma_s) |_{M_l [u]}$ is conjugate to $\lambda_l \circ \tilde r \, (\sigma_s) |_{M_l [u_1]} \,$.}
\eeq

Write $\on{Spec} (\lambda_l \circ \tilde r \, (\sigma_s) |_{M_l [u_1]}) \subset \bC$ for the spectrum of the operator induced by $\lambda_l \circ \tilde r \, (\sigma_s)$ on the complex 2-space $M_l [u_1] \otimes_\bZ \bC$. We view this spectrum as a set without multiplicities. Combining equation~\eqref{eqn-on-u1} with Proposition \ref{prop-minimal-polynomials}, we can conclude that:
\begin{eqnarray*}
\on{Spec} (\lambda_l \circ \tilde r \, (\sigma_s) |_{M_l [u_1]}) &=& \{ 1 \} \;\; \text{or} \;\; \{ -1 \}, \;\;\;\; \text{if $\delta (s)$ is odd};
\\
\on{Spec} (\lambda_l \circ \tilde r \, (\sigma_s) |_{M_l [u_1]}) &=& \{ 1, -1 \}, \;\;\;\; \text{if $\delta (s)$ is even}.
\end{eqnarray*}
In view of~\eqref{eqn-conjugation}, the claim of the lemma for even $\delta (s)$ follows.

Assume now that $\delta (s)$ is odd. We need to rule out the possibility that:
\beqn
\on{Spec} (\lambda_l \circ \tilde r \, (\sigma_s) |_{M_l [u_1]}) = \{ -1 \}.
\eeqn
But this, together with~\eqref{eqn-conjugation}, would imply that:
\beqn
\on{Spec} (\lambda_l \circ \tilde r \, (\sigma_s)) = \{ -1 \},
\eeqn
which is in contradiction with Proposition \ref{prop-fundamental-class} (iii). Thus, we conclude that the operator $\lambda_l \circ \tilde r \, (\sigma_s)$ is unipotent, so we have:
\beq\label{eqn-odd-case-u1}
\lambda_l \circ \tilde r \, (\sigma_s) |_{M_l [u_1]} = \mu_l (\sigma_s) |_{M_l [u_1]} \, .
\eeq
The claim of the lemma for odd $\delta (s)$ follows.
\end{proof}

Part (i) of Proposition \ref{prop-chi-eq-one-basis} follows immediately from Lemma \ref{lemma-microlocal-minimal-polynomials} and the definition of the basis $\{ \basis_w \}_{w \in W_\La}$ in Section~\ref{subsec-picking-signs}. Turning to part (ii) of the proposition, recall that the actions $\lambda_l$ and $\mu_l$ commute with each other. Together with part (i), this implies that it suffices to verify the claim of part (ii) for $w = 1$ only. In the case where $\delta (s)$ is odd, the claim of part (ii) for $w = 1$ follows immediately from equation~\eqref{eqn-odd-case-u1}. For the case where $\delta (s)$ is even, we will need the following lemma.

Recall the submodule $M^0_l (\Pone) \subset M_l (\Pone)$ of Section~\ref{subsec-fund-class}. Pick an orientation of the $K_\bR$-orbit $C_{\barbasepta} \subset X_{\barbasepta}$
from the proof of Proposition \ref{prop-fundamental-class} (i), and let $F \in M^0_l (\Pone)$ be the corresponding generator. Let:
\beq\label{eqn-fund-class-generator}
F = \sum_{w \in W_\La} c_w \, \basis_w \, , \;\;\;\; c_w \in \baseRing.
\eeq
Let $s \in S$ be a simple reflection. Let $u_1 = \{ 1, s \} \in W_s \backslash W_\La$ be the identity coset, and let $M_l [u_1] \subset M_l (\Pone)$ be the rank $2$ submodule generated by $\{ \basis_1, \basis_s \}$, as in the proof of Lemma~\ref{lemma-microlocal-minimal-polynomials}. We define:
\beqn
F_{1, s} = c_1 \, \basis_1 + c_s \, \basis_s \in M_l [u_1] \subset M_l (\Pone),
\eeqn
where $c_1, c_s$ are as in equation~\eqref{eqn-fund-class-generator}.

\begin{lemma}\label{lemma-fund-class-coefs}
We have:
\begin{enumerate}[topsep=-1.5ex]
\item[(i)]    $c_1 \neq 0;$

\item[(ii)]   $\lambda_l \circ \tilde r \, (\sigma_s) \, F_{1, s} = F_{1, s} \, ;$

\item[(iii)]  $\mu_l (\sigma_s) \, F_{1, s} = (-1)^{\delta (s) + 1} \cdot F_{1, s} \, .$
\end{enumerate}
\end{lemma}

\begin{proof}
For every simple reflection $s' \in S$, we have a direct sum decomposition:
\beqn
M_l (\Pone) = \bigoplus_{u \in W_{s'} \backslash W_\La} M_l [u],
\eeqn
as in equation~\eqref{eqn-decomposition} of the proof of Lemma~\ref{lemma-microlocal-minimal-polynomials}. Moreover, by Proposition~\ref{prop-chi-eq-one-basis} (i) (which has been established), the operator $\lambda_l \circ \tilde r \, (\sigma_{s'})$ preserves this decomposition. By Proposition \ref{prop-fundamental-class} (iii), we have: $ \lambda_l \circ \tilde r (\sigma_{s'}) \, F = F$. It follows that:
\beq\label{eqn-coset}
\begin{gathered}
\lambda_l \circ \tilde r \, (\sigma_{s'}) (c_w \, v_w + c_{s' w} \, v_{s' w}) = c_w \, v_w + c_{s' w} \, v_{s' w} \\
\text{for each coset} \;\; \{w, s' w\} \in W_{s'} \backslash W_\La \, .
\end{gathered}
\eeq
Part (ii) follows by setting $w = 1$ and $s' = s$.

We now prove part (i). Suppose $c_1 = 0$. Pick a $w \in W_\La$, such that $c_w \neq 0$ and $c_{w'} = 0$ for every $w' \in W_\La$ with $\BruhatLength (w') < \BruhatLength (w)$. Pick an $s' \in S$, such that $\BruhatLength (s' w) < \BruhatLength (w)$. By Proposition~\ref{prop-chi-eq-one-basis} (i), we have $\lambda_l \circ \tilde r \, (\sigma_{s'}^{-1}) \, \basis_w = \basis_{s' w}$. Combining this with equation~\eqref{eqn-coset}, we obtain a contradiction as follows: $c_w \, v_w = c_w \, v_w + c_{s' w} \, v_{s' w} = \lambda_l \circ \tilde r \, (\sigma_{s'}^{-1}) (c_w \, v_w + c_{s' w} \, v_{s' w}) = c_w \, v_{s' w}$. Part (i) follows.

For part (iii), consider the set of left cosets $W_\La / W_s$. As for the right cosets, for each $u \in W_\La / W_s$, let $M_l [u] \subset M_l (\Pone)$ be the rank $2$ submodule generated by the basis elements $\{ \basis_w \}_{w \in u}$.
We obtain a direct sum decomposition:
\beqn
M_l (\Pone) = \bigoplus_{u \in W_\La / W_s} M_l [u].
\eeqn
As in the first paragraph of the proof of Lemma \ref{lemma-microlocal-minimal-polynomials}, the monodromy operator $\mu_l (\sigma_s)$ preserves the submodule $M_l [u_1] \subset M_l (\Pone)$ (note that $u_1 = W_s \subset W_\La$ is both a left and a right coset). Moreover, we can use an argument as in the second paragraph of the proof of Lemma \ref{lemma-microlocal-minimal-polynomials}, reversing the roles of the actions $\lambda_l$ and $\mu_l$, to conclude that $\mu_l (\sigma_s)$ preserves the submodule $M_l [u] \subset M_l (\Pone)$ for every $u \in W_\La / W_s$. Together with Proposition \ref{prop-fundamental-class} (iv), this implies part (iii) of the lemma.
\end{proof}

Assume now that $\delta (s)$ is even. Lemma \ref{lemma-fund-class-coefs} enables us to conclude that the sign in equation~\eqref{eqn-on-u1} is a minus, i.e., we have:
\beqn
\lambda_l \circ \tilde r \, (\sigma_s) |_{M_l [u_1]} = - \mu_l (\sigma_s) |_{M_l [u_1]} \, .
\eeqn
The claim of Proposition \ref{prop-chi-eq-one-basis} (ii) for $w = 1$ follows. This completes the proof of Proposition \ref{prop-chi-eq-one-basis}.

\subsection{Completing the proof of Theorem \ref{thm-chi-eq-one}}
To complete the proof of Theorem \ref{thm-chi-eq-one}, it remains to describe the action of the subgroup $I \subset \widetilde B_{W_\La}$ on $M_l (\Pone)$.

\begin{lemma}\label{lemma-trivial-inertia}
The restriction $\lambda_l |_{I}$ of the microlocal monodromy action:
\beqn
\lambda_l : \widetilde B_{W_\La} \to \on{Aut} (M_l (\Pone)),
\eeqn
is given by the character $\tau : I \to \{ \pm 1 \}$ of Section~\ref{subsec-char}.
\end{lemma}

\begin{proof}
Pick an element $u \in I$, and an element $g \in Z_K (\La)$ representing $u$. The action of $\lambda_l (u)$ on $M_l (\Pone)$ is induced by the action of $g$ on $X_{\barbasepta}$, via the isomorphism of Lemma \ref{lemma-morse-dual}. Note that $g (l) = l$, and so $g$ acts as an automorphism of the pair $(X_{\barbasepta}, \{ \xi_l (x) \geq \xi_0 \})$ appearing in that lemma. Let:
\beqn
\kappa : (\rD^d, \partial \rD^d) \to (X_{\barbasepta}, \{ \xi_l (x) \geq \xi_0 \}),
\eeqn
be an embedded disk representing the class $\basis_1 = PL [e_1, \gamma_1, o_1]$, as in Section~\ref{subsec-PLclasses}. Then the composition:
\beqn
g \circ \kappa : (\rD^d, \partial \rD^d) \to (X_{\barbasepta}, \{ \xi_l (x) \geq \xi_0 \}),
\eeqn
with a suitable orientation, represents the class $\lambda (u) \, \basis_1$. Since $g (l) = l$, we have $h_l \circ g \circ \kappa = h_l \circ \kappa$. It follows that the embedding $g \circ \kappa$ satisfies conditions (i)-(iii) of Section~\ref{subsec-PLclasses}, which suffice to determine the relative homology class represented by $g \circ \kappa$ up to sign. Namely, we have:
\beqn
\lambda_l (u) \, \basis_1 = \pm \basis_1 \, .
\eeqn
Note that we do not assert that the embedding $g \circ \kappa$ satisfies condition (iv) of Section~\ref{subsec-PLclasses}.

Recall that the action $\lambda_l$ commutes with the monodromy action:
\beqn
\mu_l : B_{W_\La} \to \on{Aut} (M_l (\Pone)), 
\eeqn
and that $\basis_1 \in M_l (\Pone)$ is a cyclic element for the latter (this follows from Proposition~\ref{prop-partial-description} (ii)). It follows that we have:
\beqn
\lambda_l (u) = \pm \on{Id}_{M_l (\Pone)}.
\eeqn
Finally, by Proposition \ref{prop-fundamental-class} (iii), the sign in the above equation is equal to $\tau (u)$.
\end{proof}

Theorem \ref{thm-chi-eq-one} now follows from Propositions \ref{prop-fourier} and \ref{prop-chi-eq-one-basis}, and Lemma \ref{lemma-trivial-inertia}. The isomorphism of Theorem \ref{thm-chi-eq-one} (i) comes from a $\widetilde B_{W_\La}$-module isomorphism $\cM = \cH_1 \to M_l (\Pone)$ which sends $1 \in \cH_1$ into $\basis_1 \in M_l (\Pone)$.

\section{Proof of the main theorem}
\label{sec-proof}

Our proof of Theorem \ref{thm-main} will in many ways be parallel to the proof of Theorem \ref{thm-chi-eq-one} in Section~\ref{sec-proof-of-thm-chi-eq-one}. The main distinction is that the perverse sheaf $\Pone = P_1$ of Section~\ref{sec-proof-of-thm-chi-eq-one} is equipped with a monodromy action $\mu: B_{W_\La} \to \on{Aut} (\Pone)$ of the full braid group $B_{W_\La}$ (see \eqref{eqn-mu-chi-eq-one}), whereas the sheaf $P_\chi$ for $\chi \neq 1$ is only equipped with a monodromy action $\mu: B_{W_\La}^\chi \to \on{Aut} (P_\chi)$ (see \eqref{eqn-mu}). We will get around this limitation by utilizing the extended monodromy in the family, introduced in Section~\ref{subsec-twisted-monodromy} (see \eqref{ext-mon}).

\subsection{A canonical basis for \texorpdfstring{$M_l (P_\chi)$}{Lg}}
As in Section~\ref{subsec-fourier}, for every $y \in \Lp^{rs}$, we consider the Morse group $M_y (P_\chi) = M_{(0, h_y)} (P_\chi)$, and we write $M (P_\chi)$ for the resulting Morse local system on $\Lp^{rs}$.

\begin{prop}\label{prop-fourier-chi}
We have:
\beqn
\fF P_\chi \cong \on{IC} (\Lp^{rs}, M (P_\chi)).
\eeqn
\end{prop}

\begin{proof}
The proof of Proposition~\ref{prop-fourier} goes through in the the presence of the local system $\cL_\chi$; see \cite[Remark 1.4]{G1}.
\end{proof}

\begin{remark}\label{rmk-Z-to-k-chi}
As in Remark \ref{rmk-Z-to-k}, we can conclude from Proposition \ref{prop-fourier-chi} that the claim of Theorem \ref{thm-main} for $\baseRing = \bZ$ implies the claim for general $\baseRing$.
\end{remark}

Recall the basepoint $l \in \La_\bR^+$, and consider the Morse group $M_l (P_\chi)$. As in Section~\ref{sec-proof-of-thm-chi-eq-one}, we will write:
\beqn
\lambda_l : \widetilde B_{W_\La} \to \on{Aut} (M_l (P_\chi)),
\eeqn
for the microlocal monodromy action. Our proof of Theorem \ref{thm-main} will be based on the following analog of Proposition \ref{prop-chi-eq-one-basis}.

\begin{prop}\label{prop-general-basis}
The $\baseRing$-module $M_l (P_\chi)$ is free, and there exists a basis $\{ \basis_w \}_{w \in W_\La}$ of $M_l (P_\chi)$ such that, for every $w \in W_\La$ and every $s \in S$, we have:
\beqn
\begin{array}{ll}
\lambda_l \circ \tilde r \, (\sigma_s) \, \basis_w = - \basis_{sw} + 2 \, \basis_w \, , &
\text{if } \, \BruhatLength (sw) < \BruhatLength (w), \; w^{-1} s w \in W_{\La, \chi}^0 \, , 
\\& \text{and } \delta (s) \text{ is odd}; \\
\lambda_l \circ \tilde r \, (\sigma_s) \, \basis_w = \basis_{sw} \, , &
\text{otherwise}.
\end{array}
\eeqn
\end{prop}

Sections~\ref{subsec-PLclasses-general} - \ref{subsec-general-w} below will be devoted to the proof of Proposition \ref{prop-general-basis}.

\subsection{Picard-Lefschetz classes}
\label{subsec-PLclasses-general}
As in Section~\ref{subsec-PLclasses}, we have the following lemma.

\begin{lemma}
The Morse group $M_l (P_\chi)$ can be identified as follows:
\beqn
M_l (P_\chi) \cong H_d (X_{\barbasepta}, \{ x \in X_{\barbasepta} \; | \; \xi_l (x) \geq \xi_0 \}; \cL_\chi),
\eeqn
where $\xi_0$ is any real number with $\xi_0 > \xi (e_1)$.
\end{lemma}

\begin{proof}
The proof is the same as those of Lemmas~\ref{lemma-morse-dual} and \ref{lemma-large-xi}.
\end{proof}

As in Section~\ref{subsec-picking-signs}, the basis elements $\{ \basis_w \}_{w \in W_\La}$ of Proposition \ref{prop-general-basis} will be constructed as Picard-Lefschetz classes. In order to specify a Picard-Lefschetz class with coefficients in $\cL_\chi$, one needs to specify a triple $[e, \gamma, o]$ as in Section~\ref{subsec-PLclasses}, and in addition, to pick a generator $c \in (\cL_\chi)_e$ of the stalk $(\cL_\chi)_e \cong \baseRing$. The data $[e, \gamma, o, c]$ determines a Picard-Lefschetz class:
\beqn
PL [e, \gamma, o, c] \in H_d (X_{\barbasepta}, \{ x \in X_{\barbasepta} \; | \; \xi_l (x) \geq \xi_0 \}; \cL_\chi),
\eeqn
as follows. Let:
\beqn
\kappa : (\on{D}^d, \partial \on{D}^d) \to (X_{\bar a}, \{ \xi_l (x) \geq \xi_0 \}),
\eeqn
be a smoothly embedded disk, representing the Picard-Lefschetz class \linebreak $PL [e, \gamma, o]$, as in Section~\ref{subsec-PLclasses}. The local system $\cL_\chi$ is trivial on the image of $\kappa$. Therefore, the generator $c$ determines a section of $\cL_\chi |_{\on{Im} (\kappa)}$, and hence a Picard-Lefschetz class $PL [e, \gamma, o, c]$ as above.

Note that the $K$-equivariant structure on $\cL_\chi$ gives rise to a $\widetilde W_\La$-equivariant structure on the restriction of $\cL_\chi$ to the set $Z_l = \{ e_w \}_{w \in W_\La}$. Thus, for every $\widetilde w \in \widetilde W_\La$ and $w' \in W_\La$, we obtain an action map $(\cL_\chi)_{e_{w'}} \to (\cL_\chi)_{e_{q (\widetilde w) \, w'}} \,$, which we denote by $c \mapsto \widetilde w \cdot c$.

We are now prepared to define the basis elements $\{ \basis_w \}_{w \in W_\La}$. Pick a generator $c_1 \in (\cL_\chi)_{e_1}$. For each $w \in W_\La - \{ 1 \}$, define a generator $c_w \in (\cL_\chi)_{e_w}$ by setting:
\beqn
c_w = r (b_w) \cdot c_1 \, ,
\eeqn
where $r = \tilde p \circ \tilde r$, as in Definition~\ref{defn-regular-splitting}. Finally, for every $w \in W_\La$, define:
\beqn
\basis_w = PL [e_w, \gamma_w, o_w, c_w] \in M_l (P_\chi),
\eeqn
where $[e_w, \gamma_w, o_w]$ are as in Sections~\ref{subsec-basis-up-to-sign}, \ref{subsec-picking-signs}. The claim that the $\baseRing$-module $M_l (P_\chi)$ is free and that the elements $\{ \basis_w \}_{w \in W_\La}$ form a basis of $M_l (P_\chi)$ is analogous to Lemma~\ref{lemma-basis}.

\begin{lemma}\label{lemma-interpret-basis}
For every $w \in W_\La$, we have:
\beqn
\basis_w = \lambda_l \circ \tilde r (b_w) \, \basis_1 \, .
\eeqn
\end{lemma}

\begin{proof}
By induction on $\BruhatLength (w)$, it suffices to show that for every $w \in W_\La$ and every $s \in S$ such that $\BruhatLength (w) < \BruhatLength (sw)$, we have:
\beqn
\lambda_l \circ \tilde r \, (\sigma_s) \, \basis_w =  \basis_{sw} \, .
\eeqn
In the case of the trivial local system ($\chi = 1$), we have established this fact up to sign in Proposition~\ref{prop-partial-description} (i), and then fixed the signs by our choice of the orientations $\{ o_w \}_{w \in W_\La}$ in Section~\ref{subsec-picking-signs}. We now proceed as in the proof of Proposition~\ref{prop-partial-description} (i), making use of the path $\Gamma_s$. The only additional observation we need is that the effect of $\Gamma_s$ on the interval $[1/2, 1]$ on the generator $c_w \in (\cL_\chi)_{e_w}$ amounts to the action of $r (\sigma_s) \in \widetilde W_\La$. This establishes the induction step.
\end{proof}

Lemma \ref{lemma-interpret-basis} immediately implies the claim of Proposition \ref{prop-general-basis} for pairs $(w, s)$ with $\BruhatLength (w) < \BruhatLength (sw)$. For the case of pairs $(w, s)$ with $\BruhatLength (w) > \BruhatLength (sw)$, it will be convenient to reverse the roles of $w$ and $sw$, and to restate the claim of the proposition as follows.

\begin{prop}\label{prop-general-basis-restated}
For every $w \in W_\La$ and every $s \in S$ with $\BruhatLength (w) < \BruhatLength (sw)$, we have:
\beqn
\begin{array}{ll}
\lambda_l \circ \tilde r \, (\sigma_s^2) \, \basis_w = - \basis_w + 2 \, \basis_{sw} \, , &
\text{ if } \; w^{-1} s w \in W_{\La, \chi}^0 \; \text{ and } \; \delta (s) \text{ is odd}; \\
\lambda_l \circ \tilde r \, (\sigma_s^2) \, \basis_w = \basis_w \, , &
\text{ otherwise}.
\end{array}
\eeqn
\end{prop}

For the proof of Proposition \ref{prop-general-basis-restated}, we will need the geometric input of the next subsection.

\subsection{Equations for the monodromy in the family}
We have the following analog of Proposition \ref{prop-minimal-polynomials}. Recall the pure braid group $P\!B_{W_\La} \subset B_{W_\La}$.

\begin{prop}\label{prop-minimal-polynomials-general}
(i) Let $s \in S$ be a simple reflection with $s \in W^0_{\La, \chi}$. Then we have $\sigma_s \in B_{W_\La}^\chi$ and the monodromy transformation $\mu (\sigma_s) \in \on{End} (P_\chi)$ satisfies:
\begin{eqnarray*}
(\mu (\sigma_s) - 1)^2 &=& 0, \;\;\;\; \text{if $\delta (s)$ is odd};
\\
\mu (\sigma_s)^2 - 1 &=& 0, \;\;\;\; \text{if $\delta (s)$ is even}.
\end{eqnarray*}
\indent
(ii) Let $s \in S$ be a simple reflection with $s \notin W^0_{\La, \chi}$. Then we have ${\sigma_s}^2 \in P\!B_{W_\La} \subset B_{W_\La}^\chi$ and the monodromy transformation $\mu (\sigma_s^2) \in \on{End} (P_\chi)$ satisfies:
\beqn
\mu (\sigma_s^2) - 1 = 0.
\eeqn
\end{prop}

\begin{proof}
For part (i), the proof proceeds by analogy with the proof of Proposition~\ref{prop-minimal-polynomials}, but working with the family $f_\chi : \tilde \Lp_\chi \to \La / W_{\La, \chi}$ instead of $f$. Just as in that proof, the statement can be reduced to the rank one case. We continue to use the notation of that proof. Consider the map:
\beqn
\nu: (\baseptas + \bar \Lp_s) \to (\baseptas + \bar \Lp_s) \inv K_s \cong (\baseptas + \La_s^\p) / W_s \to \La / W_{\La, \chi} \, ,
\eeqn
where the last arrow is induced by the inclusion $\baseptas + \La_s^\p \to \La$ (note that we have $s \in W^0_{\La, \chi} \subset W_{\La, \chi}$). Using the map $\nu$, we obtain an embedding:
\beqn
\widetilde j_{\baseptas} : \bar \Lp_s \to \tilde \Lp_\chi, \;\; x \mapsto (\nu (\baseptas + x), \baseptas + x).
\eeqn
Recall the quotient map $\bar f_s : \bar \Lp_s \to \La_s^\p / W_s \cong \bar \Lp_s \inv K_s$. Consider the regular semisimple locus $\bar \Lp_s^{rs} = \bar f_s^{-1} (\La_s^\p / W_s - \{ 0 \}) \subset \bar \Lp_s$. It is not the case that $\widetilde j_{\baseptas} (\bar \Lp_s^{rs})$ is a subset of $\tilde \Lp_\chi^{rs}$. However, if we pick a small neighborhood of zero $U_s \subset \La_s^\p / W_s$, then we have $\widetilde j_{\baseptas} (\bar \Lp_s^{rs} \cap \bar f_s^{-1} (U_s)) \subset \tilde \Lp_\chi^{rs}$. Moreover, by Corollary~\ref{cor-sInW0} and the regularity assumption on the splitting homomorphism $\tilde r$ (see Defintion~\ref{defn-regular-splitting}), the pullback $\widetilde j_{\baseptas}^* \hat \cL_\chi$ restricts to a trivial local system on $\bar \Lp_s^{rs} \cap \bar f_s^{-1} (U_s)$. Thus, we are reduced to the same problem for the rank one pair $(\bar G_s, \bar K_s)$ as in the proof of Proposition~\ref{prop-minimal-polynomials}. Part (i) follows.

For part (ii), note that $s \notin W^0_{\La, \chi}$ implies that $\delta (s) = 1$ and $\check\alpha_s (-1) \notin Z_K (\La)^0$, i.e., we are in the situation Lemma~\ref{lemma-trichotomy} (iv)(b). Since we have $\sigma_s^2 \in P\!B_{W_\La}$, we can base change the family $f$ all the way to $\La$, and consider the family:
\beqn
\tilde f: \tilde \Lp = \Lp \times_{\La / W_\La} \La \to \La.
\eeqn
Using this family, we proceed again similarly to the proof of Proposition~\ref{prop-minimal-polynomials}, but using the pullback via the arc $\gamma : \rD_\gamma \to \La$ instead of $\bar\gamma : \rD_{\bar\gamma} \to \La / W_\La$, to reduce the statement to the case of the rank one pair $(\bar G_s, \bar K_s)$. Moreover, we can further reduce to the case of the pair $(\bar G'_s, \bar K'_s)$ which is isomorphic to $(SL(2), SO(2))$, as in the proof of Lemma~\ref{lemma-trichotomy} (iv)(b). By Corollary~\ref{cor-sInW0}, the character that we obtain in this way for the pair $(SL(2), SO(2))$ is the non-trivial character $\chi_1$ of Section~\ref{sec-example}. Part (ii) now follows from Corollary \ref{cor-example}.
\end{proof}

\subsection{Proof of Proposition \ref{prop-general-basis-restated} for \texorpdfstring{$w = 1$}{Lg}}
\label{subsec-w-eq-one}
Assume that the basepoints $\basepta = l \in \La_\bR^+$ have been chosen as in Section~\ref{subsec-near-wall} (i.e., near the hyperplane $\La_s \subset \La$), and recall the path $\gamma_{1, s} : [0, 1] \to \bC$ defined in that section. As in~\eqref{eqn-pl-formula}, we have:
\beq\label{eqn-ms}
PL [e_1, \gamma_{1, s}, o_1, c_1] = \basis_1 + m_s \, \basis_s \in M_l (P_\chi),
\eeq
for some $m_s \in \bZ$, independent of $\baseRing$. By arguing as in Lemma \ref{lemma-generator-squared} (i) and using Lemma \ref{lemma-interpret-basis}, we obtain:
\beqn
\lambda_l \circ \tilde r \, (\sigma_s^2) \, \basis_1 = \lambda_l \circ \tilde r \, (\sigma_s) \, \basis_s = PL [e_1, \gamma_{1, s}, o_{1, s}, c_{1, s}],
\eeqn
where $o_{1, s} \in \cO [1]$ is an orientation which is independent of $\chi$, and $c_{1, s} \in (\cL_\chi)_{e_1}$ is given by:
\beqn
c_{1, s} = r (\sigma_s^2) \cdot c_1 \, .
\eeqn
By the definition of the equivariant local system $\cL_\chi$, we have:
\beqn
c_{1, s} = \chi (r (\sigma_s^2)) \cdot c_1 \, .
\eeqn
Thus, by Corollary~\ref{cor-char-and-lifts}, we have:
\beqn
c_{1, s} = c_1 \, , \, \text{ if} \; s \in W^0_{\La, \chi} \, , \;\;\;\; \text{and} \;\;\;\;
c_{1, s} = -c_1 \, , \, \text{ otherwise}.
\eeqn
Morover, by Proposition \ref{prop-chi-eq-one-basis} (i), we have:
\beqn
o_{1, s} = (-1)^{\delta (s)} \cdot o_1 \, .
\eeqn

Consider now the case $s \in W^0_{\La, \chi}$. In this case, we have $\sigma_s \in B_{W_\La}^\chi$, and we can consider the monodromy transformation $\mu (\sigma_s) : P_\chi \to P_\chi$. Arguing as in the proof of Proposition~\ref{prop-partial-description} (ii), and using Proposition \ref{prop-chi-eq-one-basis} (ii) and the definition of the local system $\hat \cL_\chi$ on $\tilde \Lp^{rs}_\chi$ in Section~\ref{subsec-twisted-monodromy}, we obtain:
\beqn
\mu (\sigma_s) \, \basis_1 = PL [e_s, \gamma_s, (-1)^{\delta (s) + 1} \cdot o_s, r (\sigma_s^{-1}) \cdot c_1].
\eeqn
Moreover, by Corollary~\ref{cor-char-and-lifts}, we have:
\beqn
r (\sigma_s^{-1}) \cdot c_1 = r (\sigma_s) \cdot (r (\sigma_s^{-2}) \cdot c_1) = \chi (r (\sigma_s^{-2})) \cdot r (\sigma_s) \cdot c_1 = c_s \, .
\eeqn
Thus, we have:
\beqn
\mu (\sigma_s) \, \basis_1 = (-1)^{\delta (s) + 1} \cdot \basis_s \, .
\eeqn
Similarly, as in Lemma \ref{lemma-generator-squared} (ii) and using Proposition \ref{prop-chi-eq-one-basis} (ii), we have:
\beqn
\mu (\sigma_s) \, \basis_s = PL [e_1, \gamma_{1, s}, - o_1, r (\sigma_s^{-1}) \cdot c_s] = - PL [e_1, \gamma_{1, s}, o_1, c_1] = - \basis_1 - m_s \, \basis_s \, .
\eeqn
Thus, we see that the monodromy transformation $\mu (\sigma_s)$ preserves the rank $2$ submodule $M_{1, s} \subset M_l (P_\chi)$ spanned by the basis elements $\{ \basis_1, \basis_s \}$. Let $\mu_{1, s} (\sigma_s)$ be the restriction of $\mu (\sigma_s)$ to $M_{1, s}$. Using the basis $\{ \basis_1, \basis_s \}$, we can summarize the above as follows:
\beqn
\mu_{1, s} (\sigma_s) =
\begin{pmatrix}
0                              &  -1   \\
(-1)^{\delta (s) + 1}  &  -m_s
\end{pmatrix}.
\eeqn
Using Proposition \ref{prop-minimal-polynomials-general} (i), we can conclude that the integer $m_s$ of equation~\eqref{eqn-ms} is given by:
\beqn
m_s = (-1)^{\delta (s)} - 1.
\eeqn
The claim of Proposition~\ref{prop-general-basis-restated} for $w = 1$ and $s \in W^0_{\La, \chi}$ follows.

Consider now the case $s \notin W^0_{\La, \chi}$. In this case, we necessarily have $\delta (s) = 1$. We have $\sigma_s^2 \in B_{W_\La}^\chi$, and we can consider the monodromy transformation $\mu (\sigma_s^2) : P_\chi \to P_\chi$. As before, we compute:
\beqn
\mu (\sigma_s^2) \, \basis_1 = PL [e_1, \gamma_{1, s}, - o_1, r (\sigma_s^{-2}) \cdot c_1].
\eeqn
However, in this case, we have:
\beqn
r (\sigma_s^{-2}) \cdot c_1 = \chi (r (\sigma_s^{-2})) \cdot c_1 = - c_1 \, .
\eeqn
Thus, we obtain:
\beqn
\mu (\sigma_s^2) \, \basis_1 = PL [e_1, \gamma_{1, s}, o_1, c_1] = \basis_1 + m_s \, \basis_s \, .
\eeqn
Using Proposition~\ref{prop-minimal-polynomials-general} (ii), we can conclude that $m_s = 0$. The claim of Proposition~\ref{prop-general-basis-restated} for $w = 1$ and $s \notin W^0_{\La, \chi}$ follows. This completes the proof Proposition \ref{prop-general-basis-restated} for $w = 1$.

\subsection{Proof of Proposition \ref{prop-general-basis-restated} for general \texorpdfstring{$w$}{Lg}}
\label{subsec-general-w}
Recall the extended monodromy structure introduced in equation~\eqref{ext-mon}. We use it to complete the proof of Proposition \ref{prop-general-basis-restated}. We recall that the extended monodromy is given by:
\beqn
\mu: B_{W_\La} \to \on{Aut} \Bigg( \bigoplus_{\bar w \in W_\La / W_{\La, \chi}} P_{\bar w \cdot \chi} \Bigg),
\eeqn
and that for a braid $b \in B_{W_\La}$, we have $\mu (b) : P_{\bar w \cdot \chi} \to P_{(p(b) \bar w) \cdot \chi} \,$.

Recall the basis $\{ \basis_{w'}  \}_{w' \in W_\La}$ of the Morse group $M_l (P_\chi)$ constructed in Section~\ref{subsec-PLclasses-general}. For each $w \in W_\La$, we can apply the same construction to obtain a basis:
\beqn
\{ \basis_{w'} [w] \}_{w' \in W_\La} \;\; \text{of} \;\; M_l (P_{w \cdot \chi}).
\eeqn
Each element $b \in B_{W_\La}$ induces a map $\mu_l (b) : M_l (P_{w \cdot \chi}) \to M_l (P_{(p(b) w) \cdot \chi})$. Note that we have:
\beq
\label{character-conj}
w^{-1} s w \in W^0_{\La, \chi} \iff s \in W^0_{\La, w \cdot \chi} \, .
\eeq

Let $w = s_1 \cdots s_m$, $s_i \in S$, be a reduced expression. Recall that we write $b_w = \sigma_{s_1} \cdots \sigma_{s_m} \in B_{W_\La}$. Let $\bar b_w = \sigma_{s_1}^{-1} \cdots \sigma_{s_m}^{-1} \in B_{W_\La}$. The key step of the proof is the following claim:
\beq\label{eqn-key-claim}
\mu_l (\bar b_w) \, \basis_w = \pm \basis_1 [w].
\eeq
This claim follows by induction on the length $\BruhatLength (w)$, using the following:  
\beq\label{eqn-key-claim-step}
\text{if} \;\; \BruhatLength (w) < \BruhatLength (ws) \;\; \text{then} \;\;
\mu_l (\sigma_s) \basis_w [w'] = \pm \basis_{ws} [sw'] \;\; \text{for all} \;\; w' \in W_\La \, . 
\eeq
The claim~\eqref{eqn-key-claim-step} is analogous to Proposition \ref{prop-partial-description} (ii). The claim~\eqref{eqn-key-claim} follows. 

Since extended monodromy commutes with microlocal monodromy, equation~\eqref{eqn-key-claim} allows us to reduce the claim of Proposition~\ref{prop-general-basis-restated} for general $w \in W_\La$ to a corresponding claim about the action of $\lambda_l \circ \tilde{r} \, (\sigma_s^2)$ on the basis element $\basis_1 [w] \in M_l (P_{w \cdot \chi})$. In view of equation~\eqref{character-conj}, this corresponding claim has been proved in Section~\ref{subsec-w-eq-one}. This completes the proof of  Proposition~\ref{prop-general-basis-restated}, and with it, that of Proposition~\ref{prop-general-basis}.

\subsection{Proof of Theorem \ref{thm-main}}
Let us note first that, by Lemma \ref{lemma-interpret-basis}, we have:
\beq\label{eqn-cyclic-vector}
\begin{gathered}
\text{the Morse group $M_l (P_\chi)$ is a cyclic $\baseRing [B_{W_\La}]$-module, with $\basis_1$ as a} \\
\text{generator, under the microlocal monodromy action $\lambda_l \circ \tilde r$.}
\end{gathered}
\eeq 
Let us write:
\beq\label{decomp-to-cosets}
M_l (P_\chi) = \bigoplus_{u \in W_\La / W_{\La, \chi}^0} V_u \, ,
\eeq
where $V_u = \on{span} \{ \basis_w \; | \; w \in u \}$. Let $u_0 \in W_\La / W_{\La, \chi}^0$ be the identity coset, and let $V_0 = V_{u_0}$.

The $B_{W_\La}$-action $\lambda_l \circ \tilde r$ respects the direct sum decomposition in~\eqref{decomp-to-cosets}, acting on it in the following manner. Pick a set of representatives $\{ w_u \in u \}_{u \in W_\La / W_{\La, \chi}^0}$. Then, for every $s \in S$, we have:
\beqn
\begin{array}{ll}
\lambda_l \circ \tilde r (\sigma_s) \, V_u = V_u \, ,       & \text{ if } w_u^{-1} s w_u \in W_{\La, \chi}^0 \, , \\
\lambda_l \circ \tilde r (\sigma_s) \, V_u = V_{s u} \, ,  & \text{ if } w_u^{-1} s w_u \notin W_{\La, \chi}^0 \, .
\end{array}
\eeqn
To see this, note that $w$ and $s w$ are in the same coset precisely when $w^{-1} s w \in W_{\La, \chi}^0$, then apply the formulas of Proposition~\ref{prop-general-basis}. It follows that, for every $\braid \in B_{W_\La}$, we have:
\beq\label{eqn-action1}
\lambda_l \circ \tilde r \, (\braid) : V_{u_1} \mapsto V_{u_2} \, , \, \text{ where } u_2 = p(\braid) \, u_1 \in W_\La / W_{\La, \chi}^0 \, .
\eeq
In particular, the submodule $V_0 \subset M_l (P_\chi)$ is stable under the action of $B_{W_\La}^{\chi, 0} \subset B_{W_\La}$.

Using the formulas of Proposition~\ref{prop-general-basis} once again, we see that, for every $s \in S$, we have:
\beq\label{relations1}
\begin{array}{lll}
(\lambda_l \circ \tilde r \, (\sigma_s) - 1)^2 = 0  & \text{ on } V_u \, ,  & \text{ if } w_u^{-1} s w_u \in W_{\La, \chi}^0 \text{ and } \delta(s) \text{ is odd}, \\
\lambda_l \circ \tilde r \, (\sigma_s^2) = 1          & \text{ on } V_u \, ,  & \text{ otherwise}.
\end{array}
\eeq

Let us now make use of these considerations to analyze the action of $B_{W_\La}^{\chi, 0}$ on $V_0$. First, we claim that this action factors through the braid group $B_{W_{\La, \chi}^0}$. To see this, recall from Section~\ref{subsec-main-thm} that we have the following commutative diagram:
\beqn
\xymatrix{
1 \ar[r] & \pi_1 (\La^{rs}) \ar[r] \ar@{->>}[d]_{\varphi_1} & B_{W_\La}^{\chi, 0} = \pi_1 (\La^{rs} / W_{\La, \chi}^0) \ar[r] \ar@{->>}_{\varphi}[d] & W_{\La, \chi}^0 \ar[r] \ar@{=}[d] & 1 \;\, \\
1 \ar[r] & \pi_1 (\La_\chi^{rs}) \ar[r] & B_{W_{\La, \chi}^0} = \pi_1 (\La_\chi^{rs} / W_{\La, \chi}^0) \ar[r] & W_{\La, \chi}^0 \ar[r] & 1 \, .
}
\eeqn
By inspection, we have $\ker (\varphi) = \ker (\varphi_1)$. Thus, it suffices to show that $\ker (\varphi_1)$ acts trivially on $V_0$. To do so, for every $s \in S$ and every $\braid \in B_{W_\La}$, consider the element:
\beqn
\sigma_{s, \braid} = \braid \, \sigma_s \, \braid^{-1} \in B_{W_\La} \, .
\eeqn
We have $\sigma_{s, \braid}^2 \in \ker (\varphi_1)$, whenever $p(\braid) \, s \, p(\braid)^{-1} \notin W_{\La, \chi}^0$. Moreover, the elements $\{ \sigma_{s, \braid}^2 \}$, for all such pairs $(s, \braid)$, generate the group $\ker (\varphi_1)$; this follows, for example, from \cite[Proposition A1, p$. \; 181$]{BMR}. Using equations \eqref{eqn-action1} and \eqref{relations1}, we can readily check that:
\beqn
p (\braid) \, s \, p(\braid)^{-1} \notin W_{\La, \chi}^0 \; \implies \; \lambda_l \circ \tilde r \, (\sigma_{s, \braid}^2) = 1 \text{ on } V_0 \, ,
\eeqn
thus verifying the claim.

Let us now analyze the action of $B_{W_{\La, \chi}^0}$ on $V_0$. To that end, recall the set of simple reflections $S_\chi = \{ s_{\bar\alpha_1}, \dots, s_{\bar\alpha_\nSchi} \} \subset W_{\La, \chi}^0$ introduced in Section~\ref{subsec-heck-alg}, and the associated set of counter-clockwise braid generators $\{ \sigma_1, \dots, \sigma_\nSchi \} \subset B_{W_{\La, \chi}^0}$. Each of these braid generators can we written as:
\beqn
\sigma_i = \varphi (\sigma_{s_i, b_i}),
\eeqn
for some $s_i \in S$ and $b_i \in B_{W_\La}$. (Note that the $s_i$ for different $i$ do not have to be distinct.) Moreover, we have $\delta (s_{\bar\alpha_i}) = \delta (s_i)$. Using equations \eqref{eqn-action1} and \eqref{relations1}, we can readily infer that:
\beq\label{relations2}
\begin{array}{lll}
(\lambda_l \circ \tilde r \, (\sigma_{s_i, b_i}) - 1)^2 = 0  & \text{ on } V_0 \, ,  & \text{ if } \delta(s_{\bar\alpha_i}) \text{ is odd}, \\
\lambda_l \circ \tilde r \, (\sigma_{s_i, b_i}^2) = 1          & \text{ on } V_0 \, ,  & \text{ if } \delta(s_{\bar\alpha_i}) \text{ is even}.
\end{array}
\eeq
Consider the map of $\baseRing$-modules:
\beq\label{V0eqHeck}
\baseRing [B_{W_{\La, \chi}^0}] \to V_0 \, , \, \text{ defined by } \; \varphi (\braid) \mapsto \lambda_l \circ \tilde r \, (\braid) \, \basis_1 \, , \;\; \braid \in B_{W_\La}^{\chi, 0} \, .
\eeq
By applying Lemma \ref{lemma-interpret-basis} to all $w \in W_{\La, \chi}^0$, we can see that this map is surjective. Equations \eqref{relations2} imply that this map factors through the canonical quotient map $\eta_\chi : \baseRing [B_{W_{\La, \chi}^0}] \to \cH_{W_{\La, \chi}^0}$. Since both $V_0$ and $\cH_{W_{\La, \chi}^0}$ are free $\baseRing$-modules of rank $|W_{\La, \chi}^0|$, and since $\baseRing$ is an integral domain, we have:
\beq\label{eqn-iso1}
\text{the map \eqref{V0eqHeck} yields an isomorphism of $\baseRing [B_{W_\La}^{\chi, 0}]$-modules $V_0 \cong \cH_{W_{\La, \chi}^0} \,$.}
\eeq

By an argument as in the proof of Lemma \ref{lemma-trivial-inertia}, and using that lemma for the case $\chi = 1$, one can show that, for every $u \in I$, we have:
\beq\label{IonBasis1}
\lambda_l (u) \, \basis_1 = \chi(u) \cdot \tau (u) \cdot \basis_1 \, .
\eeq
We can use Lemma \ref{lemma-interpret-basis}, plus the fact that the group $W_{\La, \chi}^0$ preserves the character $\chi$, while all of $W_\La$ preserves the character $\tau$, to generalize equation~\eqref{IonBasis1} as follows. For every $u \in I$ and every $w \in W_{\La, \chi}^0$, we have:
\beq\label{IonV0}
\lambda_l (u) \, \basis_w = \chi(u) \cdot \tau (u) \cdot \basis_w \, .
\eeq
Combining~\eqref{eqn-iso1} and~\eqref{IonV0}, we obtain an isomorphism of $\baseRing [\widetilde B_{W_\La}^{\chi, 0}]$-modules:
\beq\label{submodV0}
V_0 \cong \baseRing_\chi \otimes \cH_{W_{\La, \chi}^0} \otimes \baseRing_\tau \, .
\eeq

Consider the following map of $\baseRing [\widetilde B_{W_\La}]$-modules:
\beqn
\begin{gathered}
\Psi : \; \baseRing [\widetilde B_{W_\La}] \otimes_{\baseRing [\widetilde B_{W_\La}^{\chi, 0}]} (\baseRing_\chi \otimes \cH_{W_{\La, \chi}^0} \otimes \baseRing_\tau) \to M_l (P_\chi),
\\
\Psi : \; \tilde\braid \otimes (1_\chi \otimes 1_\cH \otimes 1_\tau) \mapsto \lambda_l (\tilde\braid) \, \basis_1 \, ,
\end{gathered}
\eeqn
where $1_\cH \in \cH_{W_{\La,\chi}^0}$ denotes the identity. This map is well-defined in view of equation~\eqref{submodV0}. By~\eqref{eqn-cyclic-vector}, the map $\Psi$ is surjective. As before, we observe that both the domain and the range of $\Psi$ are free $\baseRing$-modules of rank $|W_\La|$. It follows that $\Psi$ is an isomorphism. Finally, we note that:
\beqn
\baseRing [\widetilde B_{W_\La}] \otimes_{\baseRing [\widetilde B_{W_\La}^{\chi, 0}]} (\baseRing_\chi \otimes \cH_{W_{\La, \chi}^0} \otimes \baseRing_\tau) \cong \left(\baseRing [\widetilde B_{W_\La}] \otimes_{\baseRing [\widetilde B_{W_\La}^{\chi, 0}]} (\baseRing_\chi \otimes \cH_{W_{\La, \chi}^0})\right) \otimes \baseRing_\tau \, ,
\eeqn
since $\tau$ is preserved by $W_\La$. By the definition of $\cM_\chi$ in \eqref{eqn-M-chi}, this completes the proof of Theorem~\ref{thm-main}.


\begin{thebibliography}{999999}

\bibitem[BBD]{BBD}
A. A. Belinson, J. Bernstein and P. Deligne, Faisceaux pervers, in {\it Analysis and topology on singular spaces, I (Luminy, 1981)}, Ast\'erisque {\bf 100}, Soc. Math. France, Paris, 5--171.

\bibitem[BMR]{BMR}
M. Brou\'e, G. Malle and R. Rouquier, Complex reflection groups, braid groups, Hecke algebras, J. Reine Angew. Math. {\bf 500} (1998), 127--190.

\bibitem[CVX]{CVX}
T.-H. Chen, K. Vilonen and T. Xue, Springer correspondence for the split symmetric pair in type $A$, Compos. Math. {\bf 154} (2018), no.~11, 2403--2425.

\bibitem[DK]{DK}
J. Dadok and V. Kac, Polar representation, Journal of Algebra {\bf 92} (1985), no.~2, 504--524.

\bibitem[Gi]{Gi}
V. Ginzburg, Characteristic varieties and vanishing cycles, Invent. Math. {\bf 84} (1986), no.~2, 327--402.


\bibitem[G1]{G2}
M. Grinberg, A generalization of Springer theory using nearby cycles, Represent. Theory {\bf 2} (1998), 410--431.

\bibitem[G2]{G3}
M. Grinberg, Morse groups in symmetric spaces corresponding to the symmetric group, Sel. Math., New Ser. {\bf 5} (1999), 303--323.

\bibitem[G3]{G1}
M. Grinberg, On the specialization to the asymptotic cone, J. Algebraic Geom. {\bf 10} (2001), no.~1, 1--17.


\bibitem[G4]{G4}
M. Grinberg, Errata and notes on the paper ``A generalization of Springer theory using nearby cycles", arXiv:2002.11568.

\bibitem[H]{H}
J. E. Humphreys, {\it Reflection groups and Coxeter groups}, Cambridge Studies in Advanced Mathematics, 29. Cambridge University Press, Cambridge, UK, 1990.

\bibitem[K]{K}
A. W. Knapp, {\it Lie groups beyond an introduction}, second edition, Progress in Mathematics, 140. Birkh\"auser Boston, Boston, MA, 2002.

\bibitem[KN]{KN}
G. Kempf and L. Ness, The length of vectors in representation spaces, {\it Algebraic geometry}, Lecture Notes in Math., 732, 233--243. Springer-Verlag, Berlin, 1978.

\bibitem[Ko]{Ko}
R. R. Kocherlakota, Integral homology of real flag manifolds and loop spaces of symmetric spaces, Adv. Math. {\bf 110} (1995), no. 1, 1--46. 

\bibitem[KR]{KR}
B. Kostant and S. Rallis, Orbits and representations associated with symmetric spaces, Amer. J. Math. {\bf 93} (1971), 753--809.

\bibitem[KS]{KS}
M. Kashiwara and P. Schapira, {\it Sheaves on manifolds}, second reprint of the 1990 original, Grundlehren der Mathematischen Wissenschaften, 292. Springer-Verlag, Berlin, 2002.

\bibitem[Ma]{Ma}
J. Mather, Notes on topological stability, Bull. Amer. Math. Soc. {\bf 49} (2012), no.~4, 475--506.

\bibitem [Mo]{Mo}
G. D. Mostow, Some new decomposition theorems for semi-simple groups, AMS Memoirs
{\bf 14} (1955), 31--54.

\bibitem[R]{R}
R. W. Richardson, Orbits, invariants, and representations associated to involutions of reductive groups, Invent. Math. {\bf 66} (1982), no.~2, 287--312.

\bibitem[SV]{SV}
W. Schmid and K. Vilonen, On the geometry of nilpotent orbits, 
Surv. Differ. Geom., 7, 565--623. Int. Press, Somerville, MA, 2000.

\bibitem[S]{S}
T. A. Springer, {\it Linear algebraic groups}, reprint of the 1998 second edition, Modern Birkh\"auser Classics. Birkh\"auser Boston, Boston, MA, 2009.

\bibitem[VX]{VX}
K. Vilonen and T. Xue, Character sheaves for classical symmetric pairs. With an appendix by Dennis Stanton. Represent. Theory {\bf 26} (2022), 1097--1144.

\bibitem[V]{V}
D. A. Vogan, Jr., Irreducible characters of semisimple Lie groups. IV. Character-multiplicity duality, Duke Math. J. {\bf 49} (1982), no. 4, 943--1073.

\end{thebibliography}
\end{document}